\documentclass[11pt,reqno]{amsart}
\usepackage[T1]{fontenc}
\usepackage[utf8]{inputenc}
\usepackage{hyperref}

\allowdisplaybreaks[4]

\usepackage{amsmath}
\usepackage{amssymb}
\usepackage{mathrsfs}
\usepackage{amsthm}
\usepackage{bm}
\usepackage{graphicx}
\usepackage{color}
\usepackage{booktabs}
\usepackage{float}

\usepackage[caption=false]{subfig}

\usepackage{geometry}


\theoremstyle{plain}
\newtheorem{lemma}{Lemma}[section]
\newtheorem{theorem}[lemma]{Theorem}
\newtheorem{definition}[lemma]{Definition}
\newtheorem{corollary}[lemma]{Corollary}
\newtheorem{remark}[lemma]{Remark}
\newtheorem{proposition}[lemma]{Proposition}
\newtheorem{assumption}{Assumption}
\newtheorem{example}{Example}

\begin{document}
\title[HJB equation for optimal control of SWHS on graph]{Hamilton--Jacobi--Bellman equation for optimal control of stochastic Wasserstein--Hamiltonian system on graphs}
\author{Jianbo Cui, Tonghe Dang}
\address{Department of Applied Mathematics, The Hong Kong Polytechnic
University, Hung Hom, Kowloon, Hong Kong, China.}
\email{jianbo.cui@polyu.edu.hk; tonghe.dang@polyu.edu.hk(Corresponding author)}
\thanks{This work is supported by MOST National Key R\&D Program No. 2024FA1015900, the Hong Kong Research Grant Council GRF grant 15302823, GRF grant 15301025, NSFC/RGC Joint Research Scheme N$\_$PolyU5141/24, NSFC grant 12522119, NSFC grant 12301526, internal funds (P0041274, P0045336)  from Hong Kong Polytechnic University,  and the
CAS AMSS-PolyU Joint Laboratory of Applied Mathematics.}
\begin{abstract}
Stochastic optimal control problems for Hamiltonian dynamics on graphs have wide-ranging applications in mechanics and quantum field theory, particularly in systems with graph-based structures. 
In this paper, we establish  the existence and uniqueness of viscosity solutions for a new class of Hamilton--Jacobi--Bellman (HJB) equations arising from the optimal control of stochastic Wasserstein--Hamiltonian systems (SWHSs) on graphs. One distinctive feature of these HJB equations is the simultaneous involvement of the Wasserstein geometry on the Wasserstein space over graphs and the Euclidean geometry in physical space. The nonlinear geometric structure, along with the logarithmic potential induced by the graph-based state equation, adds further complexity to the analysis.
To address these challenges, we introduce an energy-truncation technique within the doubling of variables framework, specifically designed to handle the interaction between the interiorly defined Wasserstein space on graphs and the unbounded Euclidean space. 
In particular, our findings demonstrate the well-posedness of HJB equations related to optimal control problems for both stochastic Schr\"odinger equation with polynomial nonlinearity and stochastic logarithmic Schr\"odinger equation on graphs. To the best of our knowledge, this work is the first to develop HJB equations for the optimal control of SWHSs on graphs.

\end{abstract}
\maketitle
\section{Introduction}
The stochastic control theory for Hamiltonian dynamics lies at the intersection of applied mathematics, physics, and engineering, with important applications in mechanics and quantum field theory 
\cite{AAP1,Sergio}. The introduction of stochasticity further captures essential features such as environmental decoherence and control uncertainties that are intrinsic to real-world technologies. 
The stochastic control problems of Hamiltonian systems in Euclidean spaces have been deeply explored in the monograph \cite{YongJM2}, and related studies for Schr\"odinger equations on continuous Euclidean domains have also been extensively developed in both deterministic case (see e.g. \cite{Ito19,Aronna19,FZ19,Sergio}) and stochastic case (see e.g. \cite{Lvqi,ZhangD18,LQ2013,LQbook}). In parallel, the growing role of graph-structured models in networks, optical waveguide arrays, and engineered quantum devices has sparked increasing interest in the study of optimal transport and quantum phenomena on discrete graphs \cite{model-book}. Over the past decade, research on partial differential equations on graphs, including discrete optimal transport  \cite{ARMA2012,Mielke}, gradient flows \cite{Maas}, Hamiltonian systems \cite{Cui_MC22,cuiSIAM}, and Hamilton--Jacobi equations  \cite{MCC}, has been actively investigated. Some works have also addressed controllability problems on quantum graphs for Schr\"odinger-type equations \cite{star-graph1,star3,star4}, heat equation \cite{heat1}, {K}orteweg--de {V}ries
              equation \cite{star2}, etc. Comparatively, 
the stochastic control theory for stochastic differential equations with graph structures 
remains far less developed. 
A recent advance in this direction was made in \cite{Cui23}, where the authors establish the existence of optimal controls for the stochastic nonlinear Schr\"odinger equation on graphs and provide a description of the optimal condition via the forward and
backward stochastic differential equations. 
Nevertheless, to the best of our knowledge, no results are currently available concerning the Hamilton--Jacobi--Bellman (HJB) equations associated with stochastic control problems of stochastic Hamiltonian systems on graphs.

In this paper, we introduce a finite graph $G=(V,E,\omega)$, which is a undirected,  connected, and  weighted graph with vertex set $V=\{1,\ldots,n\}$, edge set $E \subset V \times V$, and weights $\omega_{jl} = \omega_{lj} > 0$ if $(j,l) \in E$, and $\omega_{jl} = 0$ otherwise. 
We consider the following stochastic Wasserstein--Hamiltonian system (SWHS) on the cotangent bundle of
the density manifold $\mathcal P(G)$ 
 \begin{align}\label{SWHS-1}
 \begin{cases}
\mathrm d\rho(t) = \frac{\partial }{\partial S}\mathcal H^{\mathbb V}_0(\rho(t) ,S(t))\,\mathrm dt+\frac{\partial }{\partial S} \mathcal H_1(\rho(t) ,S(t))\circ\mathrm dW(t),\\
\mathrm dS(t)=-\frac{\partial }{\partial \rho} \mathcal H^{\mathbb V}_0(\rho(t) ,S(t))\,\mathrm dt-\frac{\partial }{\partial \rho} \mathcal H_1(\rho(t) ,S(t))\circ \mathrm dW(t),
\end{cases}
\end{align} 
with $t\in(t_0,T]$  and initial value $\rho(t_0)\in\mathcal P(G),S(t_0)\in\mathbb R^n$, where $0\leq t_0<T,$ $\mathcal P(G)$ is the Wasserstein density space on graph, $\{ W(t)=(W_1(t),\ldots,W_n(t))\}_{t\ge t_0}$ is a standard $n$-dimensional Brownian motion on a complete filtered probability space $(\Omega,\mathcal F ,\{ \mathcal F_t\}_{t\ge t_0},\mathbb P )$.  Here $``\circ"$ stands for the Stratonovich type stochastic integration, and $\frac{\partial}{\partial S},\frac{\partial}{\partial\rho}$  denote the Euclidean partial derivatives with variables $S,\rho$, respectively.  The solution $(\rho(t),S(t))\in \mathcal P(G)\times\mathbb R^n,\,t\ge t_0$ is referred to as the state process, 
$\mathcal H^{\mathbb V}_0$ is the dominant energy depending on the control process $\mathbb V,$ and $\mathcal H_1$  is a $\mathbb R^n$-valued function; see Section \ref{sec_2} for details. 

The SWHS \eqref{SWHS-1} extends the stochastic Hamiltonian systems on Euclidean spaces \cite{Milstein,Hong} to spaces endowed with graph structures,
which may stand for some quantum models on graphs \cite{cuiSIAM}. One typical example is the stochastic Schr\"odinger equation (SSE) on graph \cite{ZhouJFA,Cui23}: 
\begin{align}\label{schro-intro}
\mathbf{i}\mathrm du_j=-\frac12(\Delta_Gu)_j+(u_j\mathbb V_j+u_jf_j(|u|^2))\mathrm dt+\sigma_ju_j\circ \mathrm dW_j(t),\quad j\in V,
\end{align}
where $\Delta_G$ is a nonlinear discretization of the Laplacian operator on $G$ (see \eqref{nonlinear_lap}),   
 $f_j:\mathbb R_+\to\mathbb R$ is a continuous function, $\mathbb V_j:\mathbb R_+\to\mathbb R$ denotes the control function,  $\sigma_j \in \mathbb R$ represents the diffusion coefficient.  We refer to Section \ref{sec_6} for more details.  
It is worth noting that the SSE on graphs can be viewed as a spatial discretization of the classical SSE in Euclidean space,  when $G$ is chosen as a lattice obtained from discretizing a continuous domain \cite{ZhouJFA}. Moreover, the nonlinear discretization of the Laplacian in \eqref{schro-intro} can preserve several physical 
properties, particularly the full continuum nonlinear dispersion relation, which is unattainable through
linear discretizations \cite{Cui23}. 

The stochastic optimal control problem of SWHS \eqref{SWHS-1} consists of minimizing a cost of the form 
\begin{align*}
\mathcal J(t, \rho, x; \mathbb V) := \mathbb E_t \Big[ \int_t^T F(s, \rho(s), S(s), \mathbb V(s)) \, \mathrm ds + h(\rho(T), S(T)) \Big],
\end{align*} 
over all admissible controls $\mathbb V\in\mathscr V_{\ell}[t,T]$ with the control set $\mathscr V_{\ell}[t,T]:=\{\mathbb V:[t,T]\to B_{\ell}:\mathbb V(s)\text{ is }\mathcal F_s\text{-adapted},s\in[t,T] \}$.
Here,   
$\rho(t)=\rho,S(t)=x,$ $B_{\ell}$ denotes the ball with radius $\ell>0,$ referred to as 
the control region, and functionals $F,h$ are respectively, the running cost and terminal cost.  
We aim to characterize  the value function $U(t,\rho,x)=\inf\limits_{\mathbb V\in\mathscr V_{\ell}[t,T]}\mathcal J(t,\rho,x;\mathbb V)$ by studying the following HJB equation  
\begin{align}&\frac{\partial}{\partial t}U(t,\rho,x)
+\inf_{\mathbb V\in B_{\ell}} \Big\{\Big\langle \partial_{\rho}U(t,\rho,x), D_x\mathcal H^{\mathbb V}_0(\rho,x) \Big\rangle - \Big\langle D_xU(t,\rho,x), D_{\rho}\mathcal H^{\mathbb V}_0(\rho,x)\Big\rangle\notag\\& + \frac{1}{2} \mathrm{tr}(\sigma \sigma^\top D^2_xU(t,\rho,x)) + F(t, \rho, x, \mathbb V)\Big\}=0, \;\;t \in (0,T), \rho \in \mathcal P^{\circ}(G), x \in \mathbb R^n,\label{HJB-intro}
\end{align}
with the terminal condition $U(T,\rho,x)=h(\rho,x)$, where $\partial_{\rho}U$ is the Fr\'echet derivative on the Wasserstein space on graphs, and $D_xU,D^2_xU$ are respectively,  the usual Euclidean gradient and Hessian matrix; see Section \ref{sec_4}  for more details.

The concept of viscosity solutions \cite{CL83}  for HJB equations  has become a fundamental tool in optimal control theory \cite{YongJM2,LionsI,LionsII,AAP3,YongJM,Atar} and many other related fields \cite{ZJJ1,Hata12,AAP1}. For comprehensive accounts of the theory, we refer the reader to  e.g. \cite{guide,book-controlled}; 
For developments in infinite-dimensional Hilbert space and Wasserstein space, see e.g.  \cite{Fabbri,Gozzi,Lions-infi1,unbounded} and \cite{Erhan,approximation_HJB,approximation_HJB2}, respectively. Nevertheless, the study of HJB equations with graph structures is still at an early stage. When graph structures are involved, the HJB equation \eqref{HJB-intro} exhibits several distinctive features.  First, it exhibits a mixed geometric structure,  involving both 
the graph-structured Wasserstein geometry on the density space and 
the usual Euclidean geometry on the momentum space. In particular, the interplay between the underlying graph geometry and the probability simplex constraint also introduces intrinsic structural complexities. Second, the dominant energy $\mathcal H^{\mathbb V}_0$ of the SWHS \eqref{SWHS-1} may contain penalization terms resembling the discrete Fisher information 
\cite{ZhouJFA}, which is formulated in terms of the graph structure and logarithmic geometry.  As a result, the HJB equation \eqref{HJB-intro} is defined only on the interior of the probability simplex, which is different from the classical HJB equation on bounded Euclidean domains with certain boundary conditions. 
Furthermore, the intricate structure of the state equation \eqref{SWHS-1} causes
 the absence of the moment estimates for solutions $\rho(\cdot)$ and $S(\cdot)$ themselves, which brings additional difficulties for establishing the uniqueness of the viscosity solution on the unbounded domain $\mathcal P^{\circ} (G)\times\mathbb R^n$.  
In this setting, it becomes necessary to introduce a mass conservation assumption on the SWHS to ensure the uniqueness of the solution.

The main contribution of this paper is to establish the existence and uniqueness of the viscosity solution for the HJB equation \eqref{HJB-intro}. This is achieved within a framework that integrates dynamic programming and the doubling of variables methods, while carefully reconciling the discrete graph geometry with probability space constraints. To this end, we first employ the dynamic programming principle method to investigate the connection between the value function of the stochastic optimal control problem and the associated HJB equation \eqref{HJB-intro} (see Section \ref{sec_33}). A key prerequisite is to establish the Bellman principle of optimality (see Proposition \ref{dynamic-pro}),  which asserts that the optimal cost at a given state and time can be decomposed into the cost over an initial interval and the optimal cost from the resulting state at a later time. 
By introducing the stopping time technique to truncate the energy functional,  and utilizing moment estimates for the energy of the SWHS  (see Proposition \ref{exact}), we prove the continuity and growth properties of the value function for the optimal control problem (see Proposition \ref{conti}). Combining the Bellman principle with the It\^o formula, 
we then demonstrate that the value function satisfies the viscosity sub- and super-solution inequalities,  thereby confirming that the value function of the optimal control  problem is indeed a viscosity solution of the HJB equation (see Theorem \ref{thm_existence}).

To further establish the uniqueness of the viscosity solution, we adopt the doubling of variables method adapted to the graph-structural setting (see Section \ref{sec_5}). The proof relies on two crucial technical ingredients. First, to address the difficulty caused by the unbounded domain for the state equation, we introduce an energy-truncation procedure that localizes the problem to compact subsets via smooth cut-off functions. 
 This leads to the formulation of a truncated HJB equation \eqref{HJB_R}, which preserves the structure of the original HJB equation within the level set of the energy. 
The second technical ingredient concerns the regularity properties at the maximum point of the auxiliary bivariate functional arising in the doubling of variables method.  
To this end, we construct modified auxiliary functionals (see \eqref{auxi-1} and \eqref{auxi-2}), equipped with truncation functions and barrier functions that could simultaneously control the growth in the Euclidean directions and respect the Wasserstein structure imposed by the graph geometry. 
 These constructions enable us to simultaneously control the growth in the Euclidean directions while accommodating the Wasserstein structure dictated by the graph geometry and the probability simplex constraint. By establishing the necessary regularity properties for the auxiliary function, we derive the viscosity inequalities along with suitable estimates, which ultimately ensure the uniqueness of the viscosity solution 
 (Steps 3--4 in the proof of Theorem \ref{uniqueness} in Section \ref{sec_6-2}).

Furthermore, by considering various forms of dominant energies, we apply our results to optimal control problems for the SSE \eqref{schro-intro} on graphs, including both the stochastic nonlinear Schr\"odinger
equation with polynomial nonlinearity and stochastic logarithmic Schr\"odinger equation on
graphs. In both cases, we establish the existence and uniqueness of viscosity solutions to the associated HJB equations arising from the corresponding stochastic optimal control problems. These findings may offer valuable insights into the study of optimal control problems in areas such as quantum information processing and coherent light transport on graph structures.

This paper is organized as follows.  Section \ref{sec_2} introduces
preliminaries, including notations and the well-posedness result for SWHSs on graphs. 
In Section \ref{sec_4},  we present our main results concerning the stochastic optimal control problem and the associated HJB equation on the Wasserstein space over graphs. Section \ref{sec_6} discusses applications of our main results to the optimal control problems involving stochastic Schr\"odinger equations on graphs.  Section \ref{sec_33} is devoted to proving that the value  function of the optimal control is a viscosity solution of the associated HJB equation.  In Section \ref{sec_5}, we provide the proof of the uniqueness for the viscosity
solution of the HJB equation. Appendix \ref{app1} contains proofs of several propositions and lemmas, as well as some useful results from convex analysis.

\section{Preliminaries}\label{sec_2}

In this section, we first introduce some frequently used notations and provide definitions of the Wasserstein space and Wasserstein derivatives on graphs. We then present the SWHS on a finite graph and establish its well-posedness.

\subsection{Notations}  For a positive integer $n\in\mathbb N_+,$
we denote by $\mathcal P(G)$ the probability simplex 
\begin{align*}
\mathcal P(G)=\Big\{\rho=(\rho_1,\ldots,\rho_n)\in[0,1]^n:\sum_{i=1}^n\rho_i=1\Big\}.
\end{align*} 
For fixed $\epsilon\in(0,\frac1n),$
let $\mathcal P_{\epsilon}(G):=\mathcal P(G)\cap [\epsilon,1)^n$. By $\partial A$ and $A^{\circ}$ we denote the boundary and interior of a Borel  set $A$. 
 Accordingly, $\mathcal P^{\circ}(G)$ is the interior of $\mathcal P(G),$ $\mathcal P^{\circ}_{\epsilon}(G)$ is the interior of $\mathcal P_{\epsilon}(G)$ for $\epsilon\in(0,\frac1n),$ and  $\partial \mathcal P(G)=\mathcal P(G)\backslash \mathcal P^{\circ}(G)$ and $\partial\mathcal P_{\epsilon}(G)=\mathcal P_{\epsilon}(G) \backslash\mathcal P^{\circ}_{\epsilon}(G).$  
We use $\mathbb S^{n\times n}$ and $\mathcal S^{n\times n}$ to denote the sets of $n\times n$ skew-symmetric and  symmetric matrices, respectively.  

Throughout this paper, we impose the following assumption on the symmetric function $g$. \begin{assumption}\label{assumption_g}
\begin{itemize}
\item[(g-\romannumeral1)] $g$ is continuous on $[0,\infty)^2$ and is smooth on $(0,\infty)^2$;
\item[(g-\romannumeral2)] $t\wedge r\leq g(t,r)\leq t\vee r$ for any $t,r\in[0,\infty);$
\item[(g-\romannumeral3)] $g(\lambda t,\lambda r)=\lambda g(t,r)$ for any $\lambda,t,r\in[0,\infty);$
\item[(g-\romannumeral4)] $g$ is concave;   
\item[(g-\romannumeral5)] $\int_0^1\frac{1}{\sqrt{g(r,1-r)}}\mathrm dr<\infty.$
\end{itemize}
\end{assumption}  
  Examples of the admissible functions $g$ include the averaged probability weight \cite{ZhouJFA}: $g(t,r)=\frac{t+r}{2};$  the logarithmic probability weight \cite{ARMA2012}:  $g(t,r)=\frac{t-r}{\log t-\log r};$ and the harmonic probability weight \cite{Maas}: $g(t,r)=\frac{2}{\frac1t+\frac1r}.$ 
 
For $\rho\in\mathcal P(G),$ we say that $\upsilon,\tilde\upsilon\in\mathbb S^{n\times n}$ are $\rho$-equivalent if $$(\upsilon_{ij}-\tilde\upsilon_{ij})g_{ij}(\rho)=0,\quad \forall \,(i,j)\in E,$$ where $g_{ij}(\rho):=g(\rho_i,\rho_j).$ This equivalence relation  induces a quotient space on $\mathbb S^{n\times n}$ denoted by $\mathbb H_{\rho}.$ Under Assumption \ref{assumption_g},  one can use the function $g$ to define a metric tensor on $\mathcal P(G)$ and endow $\mathbb H_{\rho}$ with 
the inner product and discrete norm 
\begin{align}\label{g1}
\langle \upsilon,\tilde\upsilon\rangle_{\rho}:=\frac12\sum_{(i,j)\in E}\upsilon_{ij}\tilde\upsilon_{ij}g_{ij}(\rho),\quad \|\upsilon\|_{\rho}:=\sqrt{\langle\upsilon,\upsilon\rangle_{\rho}},\;\upsilon,\tilde\upsilon\in\mathbb S^{n\times n}.
\end{align}
Here the coefficient $\frac12$ accounts for the fact that whenever $(i,j)\in E$ then $(j,i)\in E.$  
For a $\mathbb R^n$-valued mapping $\phi:V\to\mathbb R^n,$ define its graph gradient as $\nabla_G\phi:=\sqrt{\omega_{ij}}(\phi_i-\phi_j)_{(i,j)\in E}.$ The adjoint of $\nabla_G$ with respect to the  inner product $\langle\cdot,\cdot\rangle_{\rho}$ is the divergence operator $-\mathrm{div}_{\rho}$ given by
\begin{align*}
\mathrm{div}_{\rho}(\upsilon):=\Big(\sum_{j\in N(i)}\sqrt{\omega_{ij}}\upsilon_{ji}g_{ij}(\rho)\Big)_{i=1}^n,\;\upsilon\in\mathbb S^{n\times n}.
\end{align*} 
Then we have the following  integration by parts formula \begin{align}\label{IBP}\langle\nabla_G\phi,\upsilon\rangle_{\rho}=-\langle\phi,\mathrm{div}_{\rho}(\upsilon)\rangle, 
\end{align}
where $\langle\upsilon,\tilde\upsilon\rangle:=\frac12 \sum_{(i,j)\in E}\upsilon_{ij}\tilde\upsilon_{ij}$ for $\upsilon,\tilde\upsilon\in\mathbb S^{n\times n}$. 
We use $\|\cdot\|$ to denote the standard Euclidean norm of a vector. For a function $\psi:\mathbb R^n\to\mathbb R,$ its gradient and Hessian are denoted by $D_x\psi,D^2_x\psi$.  
For multi-variables real-valued function $\psi(x,y),x,y\in\mathbb R^n,$ we also use notations $\frac{\partial\psi(x,y)}{\partial x},\frac{\partial\psi(x,y)}{\partial y}$ to denote the usual Euclidean partial derivatives.

For $\rho^0,\rho^1\in\mathcal P(G),$ define the $L^2$-Monge--Kantorovich metric 
\begin{align}\label{MK}
\mathcal W(\rho^0,\rho^1):=\Big(\inf_{(\sigma,\upsilon)}\Big\{\int_0^1\langle\upsilon,\upsilon\rangle_{\sigma(t)}\mathrm dt:\dot{\sigma}+\mathrm {div}_{\sigma}(\upsilon)=0,\;\sigma(0)=\rho^0,\,\sigma(1)=\rho^1\Big\}\Big)^{\frac12},
\end{align}
where the infimum is taken over the set of pairs $(\sigma,\upsilon)$ such that $\sigma\in H^1(0,1;\mathcal P(G))$ and $\upsilon :[0,1]\to\mathbb S^{n\times n}$ is measurable. Condition (g-\romannumeral5) in Assumption \ref{assumption_g} guarantees that $\mathcal W(\rho^0,\rho^1)<\infty$ for any $\rho^0,\rho^1\in\mathcal P(G);$ see \cite[Proposition 3.7]{Liwuchen}. The space $(\mathcal P(G),\mathcal W)$ is called the \textit{Wasserstein space on graphs}; see e.g. \cite{Liwuchen,MCC,Maas,ARMA2012} for further details.  
 One can define the  \textit{$\mathcal W$-differentiability}, the \textit{Wasserstein gradient} $\nabla_{\mathcal W}$, as well as the \textit{Fr\'echet derivative} $\partial_{\rho}$ on the Wasserstein space $(\mathcal P(G),\mathcal W)$; see \cite[Definitions 3.9 and 3.11]{MCC}. Below we denote by  $\mathcal T_{\rho}\mathcal P(G)$ the closure of the range of $\nabla_G$ in $\mathbb H_{\rho}$, and refer to $T _{\rho}\mathcal P(G)$ as the tangent space of $\mathcal P(G)$ in Euclidean space $\mathbb R^n$. 
 We also use the notation $\|p\|_{l^1}:=\sum_{i=1}^n|p_i|$ for $p\in\mathbb R^n.$ 
 \begin{definition}\label{def-derivatives}
(\romannumeral1) A function $f:\mathcal P(G)\to\mathbb R$ is  is said to be $\mathcal W$-differentiable at $\rho$ if there exist $v\in \mathcal T_{\rho}\mathcal P(G)$ and $C>0$ such that 
for every $\epsilon>0$ there exists $\delta>0$ with the following property: for any $\tilde v\in \mathcal T_{\rho}\mathcal P(G)$ and $\|\mu-\rho\|_{l^1}\leq \delta$ for $\mu\in\mathcal P(G)$, we have 
 \begin{align}\label{Wass-gradient}
  |f(\mu)-f(\rho)-\langle v,\tilde v\rangle_{\rho}|\leq \epsilon\mathcal W(\mu,\rho)+C\|\mu-\rho+\mathrm{div}_{\rho}(\tilde v)\|_{l^1}.
 \end{align}
 We refer to the vector 
$v$ as the Wasserstein gradient of $f$ at $\rho,$ and denote it by 
 $\nabla_{\mathcal W}f(\rho)$.

(\romannumeral2) We say that the function $f:\mathcal P(G)\to\mathbb R$ has a \textit{F\'echet derivative} at $\rho\in\mathcal P(G)$ if there exists 
$p\in T_{\rho}\mathcal P(G)$ such that
\begin{align}\label{Frechet}
\sum_{i=1}^np_i=0, \text{ and }\lim_{s\to0}\frac{f((1-s)\rho+s\mu)-f(\rho)}{s}=\langle p,\mu-\rho\rangle,\quad \forall \mu\in\mathcal P(G).
\end{align} 
We denote the vector $p$ by ${\partial_{\rho} f}(\rho)$.
\end{definition} 

\begin{remark}
(\romannumeral1) From \cite[Lemma 3.15]{MCC}, we know that  if the Fr\'echet derivative of $f$ exists and it is continuous at $\rho\in\mathcal P^{\circ}(G)$, then we have \begin{align}\label{relation-derivative}
\nabla_{\mathcal W}f(\rho)=\nabla_{G}{\partial_{\rho} f}(\rho).\end{align}
According to Definition \ref{def-derivatives} (i), the Wasserstein gradient takes values in the space $\mathcal T_{\rho}\mathcal P(G),$ and thus is a skew-symmetric matrix. The upper triangular entries of $\nabla_{\mathcal W}f(\rho)$ are: $\sqrt{\omega_{ij}}\nabla^{e_{ij}}f(\rho),1\leq i<j\leq n$. 
Here, for any $(i,j)\in E$ such that $1\leq i<j\leq n,$ we define $e_{ij}\in\mathbb R^n$ whose $i$th entry is $1$, $j$th entry is $-1$ and other entries are zero, and  
\begin{align*}
\nabla^{e_{ij}}f(\rho):=\lim_{t\to0}\frac{f(\rho+te_{ij})-f(\rho)}{t},\quad \rho\in\mathcal P^{\circ}(G).
\end{align*}

(\romannumeral2) For a function $f:\mathcal P(G)\to\mathbb R$, one can calculate that 
\begin{align}\label{Fre2}\partial_{\rho}f(\rho)=Df(\rho)-\frac1nDf(\rho)^{\top}\mathbf 1=\sum_{j=1}^n\hat e_j(Df(\rho))_j,
\end{align}
 where $\mathbf 1=(1,\ldots,1)\in\mathbb R^n,$ $Df$ is the Euclidean gradient, and  $\hat e_j$ is a vector in $\mathbb R^n$ whose $j$th entry is $1-\frac1n$ and other entries are $-\frac1n.$ 
\end{remark}
Throughout this paper, 
we use
$C$ to denote a generic positive constant which may take different values at different appearances. 
Sometimes we write $C(a,b)$ or $C_{a,b}$ to emphasize the dependence on
the parameters $a,b.$  
We use $\Re$ and $\Im $ to denote real and imaginary parts of a complex number. Denote $a\vee b:=\max\{a,b\},a\wedge b:=\min\{a,b\}$ for real-valued numbers $a,b.$ We use $\mathbf 1_{A}$ to denote the indicator function of a measuable set $A$.

 \subsection{Stochastic Wasserstein--Hamiltonian system on a graph} 
In this subsection, we introduce the SWHS,  
and present the well-posedness result of the solution and the moment estimate of the energy functional.  
Consider the SWHS on the cotangent bundle of the density manifold $\mathcal P(G)$:
\begin{align}
\mathrm d\rho(t) &= \frac{\partial }{\partial S}\mathcal H^{\mathbb V}_0(\rho(t) ,S(t))\,\mathrm dt+\frac{\partial }{\partial S} \mathcal H_1(\rho(t) ,S(t))\circ\mathrm dW(t),\label{WHflow1}\\
\mathrm dS(t)&=-\frac{\partial }{\partial \rho}\mathcal H^{\mathbb V}_0(\rho(t) ,S(t))\,\mathrm dt-\frac{\partial }{\partial \rho} \mathcal H_1(\rho(t) ,S(t))\circ \mathrm dW(t),\label{WHflow2}
\end{align} 
where $t\in(t_0,T]$ for some constant $t_0\in(0,T]$ with $T>0$ being the final time, $\{ W(t)=(W_1(t),\ldots,W_n(t))\}_{t\ge t_0}$ is a standard $n$-dimensional Brownian motion on a complete filtered probability space $(\Omega,\mathcal F ,\{ \mathcal F_t\}_{t\ge t_0},\mathbb P )$, and $``\circ"$ stands for the Stratonovich type stochastic differential. 
Here, 
 $\mathcal H^{\mathbb V}_0$ is the dominant energy \cite{Cui23}, and  we use the subscript `$\mathbb V$' in $\mathcal H^{\mathbb V}_0$ to indicate its dependence on the control variable $\mathbb V=(\mathbb V_1,\ldots,\mathbb V_n)$; $\mathcal H_1$ is some $\mathbb R^n$-valued function. For simplicity, the initial value is assumed to be deterministic: $\rho(t_0)=\rho\in\mathcal P^{\circ}(G),$ $S(t_0)=x\in\mathbb R^n.$  We remark that one may also study the system with random initial values satisfying the $p$th moment boundedness; see e.g., \cite{cuiSIAM}.

To obtain the well-posedness of the system, we need to give a precise characterization of the energy functions. To this end, we introduce two classes of functions as follows.
 \begin{definition}\label{def-functions}
 (\romannumeral1) We say a function $\eta\in (\mathrm{Pr})$ if it  satisfies: 
\begin{itemize}
\item[($\mathrm{Pr}$-\romannumeral1)] The function $\eta\in\mathcal C^{1}(\mathcal P^{\circ}(G)\times\mathbb R^n;\mathbb R)$  has a lower bound, i.e., there exists $C_0\in\mathbb R$ such that $\eta-C_0\ge 0.$ 
\item[($\mathrm{Pr}$-\romannumeral2)]  For each $R>0,$ $\eta(\rho,x)\leq R$  implies 
\begin{align}\label{Pr-set}
\|D_{\rho}\eta\|\vee \|D_x\eta\|\leq C_1,\;
\min_{i\in V}\rho_i\ge c, \;\sup_{(i,j)\in E}|x_i-x_j|\leq C_2,
\end{align} 
 for some constants $c\in(0,1),C_1,C_2>0$ depending on $R.$ 
 
\item[($\mathrm{Pr}$-\romannumeral3)] $\eta(\rho,x)\to+\infty$ as $\rho\to\partial\mathcal P(G)$ or $|x_i-x_j|\to+\infty$ where $(i,j)\in E.$ 
 \end{itemize} 
 
(\romannumeral2) We say a function $f:\mathcal P^{\circ}(G)\times\mathbb R^n\to\mathbb R$ is \textit{locally Lipschitz continuous} if 
for each fixed $R_1\in(0,\frac1n)$ and $R_2>0,$ when $ \min\limits_{i\in V}\rho_i\wedge \min\limits_{i\in V}\mu_i\ge R_1$ and $ \max\limits_{(i,j)\in E}| x_i- x_j|\vee \max\limits_{(i,j)\in E}|y_i-y_j|\leq R_2,$ it holds 
\begin{align}\label{locally-Lip}
|f(\rho,x)-f(\mu,y)|\leq C(\|\rho-\mu\|+\|x-y\|),
\end{align}
where $C:=C(R_1,R_2)>0$ and $\|\cdot\|$ is the standard Euclidean norm of a vector. 
\end{definition}  In Section \ref{sec_6}, 
we will present  concrete  examples that satisfy the properties given in Definition \ref{def-functions}. To proceed, we propose the assumptions on $\mathcal H^{\mathbb V}_0$ and $\mathcal H_1$.   
\begin{assumption}  \label{energy} 
Let the dominant energy be given by
\begin{align}\label{def-H0}
\mathcal H^{\mathbb V}_0(\rho,x)=\mathcal H_0(\rho,x)+\mathcal V(\rho),\quad \rho\in\mathcal P^{\circ}(G),x\in\mathbb R^n,
\end{align} where $\mathcal V(\rho ) := \sum_{i=1}^n\mathbb V_i\rho_i$ is  the linear control potential, and $
\mathcal H_0\in(\mathrm{Pr})$    (see Definition \ref{def-functions}). Assume that $\mathcal H_0$ and its Euclidean partial derivatives are locally Lipschitz continuous in the sense of Definition \ref{def-functions} (\romannumeral2).  
Define the $\mathbb R^n$-valued function $\mathcal H_1$ by  
\begin{align}\mathcal H_1(\rho ,x) =
(\sigma_1\rho_1,\ldots,\sigma_n\rho_n)^{\top}
=:\sigma\rho, \label{def-H1}
\end{align}
where  
$\sigma := \mathrm{diag}(\sigma_1, \ldots, \sigma_n)$. 
\end{assumption}

The well-posedness of \eqref{WHflow1}--\eqref{WHflow2} is stated in the following proposition, whose proof is postponed to Appendix \ref{app1}. We refer to \cite[Theorem 4.1]{cuiSIAM} for more details and also refer to \cite[Theorem 3]{ZhouJFA} for the study of deterministic WHS.

\begin{proposition}\label{exact}
Let Assumptions \ref{assumption_g} and \ref{energy} hold. In addition assume that  
\begin{align}\label{bound_prop}\big\|D_x \mathcal H^{\mathbb V}_0(\rho,x)\big\|+\sum_{i=1}^n
\Big|\frac{\partial^2\mathcal H^{\mathbb V}_0}{\partial x_i^2}(\rho,x)\Big|\leq C_1\mathcal H^{\mathbb V}_0 (\rho,x)+C_2,
\end{align} with some $C_1,C_2>0.$ 
Then the system \eqref{WHflow1}--\eqref{WHflow2} admits a unique  solution on $[t_0,T]$  satisfying
\begin{align}
&\mathbb E\Big[\sup_{t\in[t_0,T]}|\mathcal H^{\mathbb V}_0(\rho(t),S(t))|^p\Big]\leq C\big(1+|\mathcal H^{\mathbb V}_0(\rho,x)|^p\big), \label{H_0estimate}\\
 &\mathbb E\Big[\sup_{r\in[t_0,t]}|\mathcal H^{\mathbb V}_0(\rho(r),S(r))-\mathcal H^{\mathbb V}_0(\rho,x)|^p\Big]\leq C(t-t_0)^{\frac p2},\label{moment-cont-1}
\end{align} 
for all $p\ge 2$, where   $C:=C(p,T)>0$ is some constant. 
\end{proposition}

\section{Main result: stochastic optimal control and HJB equation}\label{sec_4}

In this section, we present our main results on the relationship between the value function of a stochastic optimal control problem and the solution of the corresponding HJB equation. We consider \eqref{WHflow1}--\eqref{WHflow2} as the state dynamics of the system.

Fix $\ell > 0$ and define the control region 
$
B_{\ell} := \{x \in \mathbb R^n :  \|x\| \le \ell\}.
$
For $0\leq t <T$, the admissible control set is defined by
\begin{align*}
\mathscr V_{\ell}[t,T] := \Big\{ \mathbb V : \Omega \times [t,T] \to B_{\ell} \, \big| \, \mathbb V(s) \text{ is } \mathcal F_s\text{-adapted and } \|\mathbb V(s)\| \le \ell \text{ a.s. for all } s\in[t,T] \Big\}.
\end{align*} 
Given a control $\mathbb V \in \mathscr V_{\ell}[t,T]$, we define the cost functional
\begin{align}\label{def-costfunction}
\mathcal J(t, \rho, x; \mathbb V) := \mathbb E_t \Big[ \int_t^T F(s, \rho(s), S(s), \mathbb V(s)) \, \mathrm ds + h(\rho(T), S(T)) \Big],
\end{align}
where $(\rho(\cdot), S(\cdot))$ is the solution of \eqref{WHflow1}--\eqref{WHflow2} with initial conditions $\rho(t) = \rho$ and $S(t) = x$, and $\mathbb E_t[\cdot] := \mathbb E[\cdot \, | \, \mathcal F_t]$ denotes the  conditional expectation. This functional represents the accumulated expected cost over time, including both the running cost $F$ and the terminal cost $h$. We define the \emph{value function} for 
the stochastic optimal control problem:
for $(t, \rho, x) \in [0,T] \times \mathcal P^{\circ}(G) \times \mathbb R^n$, 
\begin{align}\label{def-valuefunction}
U(t, \rho, x):= \inf_{\mathbb V \in \mathscr V_{\ell}[t,T]} \mathcal J(t, \rho, x; \mathbb V).
\end{align}
The value function characterizes the minimal expected cost starting from state $(\rho,x)$ at time $t$.

When the cost functions $F$ and $h$ are nonnegative, the infimum of $\mathcal J$ always exists. Moreover, a minimizing sequence can be found: for any $(t, \rho, x) \in [0,T] \times \mathcal P^{\circ}(G) \times \mathbb R^n$, there exists a sequence $\{\mathbb V_k\}_{k \ge 1} \subset \mathscr V_{\ell}[t,T]$ such that
\[
\mathcal J(t, \rho, x; \mathbb V_k) \to U(t, \rho, x) \quad \text{a.s. as } k \to \infty.
\]
We refer to \cite[Lemma 4.4]{YongJM} for further details. 
We point out that the existence of the infimum of $\mathcal J$ does not necessarily guarantee the existence of an optimal control $\mathbb V^*$, i.e., the infimum may not be attained by any admissible control. In this work, we do not focus on the existence of an optimal control. Instead, we aim to study the value function \eqref{def-valuefunction} through the dynamic programming principle and to characterize it as a \emph{viscosity solution} of the associated HJB equation. 

To introduce the HJB equation, we now  
 define the functional
\begin{align}\label{def-mathbbH}
\mathbb H(t, \rho, x, \mathbb V, p, q, Q) := \langle p, D_x\mathcal H^{\mathbb V}_0(\rho,x) \rangle - \langle q, D_{\rho}\mathcal H^{\mathbb V}_0(\rho,x) \rangle + \frac{1}{2} \mathrm{tr}(\sigma \sigma^\top Q) + F(t, \rho, x, \mathbb V),
\end{align}
where $t \in [0,T], \rho \in \mathcal P^{\circ}(G), x \in \mathbb R^n, \mathbb V \in B_\ell, p,q \in \mathbb R^n, Q \in \mathcal S^{n \times n}$, and $\mathcal H^{\mathbb V}_0$ is defined by \eqref{def-H0}.   
The Hamiltonian functional is defined as 
\begin{align}\label{def_H}
H(t, \rho, x, p, q, Q) :&= \inf_{\mathbb V \in B_\ell} \mathbb H(t, \rho, x, \mathbb V, p, q, Q) 
\notag\\
&= \langle p, D_x\mathcal H_0(\rho,x) \rangle - \langle q, D_{\rho}\mathcal H_0(\rho,x)  \rangle + \frac{1}{2} \mathrm{tr}(\sigma \sigma^\top Q) - \widehat F(t, \rho, x, q),
\end{align}
where
\[
\widehat F(t, \rho, x, q) := \sup_{\mathbb V \in B_\ell} \big\{ \langle q, \mathbb V \rangle - F(t, \rho, x, \mathbb V) \big\}.
\] Then the associate HJB equation is given as follows 
\begin{align}\label{HJB}
&\frac{\partial U}{\partial t}  + H(t, \rho, x, \partial_{\rho} U, D_x U, D^2_x U) = 0, \quad (t, \rho, x) \in (0,T) \times \mathcal P^{\circ}(G) \times \mathbb R^n,\\
&U(T, \rho, x) = h(\rho, x),\notag
\end{align}
where $\partial_{\rho} U$ denotes the Fr\'echet  derivative with respect to $\rho\in\mathcal P^{\circ}(G)$ (Definition \ref{def-derivatives}), and $D_xU$, $D^2_x U$ are the first- and second-order Euclidean derivatives with respect to $x\in\mathbb R^n$.  
\begin{remark}\label{second-order1}
We would like to mention that the HJB equation \eqref{HJB} is of second-order only in the $x$ variable, which is due to that
the perturbed energy $\mathcal H_1$ \eqref{def-H1} is assumed to be independent of $x$. An interesting and valuable direction for future study is the case where $\mathcal H_1$ depends on both $\rho$ and $x$, which results in the second-order Wasserstein derivatives in both $\rho$ and $x$ variables for the underlying HJB equation. 
\end{remark}
Next, we introduce the definition of the viscosity solution of \eqref{HJB}; see e.g. \cite[Definition I.1, Proposition I.1]{LionsII} and  \cite[Definition 2]{Crandall1}.

\begin{definition}\label{def1-1}
(\romannumeral1) 
A continuous function $U$ is called a viscosity subsolution of \eqref{HJB} if $U(T,\cdot) \le h(\cdot)$ and for every $(t_0, \rho_0, x_0) \in (0,T) \times \mathcal P^{\circ}(G) \times \mathbb R^n$ and every $\varphi \in \mathcal C^{1,1,2}((0,T)\times\mathcal P^{\circ}(G)\times\mathbb R^n)$ such that $U-\varphi$ has a local maximum at $(t_0, \rho_0, x_0)$, 
we have
\begin{align}\label{viscosity-sub}
 \frac{\partial\varphi}{\partial t}(t_0, \rho_0, x_0) + H(t_0, \rho_0, x_0, \partial_{\rho} \varphi(t_0, \rho_0, x_0), D_x \varphi(t_0, \rho_0, x_0), D^2_x \varphi(t_0, \rho_0, x_0)) \ge 0.
\end{align}

(\romannumeral2) A continuous function $U$ is called a viscosity supersolution of \eqref{HJB}  if $U(T,\cdot) \ge h(\cdot)$ and for every $(t_0, \rho_0, x_0) \in (0,T) \times \mathcal P^{\circ}(G) \times \mathbb R^n$ and every $\varphi \in \mathcal C^{1,1,2}((0,T)\times\mathcal P^{\circ}(G)\times\mathbb R^n)$ such that $U-\varphi$ has a local minimum at $(t_0, \rho_0, x_0)$, 
we have 
\begin{align}\label{viscosity-super}
 \frac{\partial\varphi}{\partial t}(t_0, \rho_0, x_0) + H(t_0, \rho_0, x_0, \partial_{\rho} \varphi(t_0, \rho_0, x_0), D _x\varphi(t_0, \rho_0, x_0), D^2_x \varphi(t_0, \rho_0, x_0))\le 0.
\end{align}

(\romannumeral3)  A continuous function $U$ is called a viscosity solution of \eqref{HJB} if it is both a viscosity subsolution and a viscosity supersolution. 
\end{definition}

It can be verified that any classical solution $U \in \mathcal C^{1,1,2}$ of \eqref{HJB} is automatically a viscosity solution, see e.g. \cite{Crandall1,book-controlled,Achdou}. 
In general, the value function may not be smooth enough to satisfy \eqref{HJB} in the sense of classical solution. Instead, under suitable assumptions on the running cost $F$ and terminal cost $h$, we show that the value function is indeed a viscosity solution. Recall that $B_{\ell}$ is the control region.

\begin{assumption}\label{ass1}
(\romannumeral1) The functionals $F: [0,T] \times \mathcal P^{\circ}(G) \times \mathbb R^n \times B_\ell \to \mathbb R$ and $h: \mathcal P^{\circ}(G) \times \mathbb R^n \to \mathbb R$ are nonnegative and continuous. 

(\romannumeral2) For $s\in[0,T]$ and $\mathbb V\in B_{\ell},$ functionals $F(s,\cdot,\cdot,\mathbb V)$ and $ h$ are locally Lipschitz continuous in the sense given by \eqref{locally-Lip}.

(\romannumeral3) For $(\rho,x,\mathbb V)\in\mathcal P^{\circ}(G)\times\mathbb R^n\times B_{\ell},$ $F(\cdot,\rho,x,\mathbb V)$ is globally Lipschitz continuous on $[0,T].$  

(\romannumeral4) There exists constant $p_1 \ge  0$ such that for all $s \in [0,T],\mathbb V\in B_{\ell}$,
\begin{align}\label{ass11}
F(s, \rho, x, \mathbb V) \vee h(\rho, x) \le C \big( |\mathcal H_0(\rho, x)|^{p_1}  + 1 \big),
\end{align} 
for some constant $C := C(T, p_1, \ell) > 0$, where $\mathcal H_0$ is given in \eqref{def-H0}.  
\end{assumption}

The following theorem states that the value function $U$ of the stochastic optimal control problem is a viscosity solution of the HJB equation \eqref{HJB}.  This, in particular,  implies the existence of the viscosity solution of \eqref{HJB}. The proof is given in Section \ref{sec_33}.  \begin{theorem}\label{thm_existence}
Let Assumptions \ref{assumption_g}, \ref{energy},  \ref{ass1} and condition \eqref{bound_prop} hold. Then  
the value function $U$ defined by \eqref{def-valuefunction} is a viscosity solution of \eqref{HJB} in $(0,T)\times\mathcal P^{\circ}(G)\times\mathbb R^n$.
\end{theorem}
Next we present the uniqueness of the viscosity solution of the HJB equation \eqref{HJB}, whose proof is postponed to Section \ref{sec_5}. 
 To this end, we further impose the following assumption on $\mathcal H_0$. 
\begin{assumption}\label{ass_domi}
Suppose that for each $(\rho,x)\in\mathcal P^{\circ}(G)\times\mathbb R^n,$ it holds 
\begin{align}\label{assu4-eq}D_x\mathcal H_0(\rho,x)\in T_{\rho}\mathcal P(G),\;\;\text{ i.e., }  \sum_{i=1}^n(D_x\mathcal H_0(\rho,x))_i=0.
\end{align} In addition assume that 
\begin{align}
\lim_{R\to\infty}\mathbf 1_{\{\mathcal H_0(\rho, x)\in (R,2R)\}}\Big[\frac{1}{R} \|D^2_x\mathcal H_0(\rho, x)\|+\frac{1}{R^{2-\beta}}\|D_x\mathcal H_0(\rho, x)\|^2\Big]=0,\label{cond-unique1}
\end{align} 
for some $\beta\in(0,1).$ 
\end{assumption}

\begin{remark}
The condition \eqref{assu4-eq} is imposed for several reasons. First, it encompasses the dominant energy associated with the stochastic Schr\"odinger equation (see Corollaries \ref{coro-1} and \ref{coro-2} in Section \ref{sec_6}). Second, it ensures the preservation of the mass of the density $\{\rho(t)\}_{t\in[0,T]}$ in the SWHS \eqref{WHflow1}. Third, it is introduced for technical purposes in the proof of the uniqueness, to overcome the difficulty caused by structural properties of coefficients for the SWHS  
(see \eqref{technical-eq} in Section \ref{sec_5}).
\end{remark}

\begin{theorem}\label{uniqueness}
Let Assumptions \ref{assumption_g}, \ref{energy},  \ref{ass1},  \ref{ass_domi} and condition \eqref{bound_prop} hold. In addition assume that costs $F$ and $h$ are uniformly bounded.  Then there exists a unique bounded viscosity solution $U$ of \eqref{HJB} in $(0,T)\times\mathcal P^{\circ}(G)\times\mathbb R^n$.    
\end{theorem}

\section{Applications to stochastic 
Schr\"odinger equations on graphs}\label{sec_6}
In this section, we consider two important examples of the SWHS \eqref{WHflow1}--\eqref{WHflow2}: the stochastic Schr\"odinger equation (SSE) with polynomial nonlinearity on a finite graph \cite{ZhouJFA,Cui23}, and the stochastic logarithmic Schr\"odinger equation \cite{cuiSIAM} on   a finite graph graph. 
Then we apply the existence and uniqueness results of HJB equations (Theorems  \ref{thm_existence} and \ref{uniqueness}) to the optimal control problems for these quantum models on graphs.

To give a concrete example of the dominant energy, 
we introduce several functionals on $\mathcal P^{\circ}(G)\times\mathbb R^n$.  
First, we define the kinetic energy as
\begin{align}\label{kinetic}
K(\rho,x ) 
= \frac12\sum_{i=1}^n\sum_{j\in N(i)}\omega_{ij}(x_i-x_j)^2g_{ij}(\rho),
\end{align}
where $g$ is the probability weight satisfying Assumption \ref{assumption_g}. 
 The interaction potential and  the discrete entropy are defined respectively as
\begin{align}\label{entropy}\mathcal W (\rho ) = \frac12 \sum_{(i,j)\in E} \mathbb W_{ij}\rho_i\rho_j,\quad \mathcal L(\rho)=\sum^n
_{i=1}(\log(\rho_i)\rho_i - \rho_i).
\end{align} 
The discrete Fisher information on the graph  is given by \begin{align}\label{Fisher}
 \mathcal I(\rho ) = \frac12 
\sum_{i=1}^n
\sum_{j\in N(i)}
\widetilde{\omega}_{ij}\,|\log(\rho_i) - \log(\rho_j)|^2\,\tilde g_{ij}(\rho),
\end{align}
where $(\widetilde\omega,\tilde g)$ denotes another pair of edge weight and probability weight on $G$.  
Here, we take $\tilde g_{ij}(\rho):=\frac{\rho_i-\rho_j}{\log \rho_i-\log \rho_j}$, and assume $\omega_{ij}=\widetilde\omega_{ij}$ for simplicity.  Note that the Fisher information acts as a repelling force, penalizing the energy when $\rho$ approaches the boundary of $\mathcal P(G)$.

\subsection{Stochastic Schr\"odinger equation with polynomial nonlinearity on a graph}\label{example-NLS} 
In the case of the Euclidean domain, the SSE with polynomial nonlinearity was initially studied by \cite{SSE1,SSE3}, establishing it as a prototypical model for nonlinear waves in nonlinear optics and plasma physics. With the growing importance of graph-structured domains, this model was subsequently extended to graphs  \cite{ZhouJFA,Cui23}. The resulting SSE on graphs with polynomial nonlinearity serves as a fundamental framework for investigating phenomena at the intersection of graph topology, nonlinear interactions, and stochastic perturbations, with applications to the dynamics of complex networks.

To formulate the SSE with polynomial nonlinearity on a graph, we take the  dominant energy as 
\begin{align}\label{H_0-example}
\mathcal H^{\mathbb V}_0(\rho ,S) &:= \mathcal H_0(\rho,x)+\mathcal V (\rho )\text{ with }\mathcal H_0(\rho,x):=K(\rho,x ) + \tfrac18 \mathcal I(\rho ) + \mathcal W (\rho ),
\end{align}
and $\mathcal H_1$ is given by \eqref{def-H1}. Plugging $\mathcal H^{\mathbb V}_0,\mathcal H_1$ into \eqref{WHflow1}  and \eqref{WHflow2} yields  
\begin{align*}
&\mathrm d\rho_i+ \sum_{j\in N(i)} \omega_{ij}(S_j-S_i)g_{ij}(\rho)\,\mathrm dt=0,\\
&\mathrm dS_i+\frac12 \sum_{j\in N(i)} \omega_{ij}(S_i - S_j)^2\frac{ \partial g_{ij}(\rho )}{ \partial \rho_i}\,\mathrm dt+\frac18 \frac{\partial \mathcal I(\rho)}{\partial \rho_i}\,\mathrm dt+\mathbb V_i\,\mathrm dt
+\sum_{j=1}^n\mathbb W_{ij}\rho_j\,\mathrm dt + \sigma_i\,\mathrm dW_i(t)= 0,
\end{align*} 
where $i\in V,$ and $W_1(t),\ldots,W_n(t),t\ge 0$ are real-valued independent standard Brownian motions. 
Via the Madelung transformation $u_j(t) = \sqrt{\rho_j (t)}e^{\mathbf iS_j(t)},\ j\in V$, we obtain the complex SSE with polynomial nonlinearity on the graph $G$ \cite{ZhouJFA,Cui23}:
\begin{align}\label{com_NLS}
\mathbf i\mathrm du_j= \Big( -\tfrac12(\Delta_Gu)_j + u_j\mathbb V_j + u_j \sum_{l=1}^n\mathbb W_{jl}| u_l|^2\Big)\,\mathrm dt 
+ \sigma_j u_j \circ \mathrm dW_j(t).
\end{align} 
Here $\Delta_Gu$ denotes the nonlinear  Laplacian on the graph \cite{ZhouJFA}, defined  by
\begin{align}\label{nonlinear_lap}
(\Delta_Gu)_j &= -\frac{ u_j}{| u_j|^2} \Big( \sum_{l\in N(j)}\omega_{jl}( \log(u_j)-\log(u_l))g_{jl}
+ \sum_{l\in N(j)}\widetilde {\omega}_{jl}\tilde{g}_{jl}\Re (\log(u_j)-\log(u_l))\Big) \notag\\
&\quad -u_j\Big(\sum_{l\in N(j)}\omega_{jl} \frac{\partial g_{jl}}{\partial \rho_j}|\log(u_j)-\log(u_l)|^2+
 \sum_{l\in N(j)}\widetilde {\omega}_{jl} \frac{\partial \tilde g_{jl}}{\partial \rho_j}|\Re (\log(u_j)-\log(u_l))|^2\Big),
\end{align}
where $j\in V$.

For a test function $\varphi\in\mathcal C^1(\mathcal P^{\circ}(G)),$ we note that when $\partial_{\rho}\varphi$ is continuous, by definition of $\mathcal H_0$ \eqref{H_0-example}, the integration by parts formula \eqref{IBP},  and the relation between $\nabla_{\mathcal W}$ and $\partial_{\rho}$ \eqref{relation-derivative}, we have 
\begin{align}\label{relation22}
&\langle \partial_{\rho}\varphi,D_x\mathcal H_0\rangle=\Big\langle\partial_{\rho}\varphi,\Big(\sum_{j\in N(i)}\omega_{ij}(x_i-x_j)g_{ij}(\rho)\Big)_{i\in V}\Big\rangle=\langle \nabla_{\mathcal W}\varphi,\nabla_Gx\rangle_{\rho},
\end{align}
where $\langle\cdot,\cdot\rangle_{\rho}$ is the inner product on the space $\mathbb H_{\rho}$ (see \eqref{g1} for its definition).  
This shows that, for the SSE \eqref{com_NLS} on a graph, the corresponding HJB equation involves both the Wasserstein and Euclidean geometries.

Applying Theorems \ref{thm_existence} and \ref{uniqueness} to the SSE \eqref{com_NLS}, we obtain the following results for the HJB equation (24) associated with the corresponding optimal control problem.

\begin{corollary}\label{coro-1} 
Let Assumptions \ref{assumption_g}  and  \ref{ass1} hold. In addition assume that  $F$ and $h$ are uniformly bounded.  Then there exists a unique bounded viscosity solution $U$ for the HJB equation \eqref{HJB} with $\mathcal H_0$ given by \eqref{H_0-example}.  
\end{corollary}
\begin{proof} According to conditions of Theorems \ref{thm_existence} and \ref{uniqueness}, it suffices to verify that the function $\mathcal H_0$ defined in \eqref{H_0-example} satisfies Assumption \ref{energy}, Assumption  \ref{ass_domi}, and condition \eqref{bound_prop}.  

We first verify Assumption \ref{energy}, i.e., 
 $\mathcal H_0 \in (\mathrm{Pr})$, and both $\mathcal H_0$ and its partial derivatives are locally Lipschitz continuous in the sense 
given by \eqref{locally-Lip}.     
In fact, when $\mathcal H_0(\rho,x)\leq R,$ we derive that the Fisher information satisfies  $\mathcal I(\rho)\leq \mathcal H_0(\rho, x)\leq R$. This implies there exists a constant $c\in(0,1)$ depending on $R$ such that \begin{align*}
\min_{i\in V}\rho_i\ge {c},\quad \text{namely, }\rho\in\mathcal P_{c}(G).
\end{align*}
Notice that the kinetic energy satisfies $K(\rho, x)\leq \mathcal H_0( \rho, x)\leq R$, and by Assumption \ref{assumption_g} (g-\romannumeral2), 
\begin{align*}
\frac12\sum_{(i,j)\in E}\omega_{ij}(x_i-x_j)^2(\rho_i\wedge \rho_j) \leq \frac12\sum_{(i,j)\in E}\omega_{ij}(x_i-x_j)^2g_{ij}(\rho) .
\end{align*}
This yields that there exists $C:=C(R)>0$ such that 
\begin{align*}
\max_{(i,j)\in E}| x_i- x_j|\leq \sqrt{\frac{R}{\min\limits_{(i,j)\in E}\omega_{ij}\min\limits_{i\in V}\rho_i}}\leq C. 
\end{align*}
Then the locally Lipschitz continuity can be also obtained based on the definition of $\mathcal H_0$ in \eqref{H_0-example}. 

Next, we verify Assumption \ref{ass_domi}. By the symmetry of the metric tensor $g$ and the weight matrix $\omega,$ we have that $$\sum_{i=1}^n(D_x\mathcal H_0(\rho,x))_i=\sum_{(i,j)\in E}(x_i-x_j)\omega_{ij}g_{ij}(\rho)=0,$$ which means $D_x\mathcal H_0(\rho,x)\in T_{\rho}\mathcal P(G).$  
Utilizing the Cauchy--Schwartz inequality, 
\begin{align}\label{coro-pro}
\|D_x\mathcal H_0(\rho,x)\|^2=\sum_{i\in V}(\sum_{j\in V(i)}|x_i-x_j|g_{ij}(\rho))^2\leq C(n)K(\rho,x)\leq C(n)(\mathcal H_0(\rho,x)+1), 
\end{align} where we use the boundedness of $\mathbb W_{ij}$, and  we recall that $K$ is the kinetic energy \eqref{kinetic}. By virtue of \eqref{coro-pro} and $\|D^2_x\mathcal H_0(\rho,x)\|\leq C(n)$, we deduce  that \eqref{cond-unique1} holds as well. This verifies Assumption \ref{ass_domi}.

The condition \eqref{bound_prop} is guaranteed by the fact $D_x\mathcal H^{\mathbb V}_0=D_x\mathcal H_0,$ the boundedness of the linear potential control $\mathcal V,$ \eqref{coro-pro}, and $\|D^2_x\mathcal H^{\mathbb V}_0(\rho,x)\|= \|D^2_x\mathcal H_0(\rho,x)\|\leq C(n).$ This
finishes the proof.  
\end{proof}

\subsection{Stochastic logarithmic Schr\"odinger equation on a graph} 
It is worth noting that the logarithmic potential is the unique nonlinearity consistent with the separability hypothesis for noninteracting subsystems in Schr\"odinger theory \cite{SLSE}. However, unlike the case of polynomial nonlinearity, the logarithmic nonlinearity in the stochastic logarithmic Schr\"odinger equation is not locally Lipschitz. To address the challenges posed by the vacuum and the nonlocal Lipschitz property of the logarithmic potential, \cite{SLSE, SLSE3} developed regularization procedures to establish the well-posedness of the stochastic logarithmic Schr\"odinger equation in Euclidean domains. In the context of graphs, the stochastic logarithmic Schr\"odinger equation on a finite graph was recently studied in \cite{cuiSIAM}.  

To formulate the logarithmic nonlinearity counterpart on graphs, define
\begin{align}\label{H0-example2}
\mathcal H^{\mathbb V}_0(\rho ,S) &:= \mathcal H_0(\rho,x)+\mathcal V (\rho )\text{ with }\mathcal H_0(\rho,x):=K(\rho,x ) + \tfrac18 \mathcal I(\rho ) -\mathcal L(\rho),
\end{align}  
and $\mathcal H_1$ is given by \eqref{def-H1}. Plugging $\mathcal H^{\mathbb V}_0,\mathcal H_1$ into \eqref{WHflow1}  and \eqref{WHflow2} yields  
\begin{align*}
&\mathrm d\rho_i+ \sum_{j\in N(i)} \omega_{ij}(S_j-S_i)g_{ij}(\rho)\,\mathrm dt=0,\\
&\mathrm dS_i+\frac12 \sum_{j\in N(i)} \omega_{ij}(S_i - S_j)^2\frac{ \partial g_{ij}(\rho )}{ \partial \rho_i}\mathrm dt+\frac18 \frac{\partial \mathcal I(\rho)}{\partial \rho_i}\mathrm dt+\mathbb V_i\mathrm dt
 -\log(\rho_i)\mathrm dt+ \sigma_i\mathrm dW_i(t)= 0.
\end{align*} 
Via the Madelung transformation $u_j(t) = \sqrt{\rho_j (t)}e^{\mathbf iS_j(t)},\ j\in V$, we obtain the stochastic logarithmic Schr\"odinger equation on the graph $G$ \cite{cuiSIAM}:
\begin{align*}
\mathbf i\mathrm du_j= \Big( -\tfrac12(\Delta_Gu)_j + u_j\mathbb V_j -u_j\log(|u_j|^2)\Big)\,\mathrm dt 
+ \sigma_j u_j \circ \mathrm dW_j(t).
\end{align*} 

Applying Theorems \ref{thm_existence} and \ref{uniqueness} to the above equation, we obtain the following results for the HJB equation (24) associated with the corresponding optimal control problem.

\begin{corollary}\label{coro-2} 
Let Assumptions \ref{assumption_g} and \ref{ass1} hold. In addition assume that  $F$ and $h$ are uniformly bounded. Then there exists a unique bounded viscosity solution $U$ for the HJB equation \eqref{HJB} with $\mathcal H_0$ given by \eqref{H0-example2}. 
\end{corollary}
Note that the entropy $\mathcal L$ defined in \eqref{entropy} is bounded. Hence the proof of Corollary \ref{coro-2} is similar to that of Corollary \ref{coro-1} and thus is omitted.

At the end of this section, we remark that the kinetic energy \eqref{kinetic} and the Fisher information \eqref{Fisher} play a crucial role in verifying Assumptions \ref{energy} and \ref{ass_domi}. It is also possible to study the HJB equation and the associated optimal control problems for SWHS under alternative choices of energy. For example, beyond the quadratic form of the kinetic energy \eqref{kinetic}, one may consider the general even-power polynomial case
\[
K_1(\rho,x)=\sum_{(i,j)\in E}\omega_{ij}(x_i-x_j)^{2m}g_{ij}(\rho), \quad m\in\mathbb N_+.
\]
Similarly, in place of the Fisher information \eqref{Fisher}, one may employ terms with a repeller property, such as the negative-power function 
\[
\mathcal I_1(\rho)=\sum_{i=1}^n\rho_i^{-1}.
\]

\section{Proof of Theorem \ref{thm_existence}}\label{sec_33} In this section, we apply the dynamic programming principle method to prove that the value function $U$ is the viscosity solution of the HJB equation \eqref{HJB}.  
We begin by establishing several fundamental properties of the value function $U$, including the Bellman principle of optimality, growth bounds, and continuity with respect to all variables. These properties are crucial for verifying that $U$ satisfies the viscosity solution property of the HJB equation. Next, we introduce the stopping time technique, which, combined with the Bellman principle and It\^o formula, allows us to show that $U$ fulfills both the viscosity subsolution and supersolution conditions of the HJB equation.

\subsection{Properties of the value function}
 We first prove the  Bellman principle for the related control problem, which is the key to verify the viscosity properties of the value function.   

\begin{proposition}\label{dynamic-pro} 
Let Assumptions \ref{assumption_g}, \ref{energy} and \ref{ass1} hold. 
Then for 
$0\leq t<\bar t\leq T,$  the value function defined in \eqref{def-valuefunction} fulfills \begin{align*} 
U(t,\rho,x) = \inf_{\mathbb V(\cdot)\in\mathscr V_{\ell}[t,T ]} \mathbb E_{t} \Big[ \int_{t}^{\bar t} F(s,\rho(s),S(s),\mathbb V(s))  ds + U(\bar t,\rho(\bar t),S(\bar t)) \Big]. 
\end{align*}  
\end{proposition}
Since its proof is standard, we postpone it to Appendix \ref{app1} to make the article self-contained. We refer to e.g. \cite{Nisio} and \cite[Section V.2]{book-controlled} for further details.  
With this dynamic programming principle, we first show the continuity and growth properties of the value function $U$, and then prove that $U$ is a viscosity solution of the HJB equation \eqref{HJB}.

\begin{proposition}\label{conti}
Under Assumptions \ref{assumption_g}, \ref{energy} and    \ref{ass1}, the value function $U$ is continuous on $[0,T]\times\mathcal P^{\circ}(G)\times\mathbb R^n$ and satisfies that for any $t\in[0,T],$
\begin{align}\label{growth-U}
 |U(t,\rho,x)| \leq C(1+ |\mathcal H_0(\rho,x)|^{p_1}),
\end{align} for some $C := C(T,p_1,\ell)>0$, where the parameter $p_1$ is given in Assumption \ref{ass1}. 
\end{proposition}

\begin{proof}
The proof is divided into three parts: we first establish the growth property \eqref{growth-U}, then demonstrate the continuity of $U$ with respect to the variables $(\rho, x)$, and finally show the continuity in the variable $t$.

\textbf{Proof of \eqref{growth-U}.}
By the definition of $\mathcal J$ in \eqref{def-costfunction} and assumption \eqref{ass11}, 
\begin{align}\label{growth-proof}
\mathcal J(t,\rho,x;\mathbb V(\cdot)) &\leq C\mathbb E_t \Big[1+ \int_t^T  |\mathcal H_0(\rho(s),S(s))|^{p_1}  \mathrm ds + |\mathcal H_0(\rho(T),S(T))|^{p_1} \Big] \notag\\
&\leq C(\ell,T,p_1)(1 + |\mathcal H_0(\rho,x)|^{p_1}),
\end{align}
where the last inequality follows from Proposition \ref{exact}, the initial value $\rho(t)=\rho,S(t)=x,$  \eqref{def-H0} and 
$\mathbb V(s)\in B_{\ell}$ for all $s\in[t,T]$.  Taking the infimum over $\mathbb V(\cdot)\in\mathscr V_{\ell}[t,T]$ yields \eqref{growth-U}.

\textbf{ Continuity of $U$ w.r.t. $(\rho,x)$ variables.}  
Let $(\rho,x),(\tilde\rho,\tilde x)\in\mathcal P^{\circ}(G)\times\mathbb R^n$ be two distinct initial values of the SWHS \eqref{WHflow1}--\eqref{WHflow2} at the starting time $t$.  
Let $\{\mathbb V_k\},\{\widetilde{\mathbb V}_k\} \subset \mathscr V_{\ell}[t,T]$ be minimizing sequences satisfying 
\[
\mathcal J(t,\rho,x;\mathbb V_k) \to U(t,\rho,x), \quad \mathcal J(t,\tilde\rho,\tilde x;\widetilde{\mathbb V}_k) \to U(t,\tilde \rho,\tilde x) \quad \text{a.s. as } k \to \infty.
\] 
Specifically, one has that for some bounded sequences $\delta_k,\tilde{\delta}_k \to 0$ a.s., 
\[
U(t,\rho,x) \geq \mathcal J(t,\rho,x;\mathbb V_k) - \delta_k, \quad U(t,\tilde \rho,\tilde x) \geq \mathcal J(t,\tilde \rho,\tilde x;\widetilde{\mathbb V}_k) - \tilde\delta_k,\;a.s.
\]
This implies 
\begin{align}\label{conti_ineq}
\mathcal J(t,\rho,x;\mathbb V_k) - \mathcal J(t,\tilde\rho,\tilde x;{\mathbb V}_k) - \delta_k &\leq U(t,\rho,x) - U(t,\tilde\rho,\tilde x) \notag\\
&\leq \mathcal J(t,\rho,x;\widetilde{\mathbb V}_k) - \mathcal J(t,\tilde\rho,\tilde x;\widetilde{\mathbb V}_k) + \tilde\delta_k,\;a.s.
\end{align}  
where we recall $U(t,\rho,x)=\inf\limits_{\mathbb V\in\mathscr V_{\ell}[t,T]}\mathcal J(t,\rho,x;\mathbb V).$ 
Once we prove that 
\begin{align}\label{step12-prove}
\sup_{\mathbb V\in\mathscr V_{\ell}[t,T]} \big| \mathcal J(t,\rho,x;\mathbb V) - \mathcal J(t,\tilde\rho,\tilde x;\mathbb V) \big| \to 0 \quad \text{as } (\rho,x) \to (\tilde\rho,\tilde x),
\end{align}
then by \eqref{conti_ineq} we get \begin{align}\label{absolut}
|U(t,\rho,x)-U(t,\tilde\rho,\tilde x)|\leq \sup_{\mathbb V\in\mathscr V_{\ell}[t,T]} \big| \mathcal J(t,\rho,x;\mathbb V) - \mathcal J(t,\tilde\rho,\tilde x;\mathbb V) \big|+(\delta_k\vee\tilde\delta_k)\to0
\end{align} 
as $k\to\infty$ and $(\rho,x)\to(\tilde\rho,\tilde x).$

Now we prove \eqref{step12-prove}. Since $t\ge 0$ is fixed, we may assume $t=0$ for simplicity. According to definitions of $\mathcal V$ and $\mathcal H_0$ in  \eqref{def-H0},   and Proposition \ref{exact} (with $t_0=0$), the solution of the SWHS satisfies that for any $p\ge 2,$
\begin{align*}
&\quad \mathbb E\Big[\sup_{t\in[0,T]}|\mathcal H_0(\rho^{\mathbb V}(t),S^{\mathbb V}(t))|^p\Big]\\&\leq C_p\mathbb E\Big[\sup_{t\in[0,T]}|\mathcal H^{\mathbb V}_0( \rho^{\mathbb V}(t),S^{\mathbb V}(t))|^p\Big]+{C_p\mathbb E\Big[\sup_{t\in[0,T]}\|\mathbb V(t)\|^p\Big]}\\
&\leq C(\ell,p,T)(1+|\mathcal H^{\mathbb V}_0(\rho,x)|^p)\leq C(\ell,p,T)(1+|\mathcal H_0(\rho,x)|^p),
\end{align*}   which implies that for any $p\ge 2,$
\begin{align}
&\sup_{\mathbb V\in \mathscr V_{\ell}[0,T]}\mathbb E \Big[ \sup_{t\in[0,T]} |\mathcal H_0(\rho^{\mathbb V}(t),S^{\mathbb V}(t))|^p \Big] \leq C. \label{mass2}
\end{align}

For $\mathbb V\in\mathscr V_{\ell}[0,T]$ and positive constant $\alpha> \mathcal H_0(\rho,x)\vee\mathcal H_0(\tilde\rho,\tilde x)$, define the stopping time 
\[
\tau_{\alpha} := \inf \Big\{ t \in [0,T] : \sup_{s\in[0,t]} \mathcal H_0(\rho^{\mathbb V}(s),S^{\mathbb V}(s)) \geq \alpha \Big\},
\] 
and the event 
\[
\Omega_{\alpha} := \Big\{ \omega \in \Omega : \sup_{s\in[0,T]} \mathcal H_0(\rho^{\mathbb V}(s,\omega),S^{\mathbb V}(s,\omega)) \leq \alpha \Big\},
\]
for solution $(\rho^{\mathbb V}(t),S^{\mathbb V}(t))_{t\in[0,T]}$. 
Similarly define $\tilde\tau_\alpha$ and $\widetilde\Omega_\alpha$ for solution $(\tilde\rho^{\mathbb V}(t),\widetilde S^{\mathbb V}(t) )_{t\in[0,T]}$ with the initial value $(\tilde \rho,\tilde x)$. Now we split the proof of \eqref{step12-prove} into the following two steps.

 \textit{Step 1: Prove that \begin{align}\label{step1-prove}\sup\limits_{\mathbb V\in\mathscr V_{\ell}[0,T]}\int_0^T\mathbb E[\mathbf 1_{\Omega^c_{\alpha}\cup \widetilde \Omega^c_{\alpha}}(F(s,\rho^{\mathbb V}(s),S^{\mathbb V}(s),\mathbb V)-F(s,\tilde \rho^{\mathbb V}(s),\widetilde S^{\mathbb V}(s),\mathbb V))]\mathrm ds\to0 \text{ as } \alpha\to\infty.
\end{align}}

By the H\"older inequality, growth assumption  \eqref{ass11}  on $F$ and  \eqref{mass2},  we obtain 
\begin{align}\label{first1}
&\quad \sup\limits_{\mathbb V\in\mathscr V_{\ell}[0,T]}\int_0^T\mathbb E[\mathbf 1_{\Omega^c_{\alpha}\cup \widetilde \Omega^c_{\alpha}}(F(s,\rho^{\mathbb V}(s),S^{\mathbb V}(s),\mathbb V)-F(s,\tilde \rho^{\mathbb V}(s),\widetilde S^{\mathbb V}(s),\mathbb V))]\mathrm ds\notag\\
&\leq \sup\limits_{\mathbb V\in\mathscr V_{\ell}[0,T]}\int_0^T\Big\|F(s,\rho^{\mathbb V}(s),S^{\mathbb V}(s),\mathbb V)-F(s,\tilde \rho^{\mathbb V}(s),\widetilde S^{\mathbb V}(s),\mathbb V)\Big\|_{L^2(\Omega)}(\mathbb P(\Omega^c_{\alpha}\cup \widetilde \Omega^c_{\alpha}))^{\frac12} \mathrm ds\notag\\
&\leq \sup\limits_{\mathbb V\in\mathscr V_{\ell}[0,T]}C(\mathbb P(\Omega^c_{\alpha}\cup \widetilde \Omega^c_{\alpha}))^{\frac12}\leq \sup\limits_{\mathbb V\in\mathscr V_{\ell}[0,T]}C(\mathbb P(\Omega^c_{\alpha})+\mathbb P(\widetilde \Omega^c_{\alpha}))^{\frac12}\notag\\
&\leq \sup\limits_{\mathbb V\in\mathscr V_{\ell}[0,T]} C\Big(\mathbb P(\sup_{t\in[0,T]}\mathcal H_0(\rho^{\mathbb V}(t),S^{\mathbb V}(t))\ge \alpha)+ \mathbb P(\sup_{t\in[0,T]}\mathcal H_0(\tilde\rho^{\mathbb V}(t),\widetilde S^{\mathbb V}(t))\ge \alpha)\Big)^{\frac12},
\end{align}
where constant $C:=C(\ell,p_1,T,\mathcal H_0(\rho,x),\mathcal H_0(\tilde\rho,\tilde x))>0.$
Using the Chebyshev inequality and \eqref{mass2} leads to that the right-hand side of \eqref{first1} converges to zero as $\alpha\to\infty.$

\textit{Step 2: Prove that  
\begin{align}\label{step2-prove}\sup\limits_{\mathbb V\in\mathscr V_{\ell}[0,T]}\int_0^T\mathbb E[\mathbf 1_{\Omega_{\alpha}\cap \widetilde \Omega_{\alpha}}(F(s,\rho^{\mathbb V}(s),S^{\mathbb V}(s),\mathbb V)-F(s,\tilde \rho^{\mathbb V}(s),\widetilde S^{\mathbb V}(s),\mathbb V))]\mathrm ds\to 0\text{ as }(\rho,x)\to(\tilde\rho,\tilde x).
\end{align}}

By the locally Lipschitz condition of $F$ (Assumption \ref{ass1}), the H\"older inequality and \eqref{mass2}, we have  
\begin{align*}
&\sup\limits_{\mathbb V\in\mathscr V_{\ell}[0,T]}\int_0^T\mathbb E[\mathbf 1_{\Omega_{\alpha}\cap \widetilde \Omega_{\alpha}}(F(s,\rho^{\mathbb V}(s),S^{\mathbb V}(s),\mathbb V)-F(s,\tilde \rho^{\mathbb V}(s),\widetilde S^{\mathbb V}(s),\mathbb V))]\mathrm ds\\
&\leq C(\alpha)\sup\limits_{\mathbb V\in\mathscr V_{\ell}[0,T]}\int_0^T\mathbb E\Big[\mathbf 1_{\Omega_{\alpha}\cap \widetilde \Omega_{\alpha}}(\|S^{\mathbb V}(s)-\widetilde S^{\mathbb V}(s)\|+\|\rho^{\mathbb V}(s)-\tilde \rho^{\mathbb V}(s)\|)\Big] \mathrm ds\\
&\leq C(\alpha,T)\sup\limits_{\mathbb V\in\mathscr V_{\ell}[0,T]}\Big\{\mathbb E\Big[\int_0^T\mathbf 1_{\Omega_{\alpha}\cap \widetilde \Omega_{\alpha}}(\|S^{\mathbb V}(s)-\widetilde S^{\mathbb V}(s)\|^2+\|\rho^{\mathbb V}(s)-\tilde \rho^{\mathbb V}(s)\|^2)\mathrm ds\Big]\Big\}^{\frac12}\\
&\leq C(\alpha,T)\sup\limits_{\mathbb V\in\mathscr V_{\ell}[0,T]}\Big\{\mathbb E\Big[\int_0^{\tau_{\alpha}\wedge\tilde{\tau}_{\alpha}}(\|S^{\mathbb V}(s)-\widetilde S^{\mathbb V}(s)\|^2+\|\rho^{\mathbb V}(s)-\tilde \rho^{\mathbb V}(s)\|^2)\mathrm ds\Big]\Big\}^{\frac12}, 
\end{align*} 
where in the last step 
we use the fact that $\tau_{\alpha}=T$ on $\Omega_{\alpha},$ and similarly, $\tilde\tau_{\alpha}=T$ on $\widetilde\Omega_{\alpha}.$   
According to the differential formulation \eqref{WHflow1}--\eqref{WHflow2}, utilizing the It\^o formula before time $\tau_{\alpha}\wedge \tilde\tau_{\alpha}$ and the Young inequality, we obtain  
\begin{align*}
&\quad \|{\rho^{\mathbb V}(t\wedge \tau_{\alpha}\wedge \tilde \tau_{\alpha})}-{\tilde\rho^{\mathbb V}(t\wedge \tau_{\alpha}\wedge \tilde \tau_{\alpha})}\|^2\\
&=\|{\rho(0)}-{\tilde\rho(0)}\|^2\\
&\quad +
\int_0^{t\wedge \tau_{\alpha}\wedge \tilde \tau_{\alpha}}
2(\rho^{\mathbb V}(s)-\tilde\rho^{\mathbb V}(s))^{\top}\Big(D_{ S}\mathcal H^{\mathbb V}_0(\rho^{\mathbb V}(s),S^{\mathbb V}(s))-D_{ S}\mathcal H^{\mathbb V}_0(\tilde\rho^{\mathbb V}(s),\widetilde S^{\mathbb V}(s))\Big)\mathrm ds\\
&\leq \|{\rho(0)}-{\tilde\rho(0)}\|^2+C(\alpha)\int_0^{t\wedge \tau_{\alpha}\wedge \tilde \tau_{\alpha}}(\|S^{\mathbb V}(s)-\widetilde S^{\mathbb V}(s)\|^2
+\| {\rho^{\mathbb V}(s)}-{\tilde\rho^{\mathbb V}(s)}\|^2)\mathrm ds,
\end{align*}  
and 
\begin{align*}
&\quad \|S^{\mathbb V}(t\wedge \tau_{\alpha}\wedge \tilde \tau_{\alpha})-\widetilde S^{\mathbb V}(t\wedge \tau_{\alpha}\wedge \tilde \tau_{\alpha})\|^2\\&=\|S(0)-\widetilde S(0)\|^2\\
&\quad +2\int_0^{t\wedge \tau_{\alpha}\wedge \tilde \tau_{\alpha}}(S^{\mathbb V}(s)-\widetilde S^{\mathbb V}(s))^{\top}\Big(-D_{\rho}\mathcal H^{\mathbb V}_0(\rho^{\mathbb V}(s),S^{\mathbb V}(s))+D_{\rho}\mathcal H^{\mathbb V}_0(\tilde\rho^{\mathbb V}(s),\widetilde S^{\mathbb V}(s))\Big)\mathrm ds\\
 &\leq \|x-\tilde x\|^2+C(\alpha)\int_0^{t\wedge \tau_{\alpha}\wedge \tilde \tau_{\alpha}}(\|S^{\mathbb V}(s)-\widetilde S^{\mathbb V}(s)\|^2+\|{\rho^{\mathbb V}(s)}-{\tilde\rho^{\mathbb V}(s)}\|^2)\mathrm ds,
\end{align*}
where we use the locally Lipschitz condition on $\mathcal H_0^{\mathbb V}$ and its derivatives, and we also use the property $\mathcal H_0\in(\mathrm{Pr})$ (see \eqref{Pr-set} for the definition) to bound $$\mathbf 1_{s\in[0,\tau_{\alpha}\wedge \tilde\tau_{\alpha}]}\Big(\max_{(i,j)\in E}|S_i(s)-S_j(s)|+\max_{(i,j)\in E}|\widetilde S_i(s)-\widetilde S_j(s)|+\frac{1}{\min_{i\in V}\rho_i(s)}+\frac{1}{\min_{i\in V}\tilde \rho_i(s)}\Big)\leq C(\alpha).$$
Gathering together the above two inequalities and applying the Gronwall inequality give that for each fixed $\alpha,$
\begin{align*}
&\quad \sup_{s\in[0,T]}\mathbb E\Big[ \|{\rho^{\mathbb V}(s\wedge \tau_{\alpha}\wedge \tilde \tau_{\alpha})}-{\tilde\rho^{\mathbb V}(s\wedge \tau_{\alpha}\wedge \tilde \tau_{\alpha})}\|^2+\|S^{\mathbb V}(s\wedge \tau_{\alpha}\wedge \tilde \tau_{\alpha})-\widetilde S^{\mathbb V}(s\wedge \tau_{\alpha}\wedge \tilde \tau_{\alpha})\|^2\Big]\\&\leq (\|{\rho}-{\tilde\rho}\|^2+|x-\tilde x\|^2)e^{C(\alpha,T)}.
\end{align*}
Hence, 
\begin{align*}
&\sup\limits_{\mathbb V\in\mathscr V_{\ell}[0,T]}\Big\{\mathbb E\Big[\int_0^{\tau_{\alpha}\wedge\tilde{\tau}_{\alpha}}(\|S^{\mathbb V}(s)-\widetilde S^{\mathbb V}(s)\|^2+\|\rho^{\mathbb V}(s)-\tilde \rho^{\mathbb V}(s)\|^2)\mathrm ds\Big]\Big\}^{\frac12}\\
&\leq \sup\limits_{\mathbb V\in\mathscr V_{\ell}[0,T]}\Big\{\mathbb E\Big[\int_0^{T}(\|S^{\mathbb V}(s\wedge \tau_{\alpha}\wedge \tilde\tau_{\alpha})-\widetilde S^{\mathbb V}(s\wedge \tau_{\alpha}\wedge \tilde\tau_{\alpha})\|^2\\
&\quad +\|\rho^{\mathbb V}(s\wedge \tau_{\alpha}\wedge \tilde\tau_{\alpha})-\tilde \rho^{\mathbb V}(s\wedge \tau_{\alpha}\wedge \tilde\tau_{\alpha})\|^2)\mathrm ds\Big]\Big\}^{\frac12}
\\
&\leq T^{\frac12}(\|{\rho}-{\tilde\rho}\|^2+\|x-\tilde x\|^2)^{\frac12}e^{\frac12C(\alpha,T)},
\end{align*}which converges to zero 
as $(\rho,x)\to(\tilde\rho,\tilde x).$ This finishes the proof of \textit{Step 2}. 

Combining \eqref{step1-prove} and \eqref{step2-prove}, we derive 
\begin{align*}
\sup\limits_{\mathbb V\in\mathscr V_{\ell}[0,T]}\int_0^T\mathbb E[(F(s,\rho^{\mathbb V}(s),S^{\mathbb V}(s),\mathbb V)-F(s,\tilde \rho^{\mathbb V}(s),\widetilde S^{\mathbb V}(s),\mathbb V))]\mathrm ds\to 0
\end{align*} 
$\text{ as }(\rho,x)\to(\tilde\rho,\tilde x),\;\alpha\to\infty.$
Similarly, one can prove that $$\sup\limits_{\mathbb V\in\mathscr V_{\ell}[0,T]}\mathbb E[h(\rho^{\mathbb V}(T),S^{\mathbb V}(T))-h(\tilde\rho^{\mathbb V}(T),\widetilde S^{\mathbb V}(T))]\to0\text{ as }(\rho,x)\to(\tilde\rho,\tilde x).$$ Thus we obtain \eqref{step12-prove}. Then by \eqref{absolut} 
we finish the proof of continuity of $U$ on $(\rho,x)$ variables.  

\textbf{Continuity of $U$ w.r.t. $t$ variable.} Suppose that $t<\bar t.$ By Proposition \ref{dynamic-pro}, for any $\mathbb V\in\mathscr V_{\ell}[t,T],$   
\begin{align*}
&U(t,\rho,x)-\mathbb E_t[U(\bar t,\rho,x)]\leq \mathbb E_{t}\Big[U(\bar t,\rho(\bar t),S(\bar t))-U(\bar t,\rho,x)+
\int_{t}^{\bar t}F(s,\rho(s),S(s);\mathbb V)\mathrm ds\Big]\\
&\leq \mathbb E_t[|U(\bar t,\rho(\bar t),S(\bar t))-U(\bar t,\rho,x)|]+C(\bar t-t)\Big(1+\mathbb E\Big[\sup_{s\in[t,T]}|\mathcal H_0(\rho(s),S(s))|^{p_1}\Big]\Big),
\end{align*}
 where we use Assumption \ref{ass1}. By Proposition \ref{exact}, we have that the second term on the right-hand side of the above inequality converges to zero as $\bar t\to t.$ By virtue of \eqref{growth-U} and Proposition \ref{exact}, we have \begin{align*}
 &\quad \mathbb E_t[U(\bar t,\rho(\bar t),S(\bar t))]\leq \mathbb E_t\Big[\sup_{s\in[t,T]}U(s,\rho(s),S(s))\Big]\\&\leq C\mathbb E_t\Big[\sup_{s\in[t,T]}(1+|\mathcal H_0(\rho(s),S(s))|^{p_1})\Big]\leq C(1+|\mathcal H_0(\rho(t),S(t))|^{p_1}).
 \end{align*} 
  Hence by the dominant convergence theorem and the continuity of $U$ on $(\rho,x)$ variables yields that \begin{align}\label{DCT}\lim_{\bar t\to t}\mathbb E_t\Big[\Big|U(\bar t,\rho(\bar t),S(\bar t))-U(\bar t,\rho,x)\Big|\Big]= \mathbb E_t\Big[\lim_{\bar t\to t}\Big|U(\bar t,\rho(\bar t),S(\bar t))-U(\bar t,\rho,x)\Big|\Big]=0.
  \end{align} 
 
 On the other hand, according to Proposition \ref{dynamic-pro}, for any $\epsilon>0$, there exists $\mathbb V^{\epsilon}\in\mathscr V_{\ell}[t,T]$ such that 
 \begin{align*}
 &\quad U(t,\rho,x)-\mathbb E_t[U(\bar t,\rho,x)]+\epsilon\\
 &\ge \mathbb E_t[U(\bar t,\rho^{\mathbb V^\epsilon}(\bar t),S^{\mathbb V^\epsilon}(\bar t))-U(\bar t,\rho,x)]+\mathbb E_t\Big[\int_t^{\bar t}F(s,\rho^{\mathbb V^\epsilon}(s),S^{\mathbb V^\epsilon}(s);\mathbb V^{\epsilon})\mathrm ds\Big].
 \end{align*} 
Similar to the argument of \eqref{DCT}, by virtue of the dominant convergence theorem, $\rho^{\mathbb V^\epsilon}(t)=\rho,S^{\mathbb V^\epsilon}(t)=x$, \eqref{growth-U} and Proposition \ref{exact}, we get $$\mathbb E_t[U(\bar t,\rho^{\mathbb V^\epsilon}(\bar t),S^{\mathbb V^\epsilon}(\bar t))-U(\bar t,\rho,x)]\to0 \text{ as }\bar t\to t,\text{ uniformly with } \epsilon.$$  Utilizing Assumption \ref{ass1} and Proposition \ref{exact} yields 
$$\mathbb E_t\Big[\int_t^{\bar t}F(s,\rho^{\mathbb V^\epsilon}(s),S^{\mathbb V^\epsilon}(s);\mathbb V^{\epsilon})\mathrm ds\Big]\to0  
\text{ as }\bar t\to0, \text{ uniformly with } \epsilon.$$ Then we finish the proof by the arbitrariness of  $\epsilon$. 
\end{proof}

\subsection{Proof of Theorem \ref{thm_existence}}

With these preliminaries above, we are in a position to prove  Theorem \ref{thm_existence}. 

\begin{proof}[Proof of Theorem \ref{thm_existence}]
The continuity of $U$ is guaranteed by  Proposition \ref{conti}.  According to the definition of viscosity solution (see Definition \ref{def1-1}), it suffices to verify the viscosity solution inequalities \eqref{viscosity-sub} and \eqref{viscosity-super}, which are split into two steps. 

\textit{Step 1.} Suppose that $U-\varphi$ takes local maximum at $z_0:=(t_0,\rho_0,x_0)$ in the neighborhood $B_r(z_0)$ for some  $\varphi\in\mathcal C^{1,1,2}((0,T)\times\mathcal P^{\circ}(G)\times\mathbb R^n),$ where $r\in(0,\frac12)$ is some small constant chosen so that $B_r(z_0)$ remains inside the interior of the domain $(0,T) \times \mathcal P^{\circ}(G) \times \mathbb R^n$, and will be further specified (see \eqref{def-r}).   For any $t\in(t_0, T],$ by Proposition \ref{dynamic-pro}, 
we have that for any 
 admissible controls $\mathbb V\in\mathscr V_{\ell}[t_0,T],$ 
\begin{align*} 
U(t_0,\rho_0,x_0) \leq  \mathbb E_{t_0} \Big[ \int_{t_0}^{t} F(s,\rho(s),S(s),\mathbb V(s))  ds + U( t,\rho( t),S( t)) \Big]. 
\end{align*}In particular, 
for a constant control $\mathbb V(s)\equiv\mathbb V\in B_{\ell}$ for $s\in[t_0,T]$, we define the stopping time $\tau_r:=\inf\{s\in(t_0,T]:\|S(s)-x_0\|\wedge \|\rho(s)-\rho_0\|\ge r,S(t_0)=x_0,\rho(t_0)=\rho_0\}.$ Then we obtain     
\begin{align}\label{split-Iestimate}
0&\leq  \mathbb E_{t_0}\Big[\int_{t_0}^{t}F(s,\rho(s),S(s),\mathbb V)\mathrm ds+ U(t,\rho(t),S(t))-U(t_0,\rho_0,x_0)\Big]\notag\\
&= \mathbb E_{t_0}\Big[\int_{t_0}^{t}F(s,\rho(s),S(s),\mathbb V)\mathrm ds+ \big(U(t\wedge \tau_r,\rho(t\wedge \tau_r),S(t\wedge \tau_r))-U(t_0,\rho_0,x_0)\big)\Big]\notag\\&\quad + \mathbb E_{t_0}\Big[\big(U(t,\rho(t),S(t))-U(t\wedge\tau_r,\rho(t\wedge\tau_r),S(t\wedge\tau_r))\big)\Big]\notag\\
&=:I_1(t,t_0)+I_2(t,t_0).
\end{align}
Due to that $U-\varphi$ takes maximum at $z_0$ in $B_r(z_0),$ we obtain 
\begin{align}\label{insert1}
&\quad I_1(t,t_0)\notag\\&\leq \mathbb E_{t_0}\Big[\int_{t_0}^{t}F(s,\rho(s),S(s),\mathbb V)\mathrm ds+ \big(\varphi(t\wedge\tau_r,\rho(t\wedge\tau_r),S(t\wedge\tau_r))-\varphi(t_0,\rho_0,x_0)\big)\Big]\notag\\&= \mathbb E_{t_0}\Big[\int_{t_0}^{t}F(s,\rho(s),S(s),\mathbb V)\mathrm ds+ \int_{t_0}^{t\wedge\tau_r}\Big(\frac{\partial\varphi}{\partial s}(s,\rho(s),S(s))+\frac12 \mathrm{tr}(\sigma\sigma^{\top}D^2_S\varphi(s,\rho(s),S(s)))\notag\\
&\quad +\langle {\partial_{\rho}\varphi(s,\rho(s),S(s))}, D_S\mathcal H^{\mathbb V}_0(\rho(s),S(s))\rangle +\langle {D_{S} \varphi(s,\rho(s),S(s))},-D_{\rho}\mathcal H^{\mathbb V}_0(\rho(s),S(s))\rangle \Big)\mathrm ds\notag\\
&\quad - \int_{t_0}^{t\wedge\tau_r}\langle {D_{S} \varphi(s,\rho(s),S(s))},\sigma\mathrm dW(s)\rangle\Big],
\end{align}   
where in the last step  we use the It\^o formula before the stopping time $\tau_r$ (recall the SWHS \eqref{WHflow1}--\eqref{WHflow2}), making use of the property $\frac{\mathrm d}{\mathrm ds}f(\rho(s))=\langle\partial_{\rho}f(\rho(s)),\dot{\rho}(s)\rangle$ for a function $f:\mathcal P^{\circ}(G)\to\mathbb R$ whose Fr\'echet derivative $\partial_{\rho}f(\rho(s))$ exists (see \cite[Lemma 3.16]{MCC}). 

We now specify the choice of the parameter $r$ to ensure that the drift and diffusion coefficients in the SWHS remain bounded before the stopping time, which will be used in applying the dominated convergence theorem as $t$ approaches $t_0$. 
Since $z_0$ is an interior point of the domain $(0,T)\times\mathcal P^{\circ}(G)\times\mathbb R^n$, there exists a small constant $\epsilon_1>0$ such that the $\epsilon_1$-neighborhood $B_{\epsilon_1}(z_0)\subset (0,T)\times\mathcal P^{\circ}(G)\times\mathbb R^n$. Then we choose $r<\epsilon_1.$ Before the stopping time $\tau_r$, this ensures that $$\|S(s)-x_0\|\leq r<\epsilon_1,\|\rho(s)-\rho_0\|\leq r<\epsilon_1,$$ which implies  
 \begin{align}\sup_{i\in V}|\rho_i(s)-\rho_{0,i}|&\leq r,\label{eq-rho1}\\
   \sup_{(i,j)\in E}|S_i(s)-S_j(s)|&\leq \sup_{(i,j)\in E}(|S_i(s)-x_{0,i}|+|S_j(s)-x_{0,j}|+|x_{0,i}|+|x_{0,j}|)\leq 2(r+\|x_0\|),\label{eq-S1}
\end{align} 
where $\rho_0=(\rho_{0,1},\ldots,\rho_{0,n})$ and $x_0=(x_{0,1},\ldots,x_{0,n}).$ Moreover, letting $\min_{i\in V}\rho_{0,i}=\epsilon_2\in(0,1),$ we obtain from the triangle inequality and \eqref{eq-rho1} that \begin{align}\label{eq-rho2}\rho_i(s)\ge \rho_{0,i}-|\rho_{0,i}-\rho_i(s)|\ge \epsilon_2-r,\quad \forall\,i\in V.
\end{align} Hence we may take \begin{align}\label{def-r}r=\frac12(\epsilon_2\wedge \epsilon_1).
\end{align} As a result, for $s\leq \tau_r$, in view of \eqref{eq-S1} and \eqref{eq-rho2}, and by the locally Lipschitz continuity (Definition \ref{def-functions} (ii)) of partial derivatives of $\mathcal H_0$ (Assumption \ref{energy}) and linearity of the control potential $\mathcal V$, the coefficients $D_S\mathcal H^{\mathbb V}_0(\rho(s),S(s))$ and $D_{\rho}\mathcal H^{\mathbb V}_0(\rho(s),S(s))$ of the SWHS are bounded by a constant depending on $r,z_0$.  Namely, 
\begin{align*}
\|D_S\mathcal H_0^{\mathbb V}(\rho(s),S(s))\|&\leq \|D_S\mathcal H_0^{\mathbb V}(\rho(s),S(s))-D_S\mathcal H_0^{\mathbb V}(\rho_0,x_0)\|+\| D_S\mathcal H_0^{\mathbb V}(\rho_0,x_0)\|\\&\leq C(\|S(s)-x_0\|+\|\rho(s)-\rho_0\|)+ \| D_S\mathcal H_0^{\mathbb V}(\rho_0,x_0)\|\leq C(r,z_0),\\
\|D_{\rho}\mathcal H_0^{\mathbb V}(\rho(s),S(s))\|&\leq \|D_{\rho}\mathcal H_0^{\mathbb V}(\rho(s),S(s))-D_{\rho}\mathcal H_0^{\mathbb V}(\rho_0,x_0)\|+\| D_{\rho}\mathcal H_0^{\mathbb V}(\rho_0,x_0)\|\\&\leq C(\|S(s)-x_0\|+\|\rho(s)-\rho_0\|)+ \| D_{\rho}\mathcal H_0^{\mathbb V}(\rho_0,x_0)\|\leq C(r,z_0).
\end{align*}

Hence, 
dividing  by $t-t_0$ and sending $t\to t_0$, we derive from the condition \eqref{ass11}, moment estimate \eqref{H_0estimate},  dominated convergence theorem and the continuity of ${\partial_{\rho} \varphi},{D_x \varphi},D^2_x\varphi$, $D_x\mathcal H_0,D_{\rho}\mathcal H_0$ that  
\begin{align}\label{I1estimate1}
\lim_{t\to t_0}\frac{I_1(t,t_0)}{t-t_0}\leq \frac{\partial\varphi}{\partial t}(t_0,\rho_0,x_0)+\mathbb H(t,\rho,x,\mathbb V,\partial_{\rho}\varphi,D_x\varphi,D^2_x\varphi)|_{(t_0,\rho_0,x_0)},
\end{align}
where we recall \eqref{def-mathbbH}. 

For the term $I_2(t,t_0),$ define another stopping time: for some fixed $\alpha>0,$  ${\tilde \tau}_{\alpha}:=\inf\{s\in(t_0,T]:|\mathcal H_0(\rho(s),S(s))-\mathcal H_0(\rho_0,x_0)|\ge \alpha\}.$ Then we have 
\begin{align}\label{I-estimate}
I_2(t,t_0)&=\mathbb E_{t_0}\Big[\big(U(t,\rho(t),S(t))-U(\tau_r,\rho(\tau_r),S(\tau_r))\big)\mathbf 1_{\{ \omega\in\Omega:\tau_r(\omega)< t\}}\Big]\notag\\
&=\mathbb E_{t_0}\Big[\big(U(t,\rho(t),S(t))-U(\tau_r,\rho(\tau_r),S(\tau_r))\big)\mathbf 1_{\{\omega\in\Omega:\tilde\tau_{\alpha}(\omega)\ge t>  \tau_r(\omega)\}}\Big]\notag\\&\quad + \mathbb E_{t_0}\Big[\big(U(t,\rho(t),S(t))-U(\tau_r,\rho(\tau_r),S(\tau_r))\big)\mathbf 1_{\{\omega\in\Omega:\tilde\tau_{\alpha}(\omega)<t,\, t>  \tau_r(\omega)\}}\Big]\notag\\
&=:I_{2,1}(t,t_0)+I_{2,2}(t,t_0).
\end{align}
For the term $I_{2,1}(t,t_0),$ by means of the Chebyshev inequality, we deduce that for any $p\ge 1,$
\begin{align}\label{DPP-11}
&\quad \mathbb P(\tilde\tau_{\alpha}\ge t>\tau_r)\notag\\
&\leq \mathbb P\Big(\sup_{s\in[t_0,t]}(\|\rho(s)-\rho_0\|+\|S(t)-x_0\|)>r,\,\sup_{s\in[t_0,t]} |\mathcal H_0(\rho(s),S(s))-\mathcal H_0(\rho_0,x_0)|<\alpha\Big)\notag\\
&=\mathbb P\Big(\sup_{s\in[t_0,t]}(\|\rho(s)-\rho_0\|+\|S(s)-x_0\|)\mathbf 1_{\{ \sup\limits_{s\in[t_0,t]}|\mathcal H_0(\rho(s),S(s))-\mathcal H_0(\rho_0,x_0)|<\alpha\}}>r\Big)\notag\\
&\leq \mathbb E_{t_0}\Big[\sup\limits_{s\in[t_0,t]} (\|\rho(s)-\rho_0\|+\|S(s)-x_0\|)^p\mathbf 1_{\{ \sup\limits_{s\in[t_0,t]}|\mathcal H_0(\rho(s),S(s))-\mathcal H_0(\rho_0,x_0)|<\alpha\}}\Big]r^{-p}.
\end{align} 

Now we are in a position to estimate the conditional expectation on the right-hand side of \eqref{DPP-11}. 
Recalling the SWHS \eqref{WHflow1}--\eqref{WHflow2}, where $\mathcal H_1$ only depends on the density variable (Assumption \ref{energy}), and using condition \eqref{bound_prop}, we obtain 
\begin{align*}
\sup_{s\in[t_0,t]}\|\rho(s)-\rho_0\|=\sup_{s\in[t_0,t]}\Big\|\int_{t_0}^sD_S\mathcal H^{\mathbb V}_0(\rho(r),S(r))\mathrm dr\Big\|\leq  \int_{t_0}^t(C_1|\mathcal H^{\mathbb V}_0(\rho(r),S(r))|+C_2)\mathrm dr.
\end{align*} 
Taking the $p$th moment on both sides and using \eqref{H_0estimate} into account, we get
\begin{align}\label{rho-regularity}
\mathbb E_{t_0}\Big[\sup_{s\in[t_0,t]}\|\rho(s)-\rho_0\|^p\Big]\leq C(t-t_0)^p. 
\end{align}  
It follows from 
$\mathcal H_0\in(\textrm{Pr})$ (Assumption \ref{energy}) and Definition \ref{def-functions} (\textrm{Pr}-\romannumeral2) that when $|\mathcal H_0(\rho(s),S(s))-\mathcal H_0(\rho_0,x_0)|\leq \alpha,$ we have 
\begin{align*}\min_{i\in V}\rho_i(s)\ge C_1(\alpha),\quad  \sup_{(i,j)\in E}|S_i(s)-S_j(s)|\leq C_2(\alpha)
\end{align*} for some positive constants $C_1(\alpha),C_2(\alpha).$ This, together with the assumption that partial derivatives of $\mathcal H_0$ is locally Lipschitz continuous (Assumotion \ref{energy}) in the sense of Definition  \ref{def-functions} (\romannumeral2) (where we take $R_1=\min\{C_1(\alpha),\min_{i\in V}\rho_{0,i}\},R_2=\max\{C_2(\alpha),\max_{(i,j)\in E}|x_{0,i}-x_{0,j}|\}$), implies that, when $|\mathcal H_0(\rho(s),S(s))-\mathcal H_0(\rho_0,x_0)|\leq \alpha,$ \begin{align*}\|D_{\rho}\mathcal H^{\mathbb V}_0(\rho(s),S(s))\|&\leq \|D_{\rho}\mathcal H^{\mathbb V}_0(\rho_0,x_0)\|+C(R_1,R_2)(\|\rho(s)-\rho_0\|+\|S(s)-x_0\|)\\
&\leq C(\alpha,\rho_0,x_0)(1+\|\rho(s)-\rho_0\|+\|S(s)-x_0\|).
\end{align*} 
Hence, according to the SWHS \eqref{WHflow2}, when $s\leq\tilde\tau_{\alpha},$
\begin{align*}
\|S(s)-x_0\|&=\Big\|-\int_{t_0}^sD_{\rho}\mathcal H^{\mathbb V}_0(\rho(r),S(r))\mathrm dr-\int_{t_0}^s\sigma\mathrm dW(r)\Big\|\\
&\leq C(\alpha,\rho_0,x_0)\int_{t_0}^s(1+\|\rho(r)-\rho_0\|+\|S(r)-x_0\|)\mathrm dr+\Big\|\int_{t_0}^s\sigma\mathrm dW(r)\Big\|. 
\end{align*}
By the definition of $\tilde\tau_{\alpha},$ and the Burkholder inequality, we obtain that for any $p\ge 2,$
\begin{align*}
&\quad \mathbb E_{t_0}\Big[\sup_{s\in[t_0,t]}\|S(s)-x_0\|^p\mathbf 1_{\{ \sup_{s\in[t_0,t]}|\mathcal H_0(\rho(s),S(s))-\mathcal H_0(\rho_0,x_0)|<\alpha\}}\Big]
\\&\leq \mathbb E_{t_0}\Big[\sup_{s\in[t_0,t\wedge \tilde\tau_{\alpha}]}\|S(s)-x_0\|^p\Big]\\
&\leq C(\alpha,\rho_0,x_0,p)\mathbb E_{t_0}\Big[\int_{t_0}^{t\wedge\tilde\tau_{\alpha}}(1+\|\rho(r)-\rho_0\|^p+\|S(r)-x_0\|^p)\mathrm dr\Big]\\
&\quad +C(p)\mathbb E_{t_0}\Big[\sup_{s\in[t_0,t]}\Big\|\int_{t_0}^s\sigma\mathrm dW(r)\Big\|^p\Big]\\
&\leq C(\alpha,\rho_0,x_0,p)\mathbb E_{t_0}\Big[\int_{t_0}^{t}(1+\|\rho(r)-\rho_0\|^p+\sup_{u\in [t_0,r\wedge \tilde\tau_{\alpha} ]}\|S(u)-x_0\|^p)\mathrm dr\Big] +C(p)(t-t_0)^{\frac{p}{2}}, 
\end{align*} 
where in the second inequality we use the fact that $$\mathbb E_{t_0}\Big[\sup_{s\in[t_0,t\wedge\tilde\tau_{\alpha}]}\|\int_{t_0}^s\sigma\mathrm dW(r)\|^p\Big]\leq \mathbb E_{t_0}\Big[\sup_{s\in[t_0,t]}\|\int_{t_0}^s\sigma\mathrm dW(r)\|^p\Big].$$
On account of \eqref{rho-regularity} and the Gr\"onwall inequality, we arrive at 
\begin{align}\label{S-regularity}
\mathbb E_{t_0}\Big[\sup_{s\in[t_0,t]}\|S(s)-x_0\|^p\mathbf 1_{\{ \sup_{s\in[t_0,t]}|\mathcal H_0(\rho(s),S(s))-\mathcal H_0(\rho_0,x_0)|<\alpha\}}\Big]\leq C_{\alpha}(t-t_0)^{\frac p2}. 
\end{align}

Plugging the two time regularity estimates \eqref{rho-regularity} and \eqref{S-regularity} into \eqref{DPP-11} gives that for any $p\ge 2,$ 
\begin{align*}
\mathbb P(\tilde\tau_{\alpha}\ge t>\tau_r)\leq C_{\alpha}r^{-p}(t-t_0)^{\frac p2}.
\end{align*}
In particular, when $p=6$, we have $\mathbb P(\tilde\tau_{\alpha}\ge t>\tau_r)\leq C_{\alpha}r^{-6}(t-t_0)^{3}.$
This, together with the H\"older inequality, the growth property of $U$ \eqref{growth-U} and moment boundedness \eqref{H_0estimate} 
 implies that for $t\in(t_0,T],$  
\begin{align}\label{I11-estimate}
&\quad \frac{1}{t-t_0}I_{2,1}(t,t_0)\notag\\
&\leq (\mathbb E_{t_0}[|U(t,\rho(t),S(t))|^2+|U(\tau_r,\rho(\tau_r),S(\tau_r))|^2])^{\frac12}(\mathbb P( \tilde\tau_{\alpha}\ge t>\tau_r))^{\frac12}(t-t_0)^{-1}\notag\\
&\leq C\Big(1+\Big(\mathbb E_{t_0}\Big[\sup_{t\in[t_0,T]}|\mathcal H_0(\rho(t),S(t))|^{2p_1}\Big]\Big)^{\frac12}\Big) \Big(\mathbb P( \tilde\tau_{\alpha}\ge t>\tau_r)\Big)^{\frac12}(t-t_0)^{-1}\notag\\
&\leq C\Big(1+\sup_{i\in V}|\mathbb V_i|+\Big(\mathbb E_{t_0}\Big[\sup_{t\in[t_0,T]}|\mathcal H^{\mathbb V}_0(\rho(t),S(t))|^{2p_1}\Big]\Big)^{\frac12})\Big (\mathbb P( \tilde\tau_{\alpha}\ge t>\tau_r)\Big)^{\frac12}(t-t_0)^{-1}\notag\\
&\leq C_{\alpha}r^{-3}(t-t_0)^{\frac12}\to0\text{ as }t\to t_0.
\end{align}

For the term $I_{2,2}(t,t_0),$  using the Chebyshev inequality yields that for any $p\ge 1,$
\begin{align}\label{chebyshev1}
&\quad \mathbb P(\tilde\tau_{\alpha}<t,\,t>\tau_r)\leq \mathbb P(\tilde\tau_{\alpha}<t)\notag\\
&\leq \mathbb P(|\mathcal H_0(\rho(t),S(t))-\mathcal H_0(\rho_0,x_0)|\ge\alpha)\leq \alpha^{-p}\mathbb E_{t_0}[|\mathcal H_0(\rho(t),S(t))-\mathcal H_0(\rho_0,x_0)|^p].
\end{align}
 Recalling the definition of $\mathcal H^{\mathbb V}_0$ and $\mathcal V$ in \eqref{def-H0}, we use the time regularity of $\mathcal H^{\mathbb V}_0(\rho(\cdot),S(\cdot))$ in \eqref{moment-cont-1}, and \eqref{rho-regularity} to deduce that for any $p\ge 2,$  
\begin{align*}
&\quad \mathbb E_{t_0}[|\mathcal H_0(\rho(t),S(t))-\mathcal H_0(\rho_0,x_0)|^p]\\
&\leq\mathbb E_{t_0}[|\mathcal H^{\mathbb V}_0(\rho(t),S(t))-\mathcal V(\rho(t))-\mathcal H^{\mathbb V}_0(\rho_0,x_0)+\mathcal V(\rho_0)|^p] \\
&\leq C(t-t_0)^{\frac p2}+ C\mathbb E_{t_0}[\|\rho(t)-\rho_0\|^p]\leq C(t-t_0)^{\frac p2}. 
\end{align*}
Plugging this inequality into \eqref{chebyshev1} leads to 
\begin{align*}
 \mathbb P(\tilde\tau_{\alpha}<t,\,t>\tau_r)\leq C\alpha^{-p}(t-t_0)^{\frac p2}. 
\end{align*}
Similar to the estimate of $I_{2,1}(t,t_0)$ in \eqref{I11-estimate}, we derive that 
\begin{align}\label{I22-estimate}
\frac{1}{t-t_0}I_{2,2}(t,t_0)\leq C\alpha^{-3}(t-t_0)^{\frac12}\to0\text{ as }t\to t_0. 
\end{align}
Gathering estimates \eqref{I11-estimate} and \eqref{I22-estimate}, we obtain from \eqref{I-estimate} that 
\begin{align*}
\lim_{t\to t_0}\frac{I_2(t,t_0)}{t-t_0}=0. 
\end{align*}
Combining this with \eqref{I1estimate1} and  \eqref{split-Iestimate} produces 
\begin{align*}
0\leq \lim_{t\to t_0}\frac{1}{t-t_0}(I_{1}(t,t_0)+I_{2}(t,t_0))\leq \frac{\partial\varphi}{\partial t}(t_0,\rho_0,x_0)+\mathbb H(t,\rho,x,\mathbb V,\partial_{\rho}\varphi,D_x\varphi,D^2_x\varphi)|_{(t_0,\rho_0,x_0)}. 
\end{align*}  
Taking the infimum over $\mathbb V\in B_{\ell}$ verifies 
that $U$ is a viscosity subsolution of \eqref{HJB}.

\textit{Step 2.} Next, let $U-\varphi$ attain a local minimum at $z_0:=(t_0,\rho_0,x_0)$ in the neighborhood $B_r(z_0)$ with some $r\in(0,\frac12)$ for $\varphi\in\mathcal C^{1,1,2}((0,T)\times\mathcal P^{\circ}(G)\times\mathbb R^n).$  By Proposition \ref{dynamic-pro}, for any
$\epsilon,\delta>0,$ there exists $\mathbb V^{\epsilon,\delta}\in \mathscr V_{\ell}[t_0,T]$ such that
{\small  
\begin{align} \label{split2-estimate}
&\epsilon\delta>\mathbb E_{t_0}\Big[\int_{t_0}^{t_0+\epsilon}F(s,\rho^{\mathbb V^{\epsilon,\delta}}(s),S^{\mathbb V^{\epsilon,\delta}}(s),\mathbb V^{\epsilon,\delta}(s))\mathrm ds\notag\\
&+U(t_0+\epsilon,\rho^{\mathbb V^{\epsilon,\delta}}(t_0+\epsilon),S^{\mathbb V^{\epsilon,\delta}}(t_0+\epsilon))-U(t_0,\rho_0,x_0)\Big]\notag\\
&\ge \mathbb E_{t_0}\Big[\int_{t_0}^{t_0+\epsilon}F(s,\rho^{\mathbb V^{\epsilon,\delta}}(s),S^{\mathbb V^{\epsilon,\delta}}(s),\mathbb V^{\epsilon,\delta}(s))\mathrm ds\notag\\&+ \big(U((t_0+\epsilon)\wedge \tau^{\epsilon,\delta}_r,\rho^{\mathbb V^{\epsilon,\delta}}((t_0+\epsilon)\wedge \tau^{\epsilon,\delta}_r),S^{\mathbb V^{\epsilon,\delta}}((t_0+\epsilon)\wedge \tau^{\epsilon,\delta}_r))-U(t_0,\rho_0,x_0)\big)\Big]\notag\\
&+\mathbb E_{t_0}\Big[U(t_0+\epsilon,\rho^{\mathbb V^{\epsilon,\delta}}(t_0+\epsilon),S^{\mathbb V^{\epsilon,\delta}}(t_0+\epsilon))-U((t_0+\epsilon)\wedge \tau^{\epsilon,\delta}_r,\rho((t_0+\epsilon)\wedge \tau^{\epsilon,\delta}_r),S((t_0+\epsilon)\wedge \tau^{\epsilon,\delta}_r))\Big]\notag\\
&=:J_1(t_0+\epsilon,t_0)+J_2(t_0+\epsilon,t_0),
\end{align}} 
 where the stopping time $\tau^{\epsilon,\delta}_r:=\inf\{s\in(t_0,T]:\|S^{\mathbb V^{\epsilon,\delta}}(s)-x_0\|\wedge \|\rho^{\mathbb V^{\epsilon,\delta}}(s)-\rho_0\|>r,S^{\mathbb V^{\epsilon,\delta}}(t_0)=x_0,\rho^{\mathbb V^{\epsilon,\delta}}(t_0)=\rho_0\}.$ 

 For the term $J_1(t_0+\epsilon,t_0),$ the estimate is similar to that of $I_1(t,t_0)$ in \eqref{split-Iestimate}. By using that $U-\varphi$ attains a local minimum at $z_0:=(t_0,\rho_0,x_0)$ in $B_r(z_0)$, then applying the It\^o formula   to $\varphi$, dividing by $\epsilon$ and by means of the dominated convergence theorem, we deduce 
 \begin{align}\label{split2-estimate1}
\lim_{\epsilon\to0}\frac1{\epsilon}J_1(t_0+\epsilon,t_0)&\ge \frac{\partial\varphi}{\partial t}(t_0,\rho_0,x_0)+\mathbb H(t,\rho,x,\mathbb V,\partial_{\rho}\varphi,D_x\varphi,D^2_x\varphi)|_{(t_0,\rho_0,x_0)}\notag\\
&\ge \frac{\partial\varphi}{\partial t}(t_0,\rho_0,x_0)+ H(t,\rho,x,\partial_{\rho}\varphi,D_x\varphi,D^2_x\varphi)|_{(t_0,\rho_0,x_0)}. 
 \end{align} 
Here, we recall definitions of $\mathbb H$ and $H$ in  \eqref{def-mathbbH} and \eqref{def_H}, respectively.

For the term $J_2(t_0+\epsilon,t_0),$ the estimate is similar to that of $I_2(t,t_0)$ in \eqref{split-Iestimate}. Following the arguments of \eqref{I-estimate}, \eqref{I11-estimate} and \eqref{I22-estimate}, one can derive that 
\begin{align*}
\lim_{\epsilon\to0}\frac{1}{\epsilon}J_2(t_0+\epsilon,t_0)=0.
\end{align*}
Combining this with \eqref{split2-estimate} and \eqref{split2-estimate1} gives 
\begin{align*}
\delta\ge \lim_{\epsilon\to0}\frac{1}{\epsilon}(J_1(t_0+\epsilon,t_0)+J_2(t_0+\epsilon,t_0))\ge \frac{\partial\varphi}{\partial t}(t_0,\rho_0,x_0)+ H(t,\rho,x,\partial_{\rho}\varphi,D_x\varphi,D^2_x\varphi)|_{(t_0,\rho_0,x_0)}. 
\end{align*}
Then letting $\delta\to0$ yields that $U$ is a viscosity supersolution of \eqref{HJB}. 
\end{proof}

\section{Proof of Theorem \ref{uniqueness}}\label{sec_5}

This section establishes the uniqueness of viscosity solutions for the HJB equation \eqref{HJB} by means of the doubling of variable method.  The main challenge arises from the unbounded domain and the intricate structure of the state equation \eqref{WHflow1}--\eqref{WHflow2}. To address this, we employ an energy truncation technique that localizes the problem to a level set of the energy via smooth cut-off functions. Meanwhile, the truncated equation could preserve the structure of the original equation in a compact domain.  In Assumption \ref{energy}, we impose that $\mathcal H_0\in(\mathrm{Pr})$. According to Definition \ref{def-functions} (\romannumeral1), this requires the existence of $C_0>0$ such that $\widetilde{\mathcal H}_0:=\mathcal H_0-C_0\ge 0.$ For notational convenience, throughout this section we assume $\mathcal H_0\ge 0.$ Alternatively, one may work with $\widetilde{\mathcal H}_0$, which does not affect the proof.

\subsection{Truncated  HJB equation for optimal control problem of SWHS on graph}

In this subsection, we first provide a formal derivation of the truncated HJB equation. Then we introduce the semiconvex and semiconcave approximations of the viscosity solution of the truncated HJB equation. Some useful properties for these approximations are also presented. 

For a fixed $R>0,$ we introduce a positive non-increasing truncation function $\phi_R\in\mathcal C^{\infty}(\mathbb R_+;\mathbb R_+)$ satisfying that 
\begin{align}\label{phi-property}
\phi_R\equiv 1,\text{ when }0\leq r\leq R; \quad \phi_R= R^{-\beta},\text{ when } r\ge 2R,\end{align}  where $\beta\in(0,1)$ is given in Assumption \ref{ass_domi}. Moreover,  
\begin{align}\label{phi-property2}
|\phi_R(r)|\to 1\,(R\to\infty), \quad |\phi'_R(r)|\leq \frac{C}{R},\;|\phi''_R(r)|\leq \frac{C}{R^2}
\end{align} for some $C>0.$ Below we give  some concrete examples for $\phi_R$ satisfying  \eqref{phi-property} and \eqref{phi-property2}. 
\begin{example} 
Let $\phi$ be a smooth non-increasing cut-off function $\phi:\mathbb R\to[0,1]$ satisfying 
$\phi(r)=1,\,r\leq 0,$ and $\phi(r)=0,\,r\ge 1.$ We define the non-increasing truncation function $\phi_R:\mathbb R_+\to \mathbb R_+$ for $R>0$ by  
\begin{align*}
\phi_R(r)=
\begin{cases}
1,\quad \text{if }r\in[0, R];\\
R^{-\beta}+(1-R^{-\beta})\phi(\frac{r-R}{R}),\quad \text{if }r\in(R,2R);\\
R^{-\beta},\quad \text{if }r\in[2R,+\infty).
\end{cases}
\end{align*}
It can be seen  that $\phi_R$ satisfies the properties \eqref{phi-property}--\eqref{phi-property2}. 
Two concrete examples of $\phi$ are as follows: 
\begin{align*}
\phi(r)=
\begin{cases}
1,\quad \text{if } r\leq 0,\\
e^{-\frac{1}{1-r}}(e^{-\frac{1}{1-r}}+e^{-\frac{1}{r}})^{-1},\quad \text{if }r\in(0,1),\\
0,\quad \text{if }r\ge 1,
\end{cases}
\text{ and }\;
\phi(r)=
\begin{cases}
1,\quad \text{if } r\leq 0,\\
e^{1-\frac{1}{1-r^2}},\quad \text{if }r\in(0,1),\\
0,\quad \text{if }r\ge 1.
\end{cases}
\end{align*} 
\end{example}

With the truncation function $\phi_R$, we truncate $\mathcal H_0(\rho,x)$ to obtain the truncated domain: \begin{align}\label{def-A-phi}
\mathcal A_R:=\{(\rho,x):\mathcal H_0(\rho,x)< 2R\},  \quad \phi_{R,0}(\rho,x):=\phi_{R}(\mathcal H_0(\rho,x)).
\end{align}  
We now present a formal derivation of the truncated HJB equation. As a first step, we multiply \eqref{HJB} by $\phi_{R,0}$  
 \begin{align}\label{phi-multiply}
&\phi_{R,0}\frac{\partial U}{\partial t} +\phi_{R,0}\Big(\langle \partial_{\rho}U, D_x\mathcal H_0(\rho,x) \rangle - \langle D_xU, D_{\rho}\mathcal H_0(\rho,x) \rangle \notag\\& + \frac{1}{2} \mathrm{tr}(\sigma \sigma^\top D^2_xU) - \widehat F(t, \rho, x, D_xU)\Big)=0. 
 \end{align} 
Introduce the new function $\widehat U:=U\phi_{R,0},$ which will serve as the candidate solution of the truncated HJB equation. Our goal is therefore to derive an equation satisfied by $\widehat U$. To this end, we note the following identities:  
 \begin{align*}
 &\phi_{R,0}D_xU=D_x\widehat U-UD_x\phi_{R,0},\quad \phi_{R,0}\partial_{\rho}U=\partial_{\rho}\widehat U-U\partial_{\rho}\phi_{R,0},
 \end{align*}
 and for $i,j=1,\ldots,n,$\begin{align*}
 &\phi_{R,0}D^2_{x_ix_j}U=D^2_{x_ix_j}\widehat U-UD^2_{x_ix_j}\phi_{R,0}-D_{x_i}\phi_{R,0}D_{x_j}U-D_{x_j}\phi_{R,0}D_{x_i}U\\
 &=D^2_{x_ix_j}\widehat U-UD^2_{x_ix_j}\phi_{R,0}-\frac{D_{x_i}\phi_{R,0}}{\phi_{R,0}}D_{x_j}\widehat U-\frac{D_{x_j}\phi_{R,0}}{\phi_{R,0}}D_{x_i}\widehat U+2U\frac{D_{x_i}\phi_{R,0}D_{x_j}\phi_{R,0}}{\phi_{R,0}},
 \end{align*}
 due to $D_{x_j}U=\frac{D_{x_j}\widehat U-UD_{x_j}\phi_{R,0}}{\phi_{R,0}}.$  
Substituting these expressions into \eqref{phi-multiply}  yields  
\begin{align}\label{HRU-def}
 &H^R_U(t,\rho,x,\partial_{\rho}\widehat U,D_x\widehat U,D^2_x \widehat U):=\Big\{\langle \partial_{\rho}\widehat U,D_x\mathcal H_0\rangle-\langle D_x\widehat U,D_{\rho}\mathcal H_0\rangle-\frac{\sigma\sigma^{\top}D_x\phi_{R,0}}{\phi_{R,0}}D_x\widehat U\notag\\&-U\Big(\langle \partial_{\rho}\phi_{R,0},D_x\mathcal H_0\rangle-\langle D_x\phi_{R,0},D_{\rho}\mathcal H_0\rangle +\frac12\mathrm{tr}(\sigma\sigma^{\top}D^2_x\phi_{R,0})-\frac{\sigma\sigma^{\top}D_x\phi_{R,0}}{\phi_{R,0}}D_x\phi_{R,0}\Big)\notag\\
 &+\frac12\mathrm{tr}(\sigma\sigma^{\top}D^2_x\widehat U)-\sup_{\mathbb V\in B_{\ell}}\Big[\langle D_x\widehat U,\mathbb V\rangle-U\langle D_x\phi_{R,0},\mathbb V\rangle-\phi_{R,0}F(t,\rho,x,\mathbb V)\Big]\Big\}\Big|_{(t,\rho,x)}.
 \end{align} 
According to the condition \eqref{assu4-eq} in Assumption \ref{ass_domi}, i.e., $D_x\mathcal H_0(\rho,x)\in T_{\rho}\mathcal P(G),$  
 we obtain the relation 
 \begin{align}\label{technical-eq}&\langle \partial_{\rho}\phi_{R,0},D_x\mathcal H_0\rangle
  =\phi'_R\langle {\partial_{\rho}\mathcal H_0},D_x\mathcal H_0\rangle 
 =\phi'_R\langle D_{\rho}\mathcal H_0,D_x\mathcal H_0\rangle
 =\langle D_x\phi_{R,0},D_{\rho}\mathcal H_0\rangle, 
 \end{align}
where we recall that $\partial_{\rho}\mathcal H_0$ is the Fr\'echet derivative and $D_{\rho}\mathcal H_0$ is the Euclidean gradient of $\mathcal H_0$. 
 
Substituting \eqref{technical-eq} into \eqref{HRU-def}, we obtain the truncated Hamiltonian \begin{align}\label{def_H4}
&H^R_U(t,\rho,x,p,q,Q)=\langle p,D_x\mathcal H_0(\rho,x)\rangle
-\langle q,D_{\rho}\mathcal H_0(\rho,x)\rangle
 +\frac{1}{2}\mathrm {tr}(\sigma\sigma^{\top}Q)-  \frac{1}{\phi_{R,0}}\mathrm {tr}(\sigma\sigma^{\top}D_x\phi_{R,0}q^{\top})\notag\\
&-U\Big(\frac 12\mathrm {tr}(\sigma\sigma^{\top}D^2_x\phi_{R,0})-\frac{1}{\phi_{R,0}} \mathrm {tr}(\sigma\sigma^{\top}D_x\phi_{R,0}(D_x\phi_{R,0})^{\top}) \Big)-\mathbb F^R(t,\rho,x,q-UD_x\phi_{R,0}),
\end{align} 
where $t\in(0,T),\,\rho\in\mathcal P^{\circ}(G),\,x,\,p,q\in \mathbb R^{n},Q\in\mathcal S^{n\times n},\phi_{R,0}=\phi_R(\mathcal H_0)$ and \begin{align}\label{FR-def}\mathbb F^R(t,\rho,x,q-UD_x\phi_{R,0})=\sup_{\mathbb V\in B_{\ell}}\{\langle q-UD_x\phi_{R,0},\mathbb V\rangle -\phi_{R,0}F(t,\rho,x,\mathbb V)\}.
\end{align} 
 Here, we use the subscript `$U$' in `$H^R_U$' to describe the dependence on the function $U$.  
Combining \eqref{def_H4} and \eqref{phi-multiply}, we thus formally derive 
$\frac{\partial \widehat U}{\partial t}+H^R_U(t,\rho,x,\partial_{\rho}\widehat U,D_x\widehat U,D^2_x \widehat U)=0.$

 Confining our attention to the domain $\mathcal A_R$, we formulate the \textit{truncated HJB equation} on $(0,T)\times\mathcal A_R:$ 
 \begin{align}\label{HJB_R}
 \frac{\partial \widehat U}{\partial t}+H^R_U(t,\rho,x,\partial_{\rho}\widehat U,D_x\widehat U,D^2_x \widehat U)=0,\quad \widehat U(T,\rho,x)=h(\rho,x).
 \end{align}  
Below is the definition of the viscosity solution of the truncated HJB equation \eqref{HJB_R}. For details of this type of truncated equations in the bounded domain $\{x\in\mathbb R^n:\|x\|<R\}$ of the Euclidean space $\mathbb R^n$, we refer to \cite[Section V.9]{book-controlled}.

 \begin{definition}\label{def1} Let $U\in\mathcal C((0,T)\times\mathcal P^{\circ}(G)\times\mathbb R^n)$ and $\widehat U=U\phi_{R,0}.$  

(\romannumeral1) A continuous function $\widehat U$ is called a viscosity subsolution of \eqref{HJB_R} in $\mathcal A_R$ if $U(T,\cdot)\leq h(\cdot)$, and  
for every $(t_0,\rho_0,x_0)\in(0,T)\times\mathcal A_R$ and every $\varphi\in\mathcal C^{1,1,2}((0,T)\times\mathcal A_R)$  such that $\widehat U-\varphi$ takes local maximum at $(t_0,\rho_0,x_0)$ with $\widehat U(t_0,\rho_0,x_0)=\varphi(t_0,\rho_0,x_0),$    we have  
\begin{align}\label{viscosity-sub-local}
&\frac{\partial\varphi}{\partial t}(t_0,\rho_0,x_0)+ H^R_{U}(t_0,\rho_0,x_0,\partial_{\rho}\varphi(t_0,\rho_0,x_0),D_x\varphi(t_0,\rho_0,x_0),D^2_x \varphi(t_0,\rho_0,x_0))\ge 0. 
\end{align}

(\romannumeral2) A continuous function $\widehat U$ is called a viscosity supersolution of \eqref{HJB_R} in $\mathcal A_R$ if $U(T,\cdot)\ge  h(\cdot),$ and for every $(t_0,\rho_0,x_0)\in(0,T)\times\mathcal A_R$ and every $\varphi\in\mathcal C^{1,1,2}((0,T)\times\mathcal A_R)$ such that $\widehat U-\varphi$ takes local minimum at $(t_0,\rho_0,x_0)$  with $\widehat U(t_0,\rho_0,x_0)=\varphi(t_0,\rho_0,x_0),$  we have  
\begin{align}\label{viscosity-super-local}
&\frac{\partial\varphi}{\partial t}(t_0,\rho_0,x_0)+H^R_{U}(t_0,\rho_0,x_0,\partial_{\rho}\varphi(t_0,\rho_0,x_0),D_x\varphi(t_0,\rho_0,x_0),D^2_x \varphi(t_0,\rho_0,x_0) )\leq 0.
\end{align}

(\romannumeral3) A continuous function $\widehat U$ is called a viscosity solution of \eqref{HJB_R} in $\mathcal A_R$ if it is both a viscosity subsolution and a viscosity supersolution in $\mathcal A_R$. 
\end{definition}
We remark that the additional condition ``$\widehat U(t_0,\rho_0,x_0)=\varphi(t_0,\rho_0,x_0)$'' in Defnition \ref{def1} (\romannumeral1) and (\romannumeral2) is  used in the proof of Lemma \ref{coro2_local}; see also \cite[Chapter II, Definition 4.1]{book-controlled} and \cite[Lemma I.1]{LionsII} for further details on such condition. The following lemma shows the relation between the viscosity solution of \eqref{HJB} and the viscosity solution of \eqref{HJB_R}, which will be used in the uniqueness proof of Theorem \ref{uniqueness}. The proof relies on the arguments involving \eqref{phi-multiply},  \eqref{HRU-def}, and \eqref{def_H4} on $\mathcal A_R.$ 
We postpone the proof to Appendix \ref{app1} for completeness.

\begin{lemma}\label{coro2_local}
If $U$ is a viscosity subsolution (resp. supersolution) of \eqref{HJB} in the whole space $(0,T)\times\mathcal P^{\circ}(G)\times\mathbb R^n$, then $\widehat U=U\phi_{R,0}$ is a viscosity subsolution  (resp. supersolution) of \eqref{HJB_R} in $(0,T)\times\mathcal A_R$.
\end{lemma}

In general, the viscosity solution of an HJB equation is not necessarily differentiable. To construct a sequence of differentiable approximations, we next introduce the semiconvex and semiconcave approximations of $\widehat U.$  

For fixed $N>0,$ let $B_N:=\{x\in\mathbb R^n:\|x\|\leq N\}$. For $R>1,$ define the compact set \begin{align}\label{malcalA-def}
\mathcal A_{R,N}:=\{(t,\rho,x)\in[0,T]\times\mathcal P^{\circ}(G)\times\mathbb R^n: \mathcal H_0(\rho,x)\leq 2R, \|x\|\leq N\}.
\end{align}  
For 
$N,R>1, \vartheta\in(0,1),$ and $(t,\rho,x)\in \mathcal A_{R,N},$ define 
\begin{align}\label{Uvartheta}
\widehat U^{\vartheta}(t,\rho,x):=\sup_{(s,\mu,y)\in \mathcal A_{R,N} }\Big\{\widehat U(s,\mu,y)-\frac{|t-s|^2+\|\rho-\mu\|^2+\|x-y\|^2}{2\vartheta}\Big\}. 
\end{align}
\begin{lemma} \label{uniform-converge}
Let $U\in\mathcal C(\mathcal A_{R,N}). $ 
 For fixed $\vartheta\in(0,1),$ $\widehat U^{\vartheta}$ is semiconvex, i.e., there exists a constant $K>0$, 
such that $(t,\rho,x)\mapsto \widehat U^{\vartheta}(t,\rho,x)+K(|t|^2+\|\rho\|^2+\|x\|^2)$  is convex. Moreover, $\{\widehat U^{\vartheta}\}_{\vartheta}$ is the semiconvex approximation of $\widehat U$, i.e., $\widehat U^{\vartheta}$ converges to $\widehat U$ uniformly on $\mathcal A_{R,N}$ as $\vartheta\to0.$
\end{lemma}
\begin{proof}
We first show that $\widehat U^{\vartheta}$ is semiconvex. By the definition of $\widehat U^{\vartheta}$ \eqref{Uvartheta}, we have  
\begin{align*}
&(t,\rho,x)\to \widehat U^{\vartheta}(t,\rho,x)+\frac{1}{2\vartheta}(|t|^2+\|\rho\|^2+\|x\|^2)\\
&=\sup_{(s,\mu,y)\in\mathcal A_{R,N}}\Big\{\widehat U(s,\mu,y)-\frac{|t-s|^2+\|\rho-\mu\|^2+\|x-y\|^2}{2\vartheta}+\frac{1}{2\vartheta}(|t|^2+\|\rho\|^2+\|x\|^2)\Big\}\\
&=\sup_{(s,\mu,y)\in\mathcal A_{R,N} }\Big\{\widehat U(s,\mu,y)-\frac{|s|^2+\|\mu\|^2+\|y\|^2}{2\vartheta}+\frac{ts+\rho^{\top}\mu+x^{\top}y}{\vartheta}\Big\}. 
\end{align*} 
Since this is the supremum of a family of linear functions in $(t,\rho,x)$, it follows that the map 
 $(t,\rho,x)\mapsto \widehat U^{\vartheta}(t,\rho,x)+\frac{1}{2\vartheta}(|t|^2+\|\rho\|^2+\|x\|^2)$ is convex, which means that $\widehat U^{\vartheta}$ is semiconvex.  

 For each fixed $(t,\rho,x)\in\mathcal A_{R,N},$ denote by $(t^{\vartheta},\rho^{\vartheta},x^{\vartheta})$  the maximizer of \eqref{Uvartheta},  i.e.,  
\begin{align}\label{ine23}\widehat U^{\vartheta}(t,\rho,x)=\widehat U( t^{\vartheta},\rho^{\vartheta},x^{\vartheta})-\frac{|t-t^{\vartheta}|^2+\|\rho-\rho^{\vartheta}\|^2+\|x-x^{\vartheta}\|^2}{2\vartheta}\ge \widehat U(t,\rho,x)
\end{align} for the fixed $\vartheta\in(0,1)$. Then we have 
\begin{align}\label{limvartheta}|t-t^{\vartheta}|^2+\|\rho-\rho^{\vartheta}\|^2+\|x-x^{\vartheta}\|^2\leq 
4\vartheta \sup_{\mathcal A_{R,N}}|\widehat U(t,\rho,x)|.\end{align}
  Moreover, from \eqref{ine23} we also obtain 
 \begin{align*}
 0\leq \widehat U^{\vartheta}(t,\rho,x)-\widehat U(t,\rho,x)\leq \widehat U(t^{\vartheta},\rho^{\vartheta},x^{\vartheta})-\widehat U(t,\rho,x),
 \end{align*}
 which implies that $\lim_{\vartheta\to0}\widehat U^{\vartheta}(t,\rho,x)=\widehat U(t,\rho,x)$ uniformly for $(t,\rho,x)\in\mathcal A_{R,N}$ due to the uniform continuity of $\widehat U$ on $\mathcal A_{R,N}.$
 \end{proof}
Similarly, we have the semiconcave approximation denoted by $\widehat U_{\vartheta}:$ \begin{align*}
\widehat U_{\vartheta}(t,\rho,x):=\inf_{(s,\mu,y)\in\mathcal A_{R,N}}\Big\{\widehat U(s,\mu,y)+\frac{|t-s|^2+\|\rho-\mu\|^2+\|x-y\|^2}{2\vartheta}\Big\},
\end{align*}
i.e., there exists a constant $K>0$, 
such that $(t,\rho,x)\mapsto \widehat U_{\vartheta}(t,\rho,x)-K(|t|^2+\|\rho\|^2+\|x\|^2)$  is concave. Similar to Lemma \ref{uniform-converge}, we can also show that $\widehat U_{\vartheta}$ converges to $\widehat U$ uniformly on $\mathcal A_{R,N}.$

Typically, the semiconvex approximation $\widehat U^{\vartheta}$ is not a viscosity subsolution of  the original truncated HJB equation \eqref{HJB_R}, where the Hamiltonian  is $H^R_U$. In order to investigate the underlying HJB equation for $\widehat U^{\vartheta}$, we introduce the following auxiliary  functional 
\begin{align}\label{Hvartheta}
(H^R_{ U})^{\vartheta}(t,\rho,x,p,q,Q):=\sup_{(s,\mu,y)\in \mathcal A_{R,N}}\{H^R_{ U}(s,\mu,y,p,q,Q): |t-s|+\|\rho-\mu\|+\|x-y\|\leq \sqrt{\vartheta}\}.
\end{align}
From \eqref{Hvartheta}, the definition of $H^R_U$ \eqref{HRU-def}, and the  continuity of $H^R_U,$ 
we deduce that \begin{align}\label{appro-Hvartheta}\lim_{\vartheta\to0}(H^R_U)^{\vartheta}(t,\rho,x,p,q,Q)=H^R_U(t,\rho,x,p,q,Q)\text{ uniformly for }(t,\rho,x,p,q,Q) 
\end{align}  when
$(t,\rho,x)\in\mathcal A_{R,N}$ (see \eqref{malcalA-def} for the definition) and $(p,q,Q)$ in a compact set of $\mathbb R^n\times\mathbb R^n\times\mathcal S^{n\times n}$.  

Next, we show that the semiconvex approximation $\widehat U^{\vartheta}$ is the viscosity subsolution for the HJB equation with the Hamiltonian $(H^R_U)^{\vartheta}$ \eqref{Hvartheta}, but restricted to a shrunken domain of the original domain \eqref{malcalA-def}.  
For a small parameter $\gamma\ll1,$ we denote the $\gamma$-shrunken set by 
\begin{align}\label{shrunken}(\mathcal A_{R,N})_{\gamma}:=\{v\in \mathcal A_{R,N}:\mathrm{dist}(v,\partial\mathcal A_{R,N})>\gamma\}. 
\end{align}

\begin{lemma}\label{lem1}
Let $\widehat U=U\phi_{R,0}$ be a viscosity subsolution of \eqref{HJB_R} in $\mathcal A_R$. For each fixed $\vartheta\in(0,1),$ and each $\varphi\in\mathcal C^{1,1,2},$ if $\widehat  U^{\vartheta}-\varphi$ attains local maximum at point  $(t,\rho,x)\in (\mathcal A_{R,N})_{\sqrt{\vartheta}k_0},$ then 
\begin{align}\label{HJB1}
\frac{\partial\varphi}{\partial t}(t,\rho,x) +(H^R_{U})^{\vartheta}(t,\rho,x,\partial_{\rho}\varphi(t,\rho,x),D_x\varphi(t,\rho,x),D^2_x \varphi(t,\rho,x))\ge 0.
\end{align} 
Here, $k_0:=(1+4\sup\limits_{(t,\rho,x)\in \mathcal A_{R,N}}|\widehat U(t,\rho,x)|)^{\frac12}.$ 
\end{lemma}
\begin{proof}
The main idea is to show that if $\widehat U^{\vartheta} - \varphi$ attains a local maximum at some point, then $\widehat U - \varphi$ attains a local maximum at a nearby point. This allows us to apply the definition of a viscosity solution for $\widehat U$ to establish \eqref{HJB1}. 
For fixed $\vartheta\in(0,1)$ and $\varphi\in\mathcal C^{1,1,2}$, suppose that $\widehat U^{\vartheta}-\varphi$ attains a  local maximum 
at $(t,\rho,x)\in (\mathcal A_{R,N})_{\sqrt{\vartheta}k_0}.$ We denote by $(t^{\vartheta},\rho^{\vartheta},x^{\vartheta})$ the maximizer point of \eqref{Uvartheta}, i.e., \begin{align}\label{max-Uvartheta}\widehat U^{\vartheta}(t,\rho,x)=\widehat U(t^{\vartheta},\rho^{\vartheta},x^{\vartheta})-\frac{|t-t^{\vartheta}|^2+\|\rho-\rho^{\vartheta}\|^2+\|x-x^{\vartheta}\|^2}{2\vartheta}.\end{align}  From \eqref{limvartheta} we derive 
\begin{align*}
|t-t^{\vartheta}|^2+\|\rho-\rho^{\vartheta}\|^2+\|x-x^{\vartheta}\|^2\leq 
\vartheta k_0^2.
\end{align*} 
Therefore for every $(s,\mu,y)\in(\mathcal A_{R,N})_{\sqrt{\vartheta}k_0}$  sufficiently 
near $(t,\rho,x)$, the point $(s-t+t^{\vartheta},\mu-\rho+\rho^{\vartheta},y-x+x^{\vartheta})\in \mathcal A_{R,N}.$ 
Then for such $(s,\mu,y),$ we have 
\begin{align*}
&\quad \widehat U(t^{\vartheta},\rho^{\vartheta}, x^{\vartheta})-\varphi( t,\rho, x)\\
&=\Big(\widehat U^{\vartheta}(t,\rho,x)+\frac{1}{2\vartheta}(|t-t^{\vartheta}|^2+\|\rho-\rho^{\vartheta}\|^2+\|x-x^{\vartheta}\|^2)\Big)-\varphi(t,\rho,x)\\
&\ge \widehat U^{\vartheta}(s,\mu,y) -\varphi(s,\mu,y)+ \frac{1}{2\vartheta}(|t-t^{\vartheta}|^2+\|\rho-\rho^{\vartheta}\|^2+\|x-x^{\vartheta}\|^2)\\
&=\sup_{(r,\nu,z)\in\mathcal A_{R,N}}\Big\{\widehat U(r,\nu,z)-\frac{|s-r|^2+\|\nu-\mu\|^2+\|y-z\|^2}{2\vartheta}\Big\}\\
&\quad +\frac{1}{2\vartheta}(|t-t^{\vartheta}|^2+\|\rho-\rho^{\vartheta}\|^2+\|x-x^{\vartheta}\|^2)-\varphi(s,\mu,y)\\
&\ge \widehat U(s-t+t^{\vartheta},\mu-\rho+\rho^{\vartheta},y-x+x^{\vartheta})
-\varphi(s,\mu,y),
\end{align*} 
where for the last step, we take $r= s-t+t^{\vartheta},\nu=\mu-\rho+\rho^{\vartheta},z= y-x+x^{\vartheta} $ so that  $-\frac{|s-r|^2+\|\nu-\mu\|^2+\|y-z\|^2}{2\vartheta}+\frac{1}{2\vartheta}(|t-t^{\vartheta}|^2+\|\rho-\rho^{\vartheta}\|^2+\|x-x^{\vartheta}\|^2)=0$. 
Denote $\tau=s-t+t^{\vartheta},\eta=\mu-\rho+\rho^{\vartheta},\xi=y-x+x^{\vartheta}.$ 
 We arrive at
\begin{align*}
&\widehat U(t^{\vartheta},\rho^{\vartheta}, x^{\vartheta})-\varphi( t,\rho, x)
\\
&\ge \widehat U(\tau, \eta, \xi)-\varphi(\tau+t-t^{\vartheta},\eta+\rho-\rho^{\vartheta},\xi+x-x^{\vartheta}). 
\end{align*}
This implies that the mapping $(\tau,\eta,\xi)\mapsto  \widehat U(\tau, \eta, \xi)-\varphi(\tau+t-t^{\vartheta},\eta+\rho-\rho^{\vartheta},\xi+x-x^{\vartheta})$
has a local maximizer at $(t^{\vartheta},\rho^{\vartheta},x^{\vartheta})$. By the definition of the viscosity subsolution \eqref{viscosity-sub-local}, we derive 
\begin{align*}
\frac{\partial\varphi}{\partial t}(t,\rho,x)+H^R_{U}(t^{\vartheta},\rho^{\vartheta},x^{\vartheta},\partial_{\rho}\varphi(t,\rho,x),D_x\varphi(t,\rho,x),D^2_x\varphi(t,\rho,x))\ge 0,
\end{align*} 
where we also use the fact that \begin{align*}&(\frac{\partial\varphi}{\partial\tau},\partial_{\eta}\varphi,D_{\xi}\varphi,D^2_{\xi}\varphi)(\tau+t-t^{\vartheta},\eta+\rho-\rho^{\vartheta},\xi+x-x^{\vartheta})|_{(\tau,\eta,\xi)=(t^{\vartheta},\rho^{\vartheta},x^{\vartheta})}\\&=(\frac{\partial\varphi}{\partial t},\partial_{\rho}\varphi,D_x\varphi,D^2_x\varphi)(t,\rho,x).
\end{align*} 
By the definition of $(H^R_U)^{\vartheta}$ in \eqref{Hvartheta}, we derive \eqref{HJB1}.  
The proof is finished. 
\end{proof}
 
Similarly, define  
\begin{align}\label{Hvartheta2}
(H^R_U)_{\vartheta}(t,\rho,x,p,q,Q):=\inf_{(s,\mu,y)\in\mathcal A_{R,N}}\{H^R_U(s,\mu,y,p,q,Q): |t-s|+\|\rho-\mu\|+\|x-y\|\leq \sqrt{\vartheta}\}.
\end{align} 
We have 
\begin{align}\label{appro-Hvartheta2}\lim_{\vartheta\to0}(H^R_U)_{\vartheta}(t,\rho,x,p,q,Q)=H^R_U(t,\rho,x,p,q,Q)\text{ uniformly for }(t,\rho,x,p,q,Q) 
\end{align} when
$(t,\rho,x)\in\mathcal A_{R,N}$ and $(p,q,Q)$ in a compact set of $\mathbb R^n\times\mathbb R^n\times\mathcal S^{n\times n}.$ Similar to Lemma \ref{lem1}, one can verify that the semiconcave approximation $\widehat U_{\vartheta}$ is the viscosity supersolution for the HJB equation with the Hamiltonian $(H^R_U)_{\vartheta}$ \eqref{Hvartheta} in a shrunken domain.  

\begin{lemma}
Let $\widehat U=U\phi_{R,0}$ be a viscosity supersolution of \eqref{HJB_R}  in $\mathcal A_R$.  
For each fixed $\vartheta\in(0,1),$ and each $\varphi\in\mathcal C^{1,1,2},$ if $\widehat U_{\vartheta}-\varphi$ attains local minimum at point $(t,\rho,x)\in (\mathcal A_{R,N})_{\sqrt{\vartheta}k_0},$   then 
\begin{align}\label{HJB2}
\frac{\partial\varphi}{\partial t}(t,\rho,x)+(H^R_{U})_{\vartheta}(t,\rho,x,\partial_{\rho}\varphi(t,\rho,x) ,D_x\varphi(t,\rho,x),D^2_x\varphi(t,\rho,x))\leq 0.
\end{align}
\end{lemma}

We end this subsection by showing regularity properties of $\mathbb F^R$ given in \eqref{FR-def}. 
 \begin{lemma}  There exist constants $C_1,C_2>0$ such that for all $(t,\rho,x)\in[0,T]\times\mathcal P^{\circ}(G)\times\mathbb R^n,$\begin{align}
 &|\mathbb F^R(t,\rho,x,p)-\mathbb F^R(t,\rho,x,q)|\leq \ell\|p-q\|,\quad p,q\in\mathbb R^n,\label{pq-regularity}
 \end{align}
 and that for any $r_1\in(0,\frac1n)$ and $r_2>0,$ when  
$\rho,\mu\in\mathcal P_{r_1}(G)$, and $x,y\in\mathbb R^n$ such that $\max\limits_{(i,j)\in E}(|x_i-x_j|\vee|y_i-y_j|)\leq r_2,$ 
\begin{align}&|\mathbb F^R(t,\rho,x,p)-\mathbb F^R(s,\mu,y,p)|\leq C_2 (R,r_1,r_2,\ell)(\|\rho-\mu\|+\|x-y\|+|t-s|),\label{xy-regularity}
 \end{align}
 where $t,s\in[0,T],p\in\mathbb R^n.$ 
 \end{lemma}
 \begin{proof}
By the relation $\inf_{\mathbb V}f_1({\mathbb V})+\sup_{\mathbb V}f_2({\mathbb V})\leq \sup_{\mathbb V}(f_1({\mathbb V})+f_2({\mathbb V}))$ for two functions $f_1,f_2:B_{\ell}\to\mathbb R,$  we obtain from the definition of $\mathbb F^R$ \eqref{FR-def} that 
\begin{align*}
&\quad -\mathbb F^R(t,\rho,x,p)+\mathbb F^R(t,\rho,x,q)\\
&=\inf_{\mathbb V\in B_{\ell}}\{-\langle p,\mathbb V\rangle+\phi_{R,0}F(t,\rho,x,\mathbb V)\}+\sup_{\mathbb V\in B_{\ell}}\{\langle q,\mathbb V\rangle-\phi_{R,0}F(t,\rho,x,\mathbb V)\}\\
&\leq \sup_{\mathbb V\in B_{\ell}}\{-\langle p,\mathbb V\rangle +\phi_{R,0}F(t,\rho,x,\mathbb V)+\langle q,\mathbb V\rangle-\phi_{R,0}F(t,\rho,x,\mathbb V)\}\\&\leq \ell \|p-q\|,
\end{align*}
and similarly $\mathbb F^R(t,\rho,x,p)-\mathbb F^R(t,\rho,x,q)\leq \ell\|p-q\|$, which implies \eqref{pq-regularity}. 

 In addition, recalling the definition of $\phi_{R,0}=\phi_R(\mathcal H_0),$ by the locally Lipschitz property of $\mathcal H_0$ and $F,$ we derive 
\begin{align*}
&\quad -\mathbb F^R(t,\rho,x,p)+\mathbb F^R(s,\mu,y,p)\\&\leq \sup_{\mathbb V\in B_{\ell}}\{-\langle p,\mathbb V\rangle +\phi_{R,0}(\rho,x)F(t,\rho,x,\mathbb V)+\langle p,\mathbb V\rangle+\phi_{R,0}(\mu,y)F(s,\mu,y,\mathbb V)\}\\
&\leq \sup _{\mathbb V\in B_{\ell}}\{(\phi_{R,0}(\rho,x)-\phi_{R,0}(\mu,y))F(t,\rho,x,\mathbb V)+\phi_{R,0}(\mu,y)(F(t,\rho,x,\mathbb V)-F(s,\mu,y,\mathbb V))\}\\
&\leq C(R,r_1,r_2,\ell)(\|\rho-\mu\|+\|x-y\|+|t-s|),\end{align*}
when 
$\rho,\mu\in\mathcal P_{r_1}(G),\max_{(i,j)\in E}(|x_i-x_j|\vee|y_i-y_j|)\leq r_2.$ 
 Similarly, we can derive $$\mathbb F^R(t,\rho,x,p)-\mathbb F^R(s,\mu,y,p)\leq C (R,r_1,r_2,\ell)(\|\rho-\mu\|+\|x-y\|+|t-s|).$$ Thus \eqref{xy-regularity} is proved. 
\end{proof}

\subsection{Proof of uniqueness of viscosity solution}\label{sec_6-2}

In this subsection, we apply the truncation technique to prove the uniqueness of the viscosity solution to the HJB equation \eqref{HJB} in $(0,T]\times\mathcal P^{\circ}(G)\times\mathbb R^n$.  

\begin{proof}[Proof of Theorem \ref{uniqueness}]  The existence of the viscosity solution of the HJB equation \eqref{HJB} is guaranteed by Theorem \ref{thm_existence}, where the viscosity solution can be described by the value function \eqref{def-valuefunction} of the optimal control problem. Moreover, when costs $F,h$ are uniformly bounded    (i.e., $p_1=0$ in \eqref{ass11}), according to \eqref{growth-proof}, one can show that this viscosity solution is also bounded.

Now we prove the uniqueness of the bounded viscosity solution. Let $U,\widetilde U$ be two bounded viscosity solutions of \eqref{HJB} in $(0,T)\times\mathcal P^{\circ}(G)\times \mathbb R^n$. Suppose there exists a point $(\hat t,\hat\rho,\hat x)\in(0,T)\times\mathcal P^{\circ}(G)\times\mathbb R^n$ such that 
\begin{align}\label{hatpoint}\sigma_0:=(U-\widetilde U)(\hat t,\hat\rho,\hat x)>0.
\end{align} 
Let $R>\mathcal H_0(\hat\rho,\hat x).$  
  Then $U\phi_{R,0},\widetilde U\phi_{R,0}$ are viscosity solutions of \eqref{HJB_R} in $\mathcal A_R$ (see Lemma \ref{coro2_local}). Here, we recall definitions of $\mathcal A_R,\phi_{R,0}$ in \eqref{def-A-phi} and properties of $\phi_{R}$ in \eqref{phi-property} and \eqref{phi-property2}. 
Next, the proof is split into four steps, based on the doubling of variables method.  
  
  \textit{Step 1: Doubling of variables and properties of maximizer for an auxiliary function.} 
 
In this step, we first introduce an auxiliary function, and then show the existence and properties of its maximizer.
 For fixed $b_0\in(0,\frac12), \delta,\gamma,a,d\in(0,\frac12),$
define the auxiliary function 
\begin{align}\label{auxi-1}
&\Psi(t,s,\rho,\mu,x,y)=U(t,\rho,x)\phi_{R,0}(\rho,x)-\widetilde U(s,\mu,y)\phi_{R,0}( \mu,y)-\frac{(t-s)^2}{2\delta}\notag\\&\quad +\frac{b_0\sigma_0}{2T}(t+s)-\frac dt-\frac ds
-{\Big(a(\|x\|^2+\|y\|^2)+\frac{\|\rho-\mu\|^2+\|x-y\|^2}{2\gamma}\Big)}.\end{align}   
Note that when $a< 
\frac{b_0\sigma_0\hat t }{8T(\|\hat x\|^2\vee 0)}$  
and $d< \frac{b_0\sigma_0\hat t^2 }{8T}$ with $\hat t,\hat x$ given in \eqref{hatpoint}, it holds $\frac{b_0\sigma_0}{2T}\hat t-2a\|\hat x\|^2>\frac{b_0\sigma_0}{4T}\hat t>0$ and $ \frac{b_0\sigma_0}{2T}\hat t-\frac{2d}{\hat t}>\frac{b_0\sigma_0}{4T}\hat t>0$, and thus 
\begin{align}\label{larger0}
\Psi(\hat t,\hat t,\hat\rho,\hat\rho,\hat x,\hat x)&\ge  \phi_{R,0}(\hat\rho,\hat x)(U-\widetilde U)(\hat t,\hat x,\hat\rho)+\frac{b_0\sigma_0}{T}\hat t-\frac{2d}{\hat t}-2a\|\hat x\|^2
\notag\\&= \sigma_0+\frac{b_0\sigma_0}{T}\hat t-\frac{2d}{\hat t}-2a\|\hat x\|^2>\sigma_0+\frac{b_0\sigma_0}{2T}\hat t>\sigma_0,
\end{align} 
where we use $\phi_{R,0}(\hat \rho,\hat x)=1.$ 
For points $(\rho,x),(\mu,y)\notin \mathcal A_R,$ by \eqref{phi-property2} and the boundedness of $U$ on $\mathcal A^{c}_R=\{(t,\rho,x): t\in[0,T],\mathcal H_0(\rho,x)\ge  2R\},$ we have \begin{align}\label{smaller0}
\Psi(t,s,\rho,\mu,x,y)\leq CR^{-\beta}+\frac{b_0\sigma_0}{2T}(t+s)-\frac{(t-s)^2}{2\delta}-\frac dt-\frac ds\leq \frac12\sigma_0,
\end{align} where in the last step we use $b_0<\frac12$ and let $R>R(\sigma_0,\beta)$ be sufficiently large.  
Hence, the maximizer of $\Psi$, denoted by $\bar z=(\bar t,\bar s,\bar \rho,\bar \mu,\bar x,\bar y)$ (depending on $d,a,\gamma,\delta$) if exists, it must satisfy $(\bar\rho,\bar x),(\bar\mu,\bar y)\in \mathcal A_R.$ 
 
Since $\Psi\to-\infty$ as $\|x\|\to\infty$ or $\|y\|\to\infty$, there exists some $C_0:=C_0(a,R)>0$ such that  \begin{align}\label{small11}
\Psi(t,s,\rho,\mu,x,y)\leq C(R)+\sigma_0-a(\|x\|^2+\|y\|^2)<0\text{  for }\|x\|\vee \|y\|\ge C_0.
\end{align} Similarly, $\Psi\to-\infty$ as $t\to0$ or $s\to0$ uniformly with $\rho,\mu,x,y,$ which implies that there exists $c:=c(d)>0$ such that 
\begin{align}\label{small02}
\Psi(t,s,\rho,\mu,x,y)<0 \text{ when } t\wedge s\leq c.
\end{align} 
Hence by the continuity of $\Psi$, the maximizer $\bar z$ exists on the compact set $\mathcal A_{R,C_0,c}:=([c,T]\times(\bar{\mathcal A}_R))^2\cap (\mathcal A_{R,C_0})^2,$ and satisfies 
\begin{align}\label{max-xy}
\|\bar x\|\vee \|\bar y\|\leq C_0\text{ and }\bar t\wedge \bar s\ge c,
\end{align}  
where we recall the definition of $\mathcal A_{R,C_0}$ in \eqref{malcalA-def}.   

Moreover, from \eqref{smaller0}, \eqref{small11}, \eqref{small02} and \eqref{larger0}, we deduce that 
the maximizer is always in the interior of $\mathcal A_{R,C_0,c}$, except the potential occurrence of $\bar t=T$ or $\bar s=T,$ i.e., 
\begin{align}\label{interior}\bar z\in\mathcal A_{R,C_0,c}^{\circ}&\cup \Big(\{(\{t=T\}\times\bar{\mathcal A}_R)\times(  \{s\in[c,T]\}\times\bar{\mathcal A}_R)\}\cap\mathcal A^2_{R,C_0}\Big)\notag\\
&\cup \Big(\{(\{s=T\}\times\bar{\mathcal A}_R) \times(\{t\in[c,T]\}\times\bar{\mathcal A}_R)\}\cap\mathcal A^2_{R,C_0}\Big).
\end{align} 
We would like to mention that we will remove the case $\bar t=T$ or $\bar s=T$ at the end of \textit{Step 1}. Furthermore, we can obtain a stronger result, namely, the maximizer lies in the interior $\mathcal A_{R,C_0}^{\circ}$ within a small neighborhood of radius $\epsilon_1 > 0$.  Indeed,  
 by the fact that 
on $\mathcal A_{R,C_0,c}\cap(\partial\mathcal A_{R,C_0})^2$ it holds $\Psi(t,s,\rho,\mu,x,y)\leq \frac12\sigma_0$ due to \eqref{smaller0}, we have that for the number $\frac{\sigma_0}{2}>0,$ 
there exists a constant $\epsilon_1>0,$ such that 
  \begin{align}\label{epsilon1-def}
B_{\epsilon_1}(\bar t,\bar \rho,\bar x), B_{\epsilon_1}(\bar s,\bar\mu,\bar y)\in \mathcal A^{\circ}_{R,C_0},\text{ i.e., }(\bar t,\bar \rho,\bar x),(\bar s,\bar\mu,\bar y)\in(\mathcal A_{R,C_0})_{\epsilon_1},
\end{align}  where we recall the definition of $(\mathcal A_{R,C_0})_{\epsilon_1}$ in \eqref{shrunken}, and $B_{\epsilon_1}(\bar t,\bar \rho,\bar x):=\{(t,\rho,x):|t-\bar t|+\|\rho-\bar\rho\|+\|x-\bar x\|\leq \epsilon_1\}.$
 In addition, by  Assumption \ref{energy},  
there exist constants $C_1:=C_1(R)\in(0,1)$ and $C_2:=C_2(R)>0$  such that \begin{align}
&\min\{\min_{i=1,\ldots,n}\bar\rho_i,\min_{i=1,\ldots,n}\bar\mu_i\}\ge {C_1},\quad \text{namely, }\bar\rho,\bar\mu\in\mathcal P_{C_1}(G),\label{bounded-rho}\\
&\max_{(i,j)\in E}(\bar y_i-\bar y_j) \vee\max_{(i,j)\in E}(\bar x_i-\bar x_j)\leq \sqrt{\frac{2R}{\min\limits_{(i,j)\in E}\omega_{ij}\min_{i=1,\ldots,n}\bar\rho_i}}\leq C_2. \label{bounded-xy}
\end{align}

Next, we establish some smallness properties at the maximum point $\bar z$, which will be employed in \textit{Step 4} to obtain some relevant estimates. 
 To this end, define \begin{align*}
&M_{a,\gamma,\delta}:=\Psi(\bar t,\bar s,\bar\rho,\bar\mu,\bar x,\bar y)=\sup_{(t,s,\rho,\mu,x,y)\in((0,T)\times \mathcal P^{\circ}(G)\times \mathbb R^n)^2}\Psi(t,s,\rho,\mu,x,y),\\ 
&M_{a,\gamma}:=\sup_{(t,\rho,\mu,x,y)\in(0,T)\times(\mathcal P^{\circ}(G)\times\mathbb R^n)^2}\Psi( t, t,\rho,\mu,x,y),\\& M_{a}:=\sup_{(t,\rho,x)\in (0,T)\times \mathcal P^{\circ}(G)\times \mathbb R^n} \Psi(t,t,\rho,\rho,x,x).
\end{align*}

\textit{Claim: It holds that 
\begin{align}
&\lim_{\delta\to0}\frac{(\bar t-\bar s)^2}{\delta}=0, \quad \lim_{\gamma\to0}\limsup_{\delta\to0}\frac{\|\bar\rho-\bar\mu\|^2+\|\bar x-\bar y\|^2}{\gamma}=0,\label{max-relation1}\\&\lim_{a\to0}\limsup_{\gamma,\delta\to0}a(\|\bar x\|^2+\|\bar y\|^2)=0.\label{max-relation2}
\end{align}}

By the relation
\begin{align*}
M_{a,\gamma,\delta}+\frac{(\bar t-\bar s)^2}{4\delta}=\Psi(\bar t,\bar s,\bar\rho,\bar\mu,\bar x,\bar y)+\frac{(\bar t-\bar s)^2}{4\delta}\leq M_{a,\gamma,2\delta},
\end{align*}
and $\lim_{\delta\to0}M_{a,\gamma,\delta}$ exists due to the monotone convergence theorem, 
we deduce the first relation in \eqref{max-relation1}. 
 Similar to the proof of \cite[Eq. (23)]{dang2025graph}, it holds $\lim_{\delta\to0}M_{a,\gamma,\delta}=M_{a,\gamma}.$
By the relation 
$$M_{a,\gamma,\delta}+\frac{\|\bar \rho-\bar\mu\|^2+\|\bar x-\bar y\|^2}{4\gamma}
+\frac{(\bar t-\bar s)^2}{4\delta}\leq M_{a,2\gamma,2\delta},$$ and taking supremum limit with $\delta\to0,$ we derive 
$$M_{a,\gamma}+\limsup_{\delta\to0}\frac{\|\bar\rho-\bar\mu\|^2+\|\bar x-\bar y\|^2}{4\gamma}\leq M_{a,2\gamma}.$$ This, together with the fact that $\lim_{\gamma\to0}M_{a,\gamma}$ exists, yields 
\begin{align*}
\lim_{\gamma\to0}\limsup_{\delta\to0}\frac{\|\bar\rho-\bar\mu\|^2+\|\bar x-\bar y\|^2}{4\gamma}=0.
\end{align*}
The proof of \eqref{max-relation2} is analogous and therefore omitted.

From the above argument, we can also derive that $\bar t,\bar s\neq T.$ Indeed, if either $\bar t$ or $\bar s$  was equal to $T$, then by \eqref{max-relation1} and \eqref{max-relation2}, one could choose $\delta,\gamma,a,$ and $d$ sufficiently small such that $\Psi(\bar z)<\sigma_0$, which contradicts \eqref{larger0}. This combining with \eqref{interior}, we deduce that $\bar z\in\mathcal A^{\circ}_{R,C_0,c}.$

\textit{Step 2: Perturb  $\Psi$ to obtain a twice differentiable maximizer.} 

In this step, we perturb $\Psi$ in order to ensure the existence of a maximizer at which the perturbed function is twice differentiable. This property will be used in \textit{Step 3} to derive certain relations at the maximum point.  
 Since $\Psi$ may not be differentiable at the maximizer point due to the fact that $U,\widetilde U$ may not be differentiable. We consider the corresponding  semiconvex and semiconcave approximations of $U\phi_{R,0},\widetilde U\phi_{R,0}$ on the set $\mathcal A_{R,C_0,c}$,  
denoted by $(U\phi_{R,0})^{\vartheta}$ and $(\widetilde U\phi_{R,0})_{\vartheta}$ respectively, where $\vartheta\in(0,1).$  
Then we have $((U\phi_{R,0})^{\vartheta}-(\widetilde U\phi_{R,0})_{\vartheta})(\hat t,\hat\rho,\hat x)\ge (U\phi_{R,0}-\widetilde U\phi_{R,0})(\hat t,\hat\rho,\hat x)=\sigma_0>0.$ 

For each $\vartheta\in(0,1),$ define the auxiliary function 
\begin{align*}
&\Psi^{\vartheta}(t,s,\rho,\mu,x,y)=(U\phi_{R,0})^{\vartheta}(t,\rho,x)-(\widetilde U\phi_{R,0})_{\vartheta}(s,\mu,y)-\frac{(t-s)^2}{2\delta}+\frac{b_0\sigma_0}{2T}(t+s)\\
&\quad -\frac{d}{t}-\frac{d}{s}
-\Big(a(\|x\|^2+\|y\|^2)+\frac{\|\rho-\mu\|^2+\|x-y\|^2}{2\gamma}\Big).\end{align*} 
 Recalling that $(\mathcal A_{R,C_0,c})^2$ is a compact set, thus for each $\vartheta\in(0,1),$ the continuous function $\Psi^{\vartheta}$ has a maximizer in this compact set, which is denoted by $\bar z^{\vartheta}$. It follows from the Bolzano--Weierstrass theorem that there exists a subsequence $\bar z^{\vartheta_k}$ with $\vartheta_k\to0$ as $k\to\infty$ such that $\bar z^{\vartheta_k}\to z_0$ with some $z_0\in(\mathcal A_{R,C_0,c})^2.$

 Recall that $\bar z$ is the maximizer of $\Psi,$ which implies $\Psi(\bar z)\ge \Psi(z_0).$  We have the following claim. 
 
 \textit{Claim: It holds that \begin{align}\label{maximim-z0}
\Psi(\bar z)\ge \Psi(z_0)=\lim_{k\to\infty} \Psi^{\vartheta_k}(\bar z^{\vartheta_k})
\ge \Psi(\bar z). 
\end{align}} 

In fact, on the  one hand, due to the maximum property of $\bar z^{\vartheta_k}$ of $\Psi^{\vartheta_k}, $ we have 
$$\Psi^{\vartheta_k}(\bar z^{\vartheta_k})\ge \Psi^{\vartheta_k}(\bar z)\ge \Psi(\bar z),$$ where in the last inequality, we use $$(U\phi_{R,0})^{\vartheta}(\bar t,\bar \rho,\bar x)\ge (U\phi_{R,0})(\bar t,\bar \rho,\bar x), \quad (\widetilde U\phi_{R,0})_{\vartheta}(\bar s,\bar \mu,\bar y)\leq (\widetilde U\phi_{R,0})(\bar s,\bar \mu,\bar y),$$  (see \eqref{ine23}). On the other hand, by the uniform convergence of $(U\phi_{R,0})^{\vartheta}$ to $U\phi_{R,0}$ (Lemma \ref{uniform-converge}) and the continuity of $U\phi_{R,0}$, we obtain \begin{align*}
&\quad |(U\phi_{R,0})^{\vartheta_k}(\bar z^{\vartheta_k})-U\phi_{R,0}(z_0)|\\
&\leq |(U\phi_{R,0})^{\vartheta_k}(\bar z^{\vartheta_k})-(U\phi_{R,0})(\bar z^{\vartheta_k})|+|(U\phi_{R,0})(\bar z^{\vartheta_k})- U\phi_{R,0}(z_0)|\\
&\leq \sup_{z\in\mathcal A_{R,C_0}}|(U\phi_{R,0})^{\vartheta_k}(z)-(U\phi_{R,0})(z)|+ |(U\phi_{R,0})(\bar z^{\vartheta_k})- U\phi_{R,0}(z_0)|\to0\text{ as }k\to\infty;
\end{align*}  And similarly, $|(\widetilde U\phi_{R,0})_{\vartheta_k}(\bar z^{\vartheta_k})- (\widetilde U\phi_{R,0})(z_0)|\to0$ as $k\to\infty.$ These implies that $\Psi^{\vartheta_k}(\bar z^{\vartheta_k})\to\Psi(\bar z_0)$ as $k\to\infty.$  As a result, the claim \eqref{maximim-z0} is proved.

From  \eqref{maximim-z0} we observe that $z_0$ is a maximizer of $\Psi.$ In general, the maximizer of $\Psi$ may not be unique, i.e., $z_0$ may not equal $\bar z$, but they both admit properties characterized in \textit{Step 1}. Here, for notational simplicity, we write $\vartheta\to0$ instead of the subsequence $\vartheta_k\to0$, and write $z_0=\bar z,$ i.e., we have 
\begin{align}\label{zlimit}
\lim_{\vartheta\to0}\bar z^{\vartheta}=\bar z.
\end{align} 
From \eqref{epsilon1-def}, for $\frac{\epsilon_1}{2}>0,$ there exists $\vartheta_0$ such that $\vartheta<\vartheta_0,$ it holds that $\bar z^{\vartheta}\in B_{\epsilon_1/2}(\bar z)\subset (\mathcal A^{\circ}_{R, C_0,c})_{\frac{\epsilon_1}{2}}^2$. 

Besides, we can always suppose the maximizer of $\Psi^{\vartheta}$ is strict with respect to $(x,y)$ variables. Otherwise we can add a purterbation term to $\Psi^{\vartheta}$ and consider: for small number $r>0,$ $$\Psi^{\vartheta}(t,s,\rho,\mu,x,y)-r(\|x-\bar x^{\vartheta}\|^2+\|y-\bar y^{\vartheta}\|^2),$$ which takes the strict maximum at $\bar z^{\vartheta}.$ 
We note that $\Psi^{\vartheta}(\bar t^{\vartheta},\bar s^{\vartheta},\bar\rho^{\vartheta},\bar\mu^{\vartheta},\cdot,\cdot)$ is twice differentiable with respect to $(x,y)$ variables almost everywhere due to Lemma \ref{Alex}, but it 
may not be twice differentiable at $(\bar x^{\vartheta},\bar y^{\vartheta}).$ On the set $\mathcal A_{R,C_0,c},$ the function $\Psi^{\vartheta}$ is semiconvex, and $\bar z^{\vartheta}$ is a strict maximizer. By Lemma \ref{Jensen},  for the above fixed $r>0,$ there exist numbers $q_0,q_1$ and  vectors $p_i,i=0,1,2,3,$ with 
\begin{align}\label{bounded-pq}
|q_0|^2+|q_1|^2+\sum_{i=0}^3\|p_i\|^2<r^2
\end{align} such that
\begin{align}\label{auxi-2}
&\tilde\Psi^{\vartheta}:=\Psi^{\vartheta}-q_0(t-\bar t^{\vartheta})-q_1(s-\bar s^{\vartheta})-p_0^{\top}
(x-\bar x^{\vartheta})-p^{\top}_1(y-\bar y^{\vartheta})-p_2^{\top}(\rho-\bar\rho^{\vartheta})-p_3^{\top}(\mu-\bar\mu^{\vartheta})
\end{align} takes local  maximum at 
\begin{align}
\label{zlimit2}
\bar z^{{\vartheta},r}:=(\bar t^{{\vartheta},r},\bar s^{{\vartheta},r},\bar\rho^{{\vartheta},r},\bar\mu^{{\vartheta},r},\bar x^{{\vartheta},r},\bar y^{{\vartheta},r})\in B_r(\bar z^{\vartheta})
\end{align} 
and is twice differentiable with respect to $(x,y)$ variables. 

\textit{Step 3: Some relations at the twice differentiable maximum point.}

In this step, we present certain equalities for the first-order derivatives and an inequality for the second-order derivative of the auxiliary function $\tilde\Psi^{\vartheta}$ at its maximum point.   
From \cite[Corollary 3.17]{MCC}, we know that at the maximum point $\bar z^{\vartheta,r}$, the Fr\'echet derivatives of the mapping $\tilde\Psi^{\vartheta}$ with respect to $\rho,\mu$ variables satisfy $$\partial_{\rho}\tilde\Psi^{\vartheta}(\bar\rho^{\vartheta,r})=0, \quad \partial_{\mu}\tilde\Psi^{\vartheta}(\bar\mu^{\vartheta,r})=0.$$ And when $(\bar t^{\vartheta,r},\bar s^{\vartheta,r},\bar\rho^{\vartheta,r},\bar\mu^{\vartheta,r})$ is fixed, the second-order derivative of $\tilde\Psi^{\vartheta}$ with respect to $(x,y)$ variables satisfy $D^2\tilde\Psi^{\vartheta}\leq 0,$  which means that $D^2\tilde\Psi^{\vartheta}$ is a $2n\times 2n$ non-positive definite symmetric matrix. 
As a result, 
we  have that the first-order derivatives of $\tilde\Psi^{\vartheta}$ with $t,s,x,y,\rho,\mu$ variables  satisfy 
\begin{align}
&\frac{\partial}{\partial t}(U\phi_{R,0})^{\vartheta}-\frac{\bar t^{\vartheta,r}-\bar s^{\vartheta,r}}{\delta}+\frac{b_0\sigma_0}{2T}+\frac {d}{(\bar t^{\vartheta,r})^2}-q_0=0,\label{equality-s}\\
& - \frac{\partial}{\partial s}(\widetilde U\phi_{R,0})_{\vartheta}-\frac{\bar s^{\vartheta,r}-\bar t^{\vartheta,r}}{\delta}+\frac{b_0\sigma_0}{2T}+\frac {d}{(\bar s^{\vartheta,r})^2}-q_1=0,\label{equality-t}\\
& D_x(U\phi_{R,0})^{\vartheta}-(\frac{ \bar x^{\vartheta,r}-\bar y^{\vartheta,r}}{\gamma}+2a \bar x^{\vartheta,r})
-p_0=0,\label{equality-x}\\
& -D_y(\widetilde U\phi_{R,0})_{\vartheta}-(\frac{ \bar y^{\vartheta,r}-\bar x^{\vartheta,r}}{\gamma}+2a \bar y^{\vartheta,r})-p_1 =0,\label{equality-y}\\
&\partial_{\rho}(U\phi_{R,0})^{\vartheta}-\frac{1}{\gamma}\sum_{j=1}^n\hat e_j(\bar\rho^{\vartheta,r}_j-\bar\mu^{\vartheta,r}_j)
-\sum_{j=1}^n\hat e_jp_{2,j}
=0,\label{equality-rho}\\
&-\partial_{\mu}(\widetilde U\phi_{R,0})_{\vartheta}-\frac{1}{\gamma} \sum_{j=1}^n\hat e_j(\bar\mu^{\vartheta,r}_j-\bar\rho^{\vartheta,r}_j)
-\sum_{j=1}^n\hat e_jp_{3,j}
=0.\label{equality-mu}
\end{align}
In \eqref{equality-rho} and \eqref{equality-mu}, we use \eqref{Fre2} to obtain  $\partial_{\rho}\|\rho-\mu\|^2=
2\sum_{j=1}^n\hat e_j(\rho_j-\mu_j),\;
\partial_{\mu}\|\rho-\mu\|^2=-\partial_{\rho}\|\rho-\mu\|^2$ for $\rho,\mu\in\mathcal P^{\circ}(G)$.
And the second-order derivatives of $\tilde\Psi^{\vartheta}$  with $x,y$ variables satisfy 
\begin{equation}\label{equality-xxyy}
\begin{pmatrix}
D^2_x(U\phi_{R,0})^{\vartheta} & 0 \\
0 & -D^2_y(\widetilde U\phi_{R,0})_{\vartheta}  \\
\end{pmatrix}\leq Y.
\end{equation}
Here, $Y$ is a symmetric matrix $Y=[Y_{1,1},Y_{1,2};Y_{1,2},Y_{2,2}]$ defined by 
\begin{align} \label{compoY1}
Y_{1,1}=(\frac{1}{\gamma}+2a)\mathrm{Id}_{n}, \;
Y_{1,2}= -\frac1\gamma\mathrm{Id}_n,\;
Y_{2,2}= (\frac{1}{\gamma}+2a)\mathrm{Id}_{n}. 
\end{align}

\textit{Step 4: Viscosity inequalities and estimates.}

In this step, we invoke the viscosity sub- and super-solution properties of the semiconvex and semiconcave approximations, respectively, to derive the corresponding viscosity inequalities. Building on the relations obtained in the previous steps, we then proceed to establish the key upper and lower bound estimates needed for the uniqueness argument.

By Lemma \ref{lem1}, 
we have that 
$(U\phi_{R,0})^{\vartheta},(\widetilde U\phi_{R,0})_{\vartheta}$ are viscosity sub- and super- solutions in $(\mathcal A_{R,C_0})_{\sqrt \vartheta k_0}$ with $(H^R_{U})^\vartheta$ and  $(H^R_{\widetilde U})_{\vartheta},$ respectively. 
Thus we obtain 
\begin{align*}
&\Big[\frac{\partial}{\partial t}(\phi_{R,0} U)^{\vartheta}+(H^R_{U})^{\vartheta}(\cdot,\partial_{\rho}(U\phi_{R,0})^\vartheta,D_x(U\phi_{R,0})^{\vartheta},D^2_x(U\phi_{R,0})^{\vartheta})\Big]\Big|_{(\bar t^{\vartheta,r},\bar\rho^{\vartheta,r},\bar x^{\vartheta,r})}\ge 0,\\
&  \Big[\frac{\partial}{\partial s}(\phi_{R,0} \widetilde U)_{\vartheta}+(H^R_{\widetilde U})_{\vartheta}(\cdot,\partial_{\mu}(\widetilde U \phi_{R,0})_{\vartheta},D_y(\widetilde U\phi_{R,0})_{\vartheta},D^2_y(\widetilde U\phi_{R,0})_{\vartheta})\Big]\Big|_{(\bar s^{\vartheta,r},\bar\mu^{\vartheta,r},\bar y^{\vartheta,r})}\leq 0.
\end{align*}
Hence  
\begin{align}\label{estimate-main}
&\quad \big[\frac{\partial}{\partial s}(\phi_{R,0}\widetilde U)_{\vartheta}-\frac{\partial}{\partial t}(\phi_{R,0}U)^{\vartheta}\big]|_{\bar z^{\vartheta,r}}\notag\\
&\leq \big[(H^R_{U})^{\vartheta}(\cdot,\partial_{\rho}(U\phi_{R,0})^{\vartheta},D_x(U\phi_{R,0})^{\vartheta},D^2_x(U\phi_{R,0})^{\vartheta})\notag\\
&\quad -(H^R_{\widetilde U})_{\vartheta}(\cdot,\partial_{\mu}(\widetilde U\phi_{R,0})_{\vartheta},D_y(\widetilde U\phi_{R,0})_{\vartheta},D^2_y(\widetilde U\phi_{R,0})_{\vartheta})\big]|_{(\bar z^{\vartheta,r})}\notag\\
&= H^R_{U}(\bar t^{\vartheta,r}_0,\bar \rho^{\vartheta,r}_0, \bar x^{\vartheta,r}_0,(\partial_{\rho}(U\phi_{R,0})^{\vartheta},D_x(U\phi_{R,0})^{\vartheta},D^2_x(U\phi_{R,0})^{\vartheta})(\bar t^{\vartheta,r},\bar \rho^{\vartheta,r}, \bar x^{\vartheta,r})) \notag\\
&\quad- H^R_{\widetilde U}(\bar s^{\vartheta,r}_0,\bar\mu^{\vartheta,r}_0, \bar y^{\vartheta,r}_0,(\partial_{\mu}(\widetilde U\phi_{R,0})_{\vartheta},D_y(\widetilde U\phi_{R,0})_{\vartheta},D^2_y(\widetilde U\phi_{R,0})_{\vartheta})(\bar s^{\vartheta,r},\bar\mu^{\vartheta,r}, \bar y^{\vartheta,r})).
\end{align} 
Here,   
we first recall the definitions of $H^R_U$ in \eqref{def_H4},  $(H^R_U)^{\vartheta}$ in \eqref{Hvartheta}, and $(H^R_{\widetilde U})_{\vartheta}$ in \eqref{Hvartheta2}.    In the last step of \eqref{estimate-main}, we take points $(\bar t^{\vartheta,r}_0,\bar \rho^{\vartheta,r}_0, \bar x^{\vartheta,r}_0)$ and $(\bar s^{\vartheta,r}_0,\bar\mu^{\vartheta,r}_0, \bar y^{\vartheta,r}_0)$ such that 
\begin{align*}
&\quad (H^R_{U})^{\vartheta}(\bar t^{\vartheta,r},\bar \rho^{\vartheta,r}, \bar x^{\vartheta,r},\partial_{\rho}(U\phi_{R,0})^{\vartheta},D_x(U\phi_{R,0})^{\vartheta},D^2_x(U\phi_{R,0})^{\vartheta})\\
&= H^R_{U}(\bar t^{\vartheta,r}_0,\bar \rho^{\vartheta,r}_0, \bar x^{\vartheta,r}_0,(\partial_{\rho}(U\phi_{R,0})^{\vartheta},D_x(U\phi_{R,0})^{\vartheta},D^2_x(U\phi_{R,0})^{\vartheta})(\bar t^{\vartheta,r},\bar \rho^{\vartheta,r}, \bar x^{\vartheta,r})),
\end{align*}
and 
\begin{align*}
&\quad (H^R_{\widetilde U})_{\vartheta}(\bar s^{\vartheta,r},\bar\mu^{\vartheta,r}, \bar y^{\vartheta,r},\partial_{\mu}(\widetilde U\phi_{R,0})_{\vartheta},D_y(\widetilde U\phi_{R,0})_{\vartheta},D^2_y(\widetilde U\phi_{R,0})_{\vartheta})\\
&=
H^R_{\widetilde U}(\bar s^{\vartheta,r}_0,\bar\mu^{\vartheta,r}_0, \bar y^{\vartheta,r}_0,(\partial_{\mu}(\widetilde U\phi_{R,0})_{\vartheta},D_y(\widetilde U\phi_{R,0})_{\vartheta},D^2_y(\widetilde U\phi_{R,0})_{\vartheta})(\bar s^{\vartheta,r},\bar\mu^{\vartheta,r}, \bar y^{\vartheta,r})),
\end{align*}
respectively. 
  According to the definition of $(H^R_U)^{\vartheta}$ and $(H^R_{\widetilde U})_{\vartheta}$ (see \eqref{Hvartheta} and \eqref{Hvartheta2}), 
the two points $(\bar t^{\vartheta,r}_0,\bar \rho^{\vartheta,r}_0, \bar x^{\vartheta,r}_0)$ and $(\bar s^{\vartheta,r}_0,\bar\mu^{\vartheta,r}_0, \bar y^{\vartheta,r}_0)$
 satisfy 
 \begin{align}\label{relation-twopoint}
&|\bar t^{\vartheta,r}-\bar t^{\vartheta,r}_0|+\|
\bar \rho^{\vartheta,r}-\bar \rho^{\vartheta,r}_0\|+
\|\bar x^{\vartheta,r}-\bar x^{\vartheta,r}_0\|
\leq \sqrt{\vartheta},\notag\\&|\bar s^{\vartheta,r}-\bar s^{\vartheta,r}_0|+\|
\bar \mu^{\vartheta,r}-\bar \mu^{\vartheta,r}_0\|+ \| \bar y^{\vartheta,r}-\bar y^{\vartheta,r}_0\|\leq \sqrt{\vartheta}.
\end{align}

It follows from \eqref{equality-s} and \eqref{equality-t} that the left-hand side of \eqref{estimate-main} is 
\begin{align}\label{left-estimate}&\quad \frac{\partial}{\partial s}(\phi_{R,0}\widetilde U)_{\vartheta}\big|_{(\bar s^{\vartheta,r},\bar\mu^{\vartheta,r},\bar y^{\vartheta,r})}-\frac{\partial}{\partial t}(\phi_{R,0}U)^{\vartheta}\big|_{(\bar t^{\vartheta,r},\bar\rho^{\vartheta,r},\bar x^{\vartheta,r})}\notag\\&=\frac{b_0\sigma_0}{T}-q_0-q_1+d(\frac{1}{(\bar t^{\vartheta,r})^2}+\frac{1}{(\bar s^{\vartheta,r})^2}).
\end{align} 
 The right-hand side of \eqref{estimate-main} can be split further as 
$\sum_{i=1}^6
II_i,
$
where \begin{align*}
&II_1=\langle \partial_{\rho}(U\phi_{R,0})^{\vartheta}(\bar\rho^{\vartheta,r},\bar x^{\vartheta,r}),D_x\mathcal H_0 (\bar\rho^{\vartheta,r}_0,\bar x^{\vartheta,r}_0)\rangle-\langle \partial_{\mu}({\widetilde U}\phi_{R,0})_{\vartheta} (\bar\mu^{\vartheta,r},\bar y^{\vartheta,r}),D_y \mathcal H_0(\bar\mu^{\vartheta,r}_0, \bar y^{\vartheta,r}_0)\rangle,\\
&II_2=-\langle D_x(U\phi_{R,0})^{\vartheta}(\bar\rho^{\vartheta,r},\bar x^{\vartheta,r}),D_{\rho}\mathcal H_0(\bar\rho^{\vartheta,r}_0,\bar x^{\vartheta,r}_0)\rangle + \langle D_y({\widetilde U}\phi_{R,0})_{\vartheta}(\bar\mu^{\vartheta,r},\bar y^{\vartheta,r}),D_{\mu}\mathcal H_0(\bar\mu^{\vartheta,r}_0,\bar y^{\vartheta,r}_0)\rangle,\\&II_3=\frac12\mathrm{tr}(\sigma\sigma^{\top}D^2_x(U\phi_{R,0})^{\vartheta}(\bar\rho^{\vartheta,r},\bar x^{\vartheta,r}))-\frac12\mathrm {tr}(\sigma\sigma^{\top}D^2_y({\widetilde U}\phi_{R,0})_{\vartheta}(\bar\mu^{\vartheta,r},\bar y^{\vartheta,r})),\\
&II_4=-\mathbb F^R(\bar t^{\vartheta,r}_0,\bar\rho^{\vartheta,r}_0,\bar   x^{\vartheta,r}_0,D_x(U\phi_{R,0})^{\vartheta}(\bar\rho^{\vartheta,r},\bar x^{\vartheta,r})-UD_x\phi_{R,0}(\bar\rho^{\vartheta,r}_0,\bar x^{\vartheta,r}_0))\\&\qquad\quad  +\mathbb F^R(\bar s^{\vartheta,r}_0,\bar\mu^{\vartheta,r}_0, \bar  y^{\vartheta,r}_0,D_y({\widetilde U}\phi_{R,0})_{\vartheta}(\bar\mu^{\vartheta,r},\bar y^{\vartheta,r})-\widetilde UD_y\phi_{R,0}(\bar\mu^{\vartheta,r}_0,\bar y^{\vartheta,r}_0)),\\
&II_5=-\frac{\sigma\sigma^{\top}D_x\phi_{R,0}}{\phi_{R,0}}(\bar\rho^{\vartheta,r}_0,\bar x^{\vartheta,r}_0)D_x(U\phi_{R,0})^{\vartheta}(\bar\rho^{\vartheta,r},\bar x^{\vartheta,r})\\&\qquad\quad  +\frac{\sigma\sigma^{\top}D_y\phi_{R,0}}{\phi_{R,0}}(\bar\mu^{\vartheta,r}_0,\bar y^{\vartheta,r}_0)D_y(\widetilde U\phi_{R,0})_{\vartheta}(\bar\mu^{\vartheta,r},\bar y^{\vartheta,r}),\\
&II_6=-U(\bar \rho^{\vartheta,r}_0,\bar x^{\vartheta,r}_0)\Big[\frac12\mathrm {tr}(\sigma\sigma^{\top}D^2_x\phi_{R,0})-\frac{\sigma\sigma^{\top}D_x\phi_{R,0}D_x\phi_{R,0}}{\phi_{R,0}}\Big](\bar \rho^{\vartheta,r}_0,\bar x^{\vartheta,r}_0)\\
&+\widetilde U(\bar \mu^{\vartheta,r}_0,\bar y^{\vartheta,r}_0)\Big[\frac12\mathrm {tr}(\sigma\sigma^{\top}D^2_y\phi_{R,0})-\frac{\sigma\sigma^{\top}D_y\phi_{R,0}D_y\phi_{R,0}}{\phi_{R,0}}\Big](\bar \mu^{\vartheta,r}_0,\bar y^{\vartheta,r}_0).
\end{align*} 
We now proceed to present, one by one, the estimates for terms   $II_1,\ldots,II_6$. For notational convenience, we adopt the following convention regarding the order of limits: $r,\vartheta,\gamma\to0$ means that $r\to0$ first, followed by $\vartheta\to0$, and finally $\gamma\to0$. This convention applies similarly to limits involving more than three variables.

\textbf{Estimate of the term $II_1$.} Plugging \eqref{equality-rho} and \eqref{equality-mu} into $II_1,$ 
we split the term $II_1$ further as  
\begin{align*}
II_{1}&=\Big\langle \frac{1}{\gamma}\sum_{j=1}^n\hat e_j(\bar\rho^{\vartheta,r}_j-\bar\mu^{\vartheta,r}_j),D_x \mathcal H_0 (\bar\rho^{\vartheta,r}_0,\bar x^{\vartheta,r}_0)- D_y\mathcal H_0(\bar\mu^{\vartheta,r}_0,\bar y^{\vartheta,r}_0)\Big\rangle\\
&+\Big(\langle\sum_{j=1}^n\hat e_j p_{2,j}, D_x\mathcal H_0(\bar\rho^{\vartheta,r}_0,\bar x^{\vartheta,r}_0)
\rangle+
 \langle \sum_{j=1}^n\hat e_j p_{3,j},D_y \mathcal H_0(\bar\mu^{\vartheta,r}_0,\bar y^{\vartheta,r}_0)\rangle\Big)=:II_{1,1}+II_{1,2}.
\end{align*} 
 By \eqref{zlimit}, \eqref{zlimit2}, \eqref{relation-twopoint}, we have \begin{align}\label{bounded-xy2}
\min_{i\in V}\bar\rho^{\vartheta,r}_{0,i}\ge c_R,\quad \max_{(i,j)\in E}|\bar x^{\vartheta,r}_{0,i}-\bar x^{\vartheta,r}_{0,j}|\leq 1+\max_{(i,j)\in E}|\bar x_{i}-\bar x_{j}|\leq C_R
\end{align} due to \eqref{bounded-xy}. Similarly, \begin{align}\label{bounded-xy3}\min_{i\in V}\bar\mu^{\vartheta,r}_{0,i}\ge c_R,\quad \max_{(i,j)\in E}|\bar y^{\vartheta,r}_{0,i}-\bar y^{\vartheta,r}_{0,j}|\leq 1+\max_{(i,j)\in E}|\bar y_{i}-\bar y_{j}|\leq C_R.
\end{align} Then, using the local Lipschitz property of the partial derivatives of ${\mathcal H_0}$, together with the H\"older inequality and the Young inequality, the term $II_{1,1}$ can be estimated as  
\begin{align*}
II_{1,1}
&\leq C_R\frac{\|\bar\rho^{\vartheta,r}-\bar\mu^{\vartheta,r}\|^2+\|\bar\rho^{\vartheta,r}_0-\bar\mu^{\vartheta,r}_0\|^2 }{\gamma}.
\end{align*}
Then proximity relations of the points \eqref{zlimit}, \eqref{zlimit2}, \eqref{relation-twopoint} yield that 
\begin{align*}
II_{1,1} \leq C_R\frac{\|\bar\rho-\bar\mu\|^2}{\gamma}\text{ as }r,\vartheta\to0.
\end{align*} 
According to \eqref{max-relation1}, by further letting $\gamma\to0$, we  derive $II_{1,1}\to0$.

For the term $II_{1,2}$, by means of \eqref{bounded-xy2}, \eqref{bounded-xy3}, and the local Lipschitz continuity of $D_x\mathcal H_0,$ we have  
\begin{align*}
II_{1,2}&\leq C_R(\|p_2\|+\|p_3\|).
\end{align*}
Then using \eqref{bounded-pq} yields $II_{1,2}\to0\text{ as }r\to0$.  

Gathering eatimates of $II_{1,1}$ and $II_{1,2}$ yields 
\begin{align}\label{conclu-1}
II_1\to0\text{ as }r,\vartheta,\gamma\to0.
\end{align}

\textbf{Estimate of the term $II_2.$}  
Plugging \eqref{equality-x} and \eqref{equality-y} into $II_2,$ we have 
\begin{align*}
II_2&=\langle 
-\big(\frac{\bar x^{\vartheta,r}-\bar y^{\vartheta,r}}{\gamma}+2a \bar x^{\vartheta,r}\big)+p_0,D_{\rho}\mathcal H_0 (\bar\rho^{\vartheta,r}_0,\bar x^{\vartheta,r}_0)\rangle
\\
&+\langle -\big(\frac{\bar y^{\vartheta,r}-\bar x^{\vartheta,r}}{\gamma}+2a \bar y^{\vartheta,r}\big)+p_1,D_{\mu}\mathcal H_0 (\bar\mu^{\vartheta,r}_0,\bar y^{\vartheta,r}_0)\rangle \\
&=-
\Big\langle \frac{\bar x^{\vartheta,r}-\bar y^{\vartheta,r}}{\gamma},D_{\rho}\mathcal H_0 (\bar\rho^{\vartheta,r}_0,\bar x^{\vartheta,r}_0) -D_{\mu}\mathcal H_0 (\bar\mu^{\vartheta,r}_0,\bar y^{\vartheta,r}_0) \Big\rangle\\&+\Big(-\langle 2a\bar x^{\vartheta,r},D_{\rho}\mathcal H_0(\bar \rho^{\vartheta,r}_0,\bar  x^{\vartheta,r}_0)\rangle - \langle 2a\bar y^{\vartheta,r},D_{\mu}\mathcal H_0(\bar \mu^{\vartheta,r}_0,\bar  y^{\vartheta,r}_0)\rangle\Big)
\\&-\Big(\langle p_0,D_{\rho}\mathcal H_0(\bar \rho^{\vartheta,r}_0,\bar  x^{\vartheta,r}_0) 
\rangle + \langle p_1,D_{\mu}\mathcal H_0(\bar \mu^{\vartheta,r}_0,\bar  y^{\vartheta,r}_0)\rangle\Big)\\
&=:II_{2,1}+II_{2,2}+II_{2,3}.
\end{align*} 
By the local Lipschitz property of the Euclidean gradient $D_{\rho}\mathcal H_0$ (Assumption \ref{energy} and Definition \ref{def-functions}), \eqref{bounded-xy2}, \eqref{bounded-xy3}, and the Young inequality,  we derive that 
\begin{align*}
II_{2,1}
&\leq C_R\frac{\|\bar x^{\vartheta,r}-\bar y^{\vartheta,r}\|^2+\| \bar x^{\vartheta,r}_0-\bar y^{\vartheta,r}_0\|^2+\| \bar \rho^{\vartheta,r}_0-\bar \mu^{\vartheta,r}_0\|^2}{\gamma}.
\end{align*}
Then by virtue of \eqref{zlimit}, \eqref{zlimit2}, and  \eqref{relation-twopoint}, we derive 
\begin{align*}
II_{2,1}\leq C_R\frac{\| \bar x-\bar y\|^2+\| \bar \rho-\bar \mu\|^2}{\gamma}\text{ as }r\to0,\vartheta\to0.
\end{align*} 
By further letting $\gamma\to0$ and
utilizing  \eqref{max-relation1} lead to 
$II_{2,1}\to0$.

 For the term $II_{2,2},$ by the boundedness estimates in \eqref{bounded-xy2}--\eqref{bounded-xy3},  and the local Lipschitz continuity of the Euclidean gradient $D_{\rho}\mathcal H_0,$  we have 
\begin{align*}
II_{2,2}\leq C_Ra(\| \bar x^{\vartheta,r}\|+\|\bar y^{\vartheta,r}\|).
\end{align*}
Further taking into account \eqref{zlimit}, \eqref{zlimit2} yields 
$II_{2,2}\leq C_Ra(\| \bar x\|+\|\bar y\|)
$ by letting $r,\vartheta\to0.$ Then using  \eqref{max-relation2}  gives $ II_{2,2}\to0\text{ as }r,\vartheta,a\to0.
$ 

The estimate of the term $II_{2,3}$ is similar to  that of $II_{2,2},$  noticing the smallness property $\|p_0\|+\|p_1\|\leq r$ due to 
 \eqref{bounded-pq}. Hence we deduce $II_{2,3}\to0$ as $r\to0$. 
 
 Gathering eatimates of $II_{2,1},II_{2,2}$ and $II_{2,3}$ yields  \begin{align}\label{conclu-2}
II_{2}\to0\text{ as }r,\vartheta,\gamma,a\to0.
\end{align}

\textbf{Estimate of the term $II_3.$} By the fact that $\mathrm{tr}(\sigma\sigma^{\top}A)=\mathrm{tr}(\sigma^{\top}A\sigma)$ for a matrix $A=(A_{ij})_{n\times n},$ we have 
\begin{align*}
II_{3}&=\frac{1}{2}(\mathrm{tr}[\sigma\sigma^{\top}D^2_x(U\phi_{R,0})^{\vartheta}(\bar\rho^{\vartheta,r},\bar x^{\vartheta,r})]-\mathrm{tr}[\sigma\sigma^{\top}D^2_y({\widetilde U}\phi_{R,0})_{\vartheta}(\bar\mu^{\vartheta,r},\bar y^{\vartheta,r})])\\&=\frac{1}{2}\mathrm{tr}(( \sigma^{\top}, \sigma^{\top})\mathrm{diag}(D^2_x(U\phi_{R,0})^{\vartheta},-D^2_y({\widetilde U}\phi_{R,0})_{\vartheta})(\sigma^{\top},\sigma^{\top})^{\top})\\
&\leq \frac{1}{2}\mathrm{tr}(( \sigma^{\top}, \sigma^{\top}) Y (\sigma^{\top},\sigma^{\top})^{\top}),
\end{align*}
where in the last step we use \eqref{equality-xxyy}. Recalling the definition of components of $Y$ given by \eqref{compoY1},    one can calculate that 
\begin{align}\label{conclu-3}
II_{3}\leq \frac12\mathrm{tr}((\sigma^{\top}, \sigma^{\top}) Y(\sigma^{\top}, \sigma^{\top}) ^{\top}\leq C\|\sigma\|^2a\to0\text{ as }a\to0.
\end{align}

\textbf{Estimate of the term $II_4.$} 
 Recall that we have proved that $\mathbb F^R$ admits the local Lipschitz property (see \eqref{pq-regularity} and \eqref{xy-regularity}). Combining this with   
 the boundness properties \eqref{max-xy}, \eqref{bounded-xy2}, \eqref{bounded-xy3}, implies that 
 {\small
\begin{align*}
II_4&=-\mathbb F ^R(\bar t^{\vartheta,r}_0,\bar\rho^{\vartheta,r}_0, \bar x^{\vartheta,r}_0 ,D_x(U\phi_{R,0})^{\vartheta}- UD_x\phi_{R,0})+\mathbb F^R(\bar s^{\vartheta,r}_0,\bar\mu^{\vartheta,r}_0, \bar y^{\vartheta,r}_0 ,D_x(U\phi_{R,0})^{\vartheta}- UD_x\phi_{R,0}) \\&-\mathbb F^R (\bar s^{\vartheta,r}_0,\bar\mu^{\vartheta,r}_0, \bar y^{\vartheta,r}_0  ,D_x(U\phi_{R,0})^{\vartheta}- UD_x\phi_{R,0})+ \mathbb F^R (\bar s^{\vartheta,r}_0,\bar\mu^{\vartheta,r}_0, \bar y^{\vartheta,r}_0 ,D_y(\widetilde U\phi_{R,0})_{\vartheta}- \widetilde UD_y\phi_{R,0})\\
&\leq C_{R,C_0,C_1} (|\bar t^{\vartheta,r}_0-\bar s^{\vartheta,r}_0|+\|\bar\rho^{\vartheta,r}_0-\bar\mu^{\vartheta,r}_0\|+\|\bar x^{\vartheta,r}_0-\bar y^{\vartheta,r}_0\|)\\ 
&+C_{R,C_0,C_1}(\| D_x(U\phi_{R,0})^{\vartheta}(\bar\rho^{\vartheta,r},\bar x^{\vartheta,r})-D_y({\widetilde U}\phi_{R,0})_{\vartheta}(\bar\mu^{\vartheta,r},\bar y^{\vartheta,r})\|)\\&+C_{R,C_0,C_1}(\|UD_x\phi_{R,0}(\bar\rho^{\vartheta,r}_0,\bar x^{\vartheta,r}_0)-\widetilde UD_y\phi_{R,0}(\bar\mu^{\vartheta,r}_0,\bar y^{\vartheta,r}_0)\|)=:II_{4,1}+II_{4,2}+II_{4,3}.
\end{align*}}  
For the term $II_{4,1},$ by taking  $r,\vartheta\to0$, it holds $$II_{4,1}\leq C_{R,C_0,C_1}(|\bar t-\bar s|+\|\bar\rho-\bar\mu\|+\|\bar x-\bar y\|)$$  according to \eqref{zlimit}, \eqref{zlimit2}, and  \eqref{relation-twopoint}. Then by further letting 
$\delta,\gamma\to0$, we derive from \eqref{max-relation1} and \eqref{max-relation2} that  $II_{4,1}\to0.$

For the term $II_{4,2},$ from the relation for the first-order derivatives given in  \eqref{equality-x} and \eqref{equality-y}, we derive that 
\begin{align*}
&\|D_x(U\phi_{R,0})^{\vartheta}(\bar\rho^{\vartheta,r},\bar x^{\vartheta,r})-D_y({\widetilde U}\phi_{R,0})_{\vartheta}(\bar\mu^{\vartheta,r},\bar y^{\vartheta,r})\|
\leq C(a\|\bar x^{\vartheta,r}\|+a\|\bar y^{\vartheta,r}\|+\|p_0\|+\|p_1\|).
\end{align*}
By virtue of \eqref{zlimit}, \eqref{zlimit2}, taking $r,\vartheta\to0,$ we get 
\begin{align*}
Ca(\|\bar x^{\vartheta,r}\|+\|\bar y^{\vartheta,r}\|)\leq Ca(\|\bar x\|+\|\bar y\|);
\end{align*}
Moreover, $a(\|\bar x\|+\|\bar y\|)\to0$ as $a\to0$ according to \eqref{max-relation2}. 
And using \eqref{bounded-pq} yields $\|p_0\|+\|p_1\|\to0$ as $r\to0.$ As a result,  
as $r,\vartheta,a\to0,$ we have that $II_{4,2}\to0.$

For  the term $II_{4,3},$ we have the estimate \begin{align*}
&\quad \|UD_x\phi_{R,0}(\bar\rho^{\vartheta,r}_0,\bar x^{\vartheta,r}_0)-\widetilde UD_y\phi_{R,0}(\bar\mu^{\vartheta,r}_0,\bar y^{\vartheta,r}_0)\|\\
&\leq \|(U(\bar\rho^{\vartheta,r}_0,\bar x^{\vartheta,r}_0)-\widetilde U(\bar\mu^{\vartheta,r}_0,\bar y^{\vartheta,r}_0))D_x\phi_{R,0}(\bar\rho^{\vartheta,r}_0,\bar x^{\vartheta,r}_0)\|\\&+\widetilde U(\bar\mu^{\vartheta,r}_0,\bar y^{\vartheta,r}_0)\|D_x\phi_{R,0}(\bar\rho^{\vartheta,r}_0,\bar x^{\vartheta,r}_0)-D_y\phi_{R,0}(\bar\mu^{\vartheta,r}_0,\bar y^{\vartheta,r}_0)\|=:II_{4,3,1}+II_{4,3,2}.
\end{align*}
For the term $II_{4,3,2},$ by virtue of the boundedness of $\widetilde U,$ the smoothness of $\phi_{R,0}$, we have 
\begin{align*}
II_{4,3,2}\leq C_R(\|\bar\rho^{\vartheta,r}_0-\bar\mu^{\vartheta,r}_0\|+\|\bar x^{\vartheta,r}_0-\bar y^{\vartheta,r}_0\|)\leq C_R(\|\bar\rho-\bar\mu\|+\|\bar x-\bar y\|)\text{ as }r,\vartheta\to0,
\end{align*}
where in the last step we also use \eqref{zlimit}, \eqref{zlimit2} and \eqref{relation-twopoint}. 
Further using \eqref{max-relation1} leads to $II_{4,3,2}\to0$ as  $\delta,\gamma\to0$. 

For the term $II_{4,3,1},$ letting $r,\vartheta\to0,$ it can be estimated as \begin{align*}
|U-\widetilde U|(\bar \mu,\bar y) \|D_y\phi_{R,0}(\bar\mu,\bar y)\|\leq \frac{C}{R}\|D_y\mathcal H_0(\bar\mu,\bar y)\|\mathbf 1_{\{\mathcal H_0(\bar\mu,\bar y)\in (R,2R)\}}\to0\text{ as }R\to\infty,
\end{align*} 
where we use Assumption \ref{ass_domi} and property of $\phi_R'$ \eqref{phi-property2}. 

Combining estimates of $II_{4,1},II_{4,2}$ and $II_{4,3},$ we have  
\begin{align}\label{conclu-4}II_4\to0\text{ as }r,\vartheta,\delta,\gamma,a\to0,\text{ and }R\to\infty.
\end{align}

\textbf{Estimate of the term $II_5.$} It follows from   \eqref{equality-x} and \eqref{equality-y}  that 
\begin{align*}
&II_5\leq \Big\|\frac{D_x\phi_{R,0}}{\phi_{R,0}}(\bar\rho^{\vartheta,r}_0,\bar x^{\vartheta,r}_0)\Big\| \| D_x(U\phi_{R,0})^{\vartheta}(\bar\rho^{\vartheta,r},\bar x^{\vartheta,r})-D_y(\widetilde U\phi_{R,0})_{\vartheta}(\bar\mu^{\vartheta,r},\bar y^{\vartheta,r})\|\\&+\|D_y(\widetilde U\phi_{R,0})_{\vartheta}(\bar\mu^{\vartheta,r},\bar y^{\vartheta,r})\|\Big\| \frac{D_x\phi_{R,0}}{\phi_{R,0}}( \bar\rho^{\vartheta,r}_0,\bar x^{\vartheta,r}_0)- \frac{D_y\phi_{R,0}}{\phi_{R,0}}( \bar\mu^{\vartheta,r}_0,\bar y^{\vartheta,r}_0)\Big\|\\&
\leq C_R\Big[\|D_x(U\phi_{R,0})^{\vartheta}(\bar\rho^{\vartheta,r},\bar x^{\vartheta,r})-D_y(\widetilde U\phi_{R,0})_{\vartheta}(\bar\mu^{\vartheta,r},\bar y^{\vartheta,r})\|\\&+ \|D_y(\widetilde U\phi_{R,0})_{\vartheta}(\bar\mu^{\vartheta,r},\bar y^{\vartheta,r})\|(\|\bar\rho^{\vartheta,r}_0-\bar \mu^{\vartheta,r}_0\|+\|\bar x^{\vartheta,r}_0-\bar y^{\vartheta,r}_0\|) \Big]\\
&\leq C_R\Big[a(\|\bar x^{\vartheta,r}\|+\|\bar y^{\vartheta,r}\|+p_0+p_1)\\&
+(\frac{\|\bar x^{\vartheta,r}-\bar y^{\vartheta,r}\|}{\gamma}+2a\|\bar y^{\vartheta,r}\|+p_1)(\|\bar\rho^{\vartheta,r}_0-\bar \mu^{\vartheta,r}_0\|+\|\bar x^{\vartheta,r}_0-\bar y^{\vartheta,r}_0\|) \Big],
\end{align*}
where in the second inequality we use the positivity and the smoothness of $\phi_{R,0}$ (see \eqref{phi-property}, \eqref{phi-property2}), as well as the boundedness of points $(\bar\rho^{\vartheta,r}_0,\bar x^{\vartheta,r}_0), (\bar\mu^{\vartheta,r}_0,\bar y^{\vartheta,r}_0)$ (see  \eqref{bounded-xy2} and \eqref{bounded-xy3}).  
Taking $r,\vartheta\to0$ yields 
\begin{align*}
II_5&\to C_R \Big[a(\|\bar x\|+\|\bar y\|)+\frac{\|\bar x-\bar y\|^2+\|\bar \rho-\bar\mu\|^2}{\gamma}\Big].\end{align*} 
 Then according to \eqref{max-relation1} and \eqref{max-relation2}, further  letting $\gamma,a\to0$ derives 
 \begin{align}\label{conclu-5}
 II_5\to0\text{ as }r,\vartheta,\gamma,a\to0.
 \end{align}

\textbf{Estimate of the term $II_6.$} 
Recalling properties of  $\phi_R,\phi_R',\phi''_R$ in \eqref{phi-property} and  \eqref{phi-property2}, and letting $r,\vartheta\to0$, we have the following upper bound estimate \begin{align*}
II_{6}&
\leq C\mathbf 1_{\{\mathcal H_0(\bar\rho,\bar x)\in (R,2R)\}}\Big[\frac{1}{R} \|D^2_x\mathcal H_0(\bar\rho,\bar x)\|+\frac{1}{R^2}\|D_x\mathcal H_0(\bar\rho,\bar x)\|^2+\frac{R^{\beta}}{R^2}\|D_x\mathcal H_0(\bar\rho,\bar x)\|^2\Big]\\
&+C\mathbf 1_{\{\mathcal H_0(\bar\mu,\bar y)\in (R,2R)\}}\Big[\frac{1}{R} \|D^2_y\mathcal H_0(\bar\mu,\bar y)\|+\frac{1}{R^2}\|D_y\mathcal H_0(\bar\mu,\bar y)\|^2+\frac{R^{\beta}}{R^2}\|D_y\mathcal H_0(\bar\mu,\bar y)\|^2\Big],
\end{align*}
where we use the boundedness of viscosity solutions $U,\widetilde U.$ 
 By Assumption \ref{ass_domi}, taking $R\to\infty$ yields 
 that \begin{align}\label{conclu-6}II_{6}\to0\text{ as }r,\vartheta\to0,\text{ and }R\to\infty.
\end{align}

Plugging \eqref{left-estimate} into \eqref{estimate-main}, we derive 
\begin{align*}
\frac{b_0\sigma_0}{T}-q_0-q_1\leq \frac{b_0\sigma_0}{T}-q_0-q_1+d(\frac{1}{(\bar t^{\vartheta,r})^2}+\frac{1}{(\bar s^{\vartheta,r})^2})= \sum_{i=1}^6II_i. 
\end{align*}
Taking limits $r,\vartheta,\delta,\gamma,a\to0$ and $R\to\infty,$ combining the property $|q_0|+|q_1|\to0$ as $r\to0$ (see \eqref{bounded-pq}), and estimates of terms $II_{1}-II_6$ (see \eqref{conclu-1}--\eqref{conclu-6}), we have  $0<\frac{b_0\sigma_0}{T}\leq 0,$  which leads to a contradiction.  
\end{proof}

\section{appendix}\label{app1} 
In this section, we present the proofs of Propositions \ref{exact} and \ref{dynamic-pro}, and  Lemma \ref{coro2_local}, along with several useful results from convex analysis. 
\subsection{Proof of Proposition \ref{exact}}
\begin{proof}
We first use the truncation and stopping time skill to prove the well-posedness of the solution of the SWHS \eqref{WHflow1}--\eqref{WHflow2}. Then we show the moment estimate  \eqref{H_0estimate} and regularity estimate \eqref{moment-cont-1} for $\mathcal H^{\mathbb V}_0(\rho(t),S(t)),t\in[t_0,T]$.  

Recall the definition of $\mathcal H_1$ in \eqref{def-H1}. 
For $M\in\mathbb N_+,$ define the truncated system
\begin{align}\label{truncated-sde}
\begin{cases}
\mathrm d\rho^M(t)=D_{ S}\mathcal H^{\mathbb V}_0(\rho^M(t),S^M(t))\psi_{M,0}( \rho^M(t),S^M(t))\mathrm dt,\\ 
\mathrm dS^M(t)=-D_{ \rho}\mathcal H^{\mathbb V}_0(\rho^M(t),S^M(t))\psi_{M,0}( \rho^M(t),S^M(t))\mathrm dt-\sigma \mathrm dW(t), 
\end{cases}
\end{align}  
where $t\in(t_0,T],\,\rho^M(t_0)=\rho,S^M(t_0)=x,$ and the truncation function 
$
\psi_{M,0}(\rho,x):=\psi_{M}(\mathcal H_0(\rho,x))
$ with $\psi_M$ being a nonincreasing real function that satisfies 
\begin{align*}
\psi_M(r)=1\quad \text{if }r\leq M,\quad \psi_M(r)=0\quad \text{if }r\ge 2M.
\end{align*} 
According to Assumption \ref{energy} and Definition \ref{def-functions},  by the locally Lipschitz continuity of $\mathcal H_0$ and partial derivatives, for each fixed $M\in\mathbb N_+,$ the coefficients of the truncated system \eqref{truncated-sde} are globally Lipschitz continuous. Thus for each fixed $M\in\mathbb N_+,$ by the standard arguments (see e.g. \cite[Chapter 5]{Fried}),  \eqref{truncated-sde} admits a unique solution $(\rho^M(t),S^M(t))_{t\in[t_0,T]}$.   

Next, we define a sequence of nondecreasing stopping times $$\tau_M:=\inf\{t>0:\mathcal H_0(\rho(t),S(t))\ge M\},\quad n\in\mathbb N_+.$$ And set $\tau_{\infty}:=\lim_{M\to\infty}\tau_M.$ Note that the coefficients of the truncated system \eqref{truncated-sde} are consistent with those in \eqref{WHflow1}--\eqref{WHflow2} when $\mathcal H_0(\rho,x)\leq M,$ due to that $\psi_{M,0}(\rho,x)=1$ when $\mathcal H_0(\rho,x)\leq M.$  This means that for $t\leq \tau_M,$ the solution $(\rho^M(t),S^M(t))$ satisfies the original system \eqref{WHflow1}--\eqref{WHflow2}. As a result, we can define the local solution $(\rho ,S)$ up to the stopping time $\tau_{\infty}$ by $(\rho ,S)=(\rho^M,S^M)$ on $\{t\leq\tau_M\}.$

To get the global existence of the solution, it suffices to show that $\lim_{M\to\infty}\tau_M\wedge T=T$ a.s. This can be guaranteed by  
\begin{align}\label{estimate-proof}
\sup_{M\in\mathbb N_+}
\mathbb E\Big[ \sup_{
s\in [0,\tau_M\wedge T]}
\mathcal H^{\mathbb V}_0(\rho (s),S(s))\Big] <\infty.
\end{align}
In fact, by \eqref{estimate-proof} we have 
\begin{align*}
\mathbb E\Big[\mathcal H^{\mathbb V}_0(\rho^M(\tau_M),S^M(\tau_M))\mathbf 1_{\{\tau_M<T\}}\Big]\leq C<\infty.
\end{align*} 
It follows from $\lim_{M\to\infty}\mathcal H^{\mathbb V}_0(\rho^M(\tau_M),S^M(\tau_M))=+\infty$ that  $\lim_{M\to\infty}\mathbb P(\tau_M<T)=0,$ which implies $\lim_{M\to\infty}\tau_M\wedge T=T$ a.s. We refer to e.g. \cite[Theorem 10.2]{HYZ} for further details. 

It remains to prove \eqref{estimate-proof}. In fact, we will establish a stronger result, namely \eqref{H_0estimate}.
Applying the It\^o formula to $\mathcal H^{\mathbb V}_0(\rho(t),S(t)),$ and noting that $\mathcal H_1$ given by \eqref{def-H1} is independent of the $x$ variable, we obtain 
\begin{align}\label{ito-2}
&\quad \mathcal H^{\mathbb V}_0(\rho(t),S(t))\notag\\
&=\mathcal H^{\mathbb V}_0(\rho,x)-\int_{t_0}^tD_{ S}\mathcal H^{\mathbb V}_0(\rho(s),S(s))^{\top}\sigma\mathrm dW(s)+\int_{t_0}^t\frac12\mathrm{tr}(\sigma\sigma^{\top}D^2_S\mathcal H^{\mathbb V}_0(\rho(s),S(s)))\mathrm ds.
\end{align}  
By the 
maximal-type Burkholder inequality \cite[Proposition  3.15]{BDG} and the condition \eqref{bound_prop}, we derive that for any $p\ge 2,$
\begin{align*}
&\quad \mathbb E\Big[\sup_{t\in[t_0,T]}|\mathcal H^{\mathbb V}_0(\rho(t),S(t))|^p\Big]\\
&\leq C(p,T)\mathbb E\Big[1+|\mathcal H^{\mathbb V}_0(\rho,x)|^p+\Big(\int_{t_0}^T\|D_{ S}\mathcal H^{\mathbb V}_0(\rho(s),S(s))\|^2\mathrm ds\Big)^{\frac p2}+\int_{t_0}^T|\mathcal H^{\mathbb V}_0(\rho(s),S(s))|^p\mathrm ds\Big]\\
&\leq C(p,T)\mathbb E\Big[1+|\mathcal H^{\mathbb V}_0(\rho,x)|^p+\int_{t_0}^T\sup_{r\in[0,s]}|\mathcal H^{\mathbb V}_0(\rho(r),S(r))|^p\mathrm ds\Big],
\end{align*}
where in the last step we use the H\"older inequality. By means of the Gr\"onwall inequality we finish the proof of \eqref{H_0estimate}. 

 Moreover, utilizing the maximal-type Burkholder inequality and the condition \eqref{bound_prop}  again, we obtain from \eqref{ito-2} that for any $p\ge2,$
\begin{align*}
&\quad \mathbb E\Big[\sup_{r\in[t_0,t]}|\mathcal H^{\mathbb V}_0(\rho(r),S(r))-\mathcal H^{\mathbb V}_0(\rho,x)|^p\Big]\\
&\leq C\mathbb E\Big[1+\Big(\int_{t_0}^t\sup_{s\in[t_0,t]}|\mathcal H^{\mathbb V}_0(\rho(s),S(s))|^2\mathrm ds\Big)^{\frac p2}+\Big(\int_{t_0}^t \sup_{s\in[t_0,t]}|\mathcal H^{\mathbb V}_0(\rho(s),S(s))|\mathrm ds\Big)^p\Big].
\end{align*}
By means of \eqref{H_0estimate} we derive that 
\begin{align*}
 \mathbb E\Big[\sup_{r\in[t_0,t]}|\mathcal H^{\mathbb V}_0(\rho(r),S(r))-\mathcal H^{\mathbb V}_0(\rho,x)|^p\Big]\leq C(t-t_0)^{\frac p2},
\end{align*}
which completes the proof of \eqref{moment-cont-1}. 
\end{proof}

\subsection{Proof of Proposition \ref{dynamic-pro}} 
\begin{proof}
The proof is split into two steps. 

\textit{Step 1:} We first prove the upper bound: for every $\mathbb V(\cdot)\in\mathscr V_{\ell}[t,T]$ and $0<t<\bar t\leq T,$
\begin{align}\label{leq-inequalityU}
U(t,\rho,x) \leq \inf_{\mathbb V\in\mathscr V_{\ell}[t,T]}\mathbb E_{t} \Big[ \int_t^{\bar t} F(s,\rho(s),S(s),\mathbb V(s))  ds + U(\bar t, \rho(\bar t),S(\bar t)) \Big].
\end{align}
By the definitions of $\mathcal J$ and $U$ (see \eqref{def-costfunction} and \eqref{def-valuefunction}), for any $\mathbb V(\cdot)\in\mathscr V_{\ell}[t,T]$,
\begin{align}
&\quad U(t,\rho,x) \leq \mathcal J(t,\rho,x;\mathbb V) \notag \\
&= \mathbb E_{t} \Big[ \int_t^{\bar t} F(s,\rho(s),S(s),\mathbb V(s))  ds + \int_{\bar t}^T F(s,\rho(s),S(s),\mathbb V(s))  ds + h(\rho(T),S(T)) \Big] \notag \\
&=\mathbb E_{t} \Big[ \int_t^{\bar t} F(s,\rho(s),S(s),\mathbb V(s))  ds + \mathbb E_{\bar t}\Big[\int_{\bar t}^T F(s,\rho(s),S(s),\mathbb V(s))  ds + h(\rho(T),S(T)) \Big]\Big] 
\notag\\
&= \mathbb E_{t} \Big[ \int_t^{\bar t} F(s,\rho(s),S(s),\mathbb V(s))  ds + \mathcal J(\bar t,\rho(\bar t),S(\bar t);\mathbb V|_{[\bar t,T]}) \Big],\label{Ueq11}
\end{align}
where $\mathbb V|_{[\bar t,T]}$ denotes the restriction of $\mathbb V$ to $[\bar t,T]$.  
Following the proof of \cite[Lemma 4.4]{YongJM}, there exists a minimizing sequence $\{\mathbb V_k\} \subset \mathscr V_{\ell}[\bar t,T]$ such that 
\[
\mathcal J(\bar t,\rho(\bar t),S(\bar t);\mathbb V_k) \to U(\bar t,\rho( \bar t),S(\bar t)) \quad \text{a.s. as } k \to \infty.
\]
By the reverse Fatou lemma, 
\begin{align*}
\limsup_{k\to\infty} \mathbb E_t [ \mathcal J(\bar t,\rho(\bar t),S(\bar t);\mathbb V_k) ] &\leq \mathbb E_{t} [ \limsup_{k\to\infty} \mathcal J(\bar t,\rho(\bar t),S(\bar t);\mathbb V_k) ] \\
&= \mathbb E_t [ U(\bar t,\rho(\bar t),S(\bar t)) ],
\end{align*}
where we use the condition that $\mathbb V_k$ is uniformly bounded by $\ell$ and  functionals $F,h$ are uniformly bounded with respect to $\mathbb V$ (see Assumption \ref{ass1}). 
Taking the infimum over $\mathbb V|_{[\bar t,T]} \in \mathscr V_{\ell}[\bar t,T]$ in \eqref{Ueq11} and use the relation $\inf_{k\in A_1}f_k\leq \inf_{k\in A_2}f_k$ for two index sets $A_2\subset A_1$ yields
\begin{align*}
U(t,\rho,x) &\leq \mathbb E_{t} \Big[ \int_t^{\bar t} F(s,\rho(s),S(s),\mathbb V(s))  ds\Big]+\inf_{k\in\mathbb N}\mathbb E_t[\mathcal J(\bar t,\rho(\bar t),S(\bar t);\mathbb V_k)]\\
&\leq \mathbb E_{t} \Big[ \int_t^{\bar t} F(s,\rho(s),S(s),\mathbb V(s))  ds\Big]+\limsup_{k\to\infty}\mathbb E_t[\mathcal J(\bar t,\rho(\bar t),S(\bar t);\mathbb V_k)]\\
&\leq 
\mathbb E_{t} \Big[ \int_t^{\bar t} F(s,\rho(s),S(s),\mathbb V(s))  ds + U(\bar t,\rho(\bar t),S(\bar t)) \Big].
\end{align*}
This gives \eqref{leq-inequalityU} by taking the infimum over all admissible controls. 

\textit{Step 2:} Next, we prove the lower bound: for every $\epsilon>0,$ there exists $\mathbb V^{\epsilon} \in \mathscr V_{\ell}[t,T]$ such that  
\begin{align}\label{ge-inequalityU}
U(t,\rho,x) \geq \inf_{\mathbb V\in\mathscr V_{\ell}[t,T]}\mathbb E_{t} \Big[ \int_t^{\bar t} F(s,\rho(s),S(s),\mathbb V(s))  ds + U(\bar t,\rho(\bar t),S(\bar t)) \Big]. 
\end{align} 
By the definitions of $\mathcal J$ and $U$, for every $\epsilon>0$, there exists $\mathbb V^{\epsilon} \in \mathscr V_{\ell}[t,T]$ such that 
\begin{align*}
&U(t,\rho,x) + \epsilon > \mathcal J(t,\rho,x;\mathbb V^{\epsilon}) \\
&= \mathbb E_{t} \Big[ \int_t^{\bar t} F(s,\rho^{\mathbb V^{\epsilon}}(s),S^{\mathbb V^{\epsilon}}(s),\mathbb V^{\epsilon}(s))  ds \\
&+\int_{\bar t}^T F(s,\rho^{\mathbb V^{\epsilon}}(s),S^{\mathbb V^{\epsilon}}(s),\mathbb V^{\epsilon}(s))  ds + h(\rho^{\mathbb V^{\epsilon}}(T),S^{\mathbb V^{\epsilon}}(T)) \Big] \\
&= \mathbb E_{t} \Big[ \int_t^{\bar t} F(s,\rho^{\mathbb V^{\epsilon}}(s),S^{\mathbb V^{\epsilon}}(s),\mathbb V^{\epsilon}(s))  ds + \mathcal J(\bar t,\rho^{\mathbb V^{\epsilon}}(\bar t),S^{\mathbb V^{\epsilon}}(\bar t);\mathbb V^{\epsilon}|_{[\bar t,T]}) \Big] \\
&\geq \mathbb E_{t} \Big[ \int_t^{\bar t} F(s,\rho^{\mathbb V^{\epsilon}}(s),S^{\mathbb V^{\epsilon}}(s),\mathbb V^{\epsilon}(s))  ds + U(\bar t, \rho^{\mathbb V^{\epsilon}}(\bar t),S^{\mathbb V^{\epsilon}}(\bar t)) \Big]\\
&\geq \inf_{\mathbb V\in\mathscr V_{\ell}[t,T]} \mathbb E_t \Big[ \int_t^{\bar t} F(s,\rho^{\mathbb V}(s),S^{\mathbb V}(s),\mathbb V(s))  ds + U(\bar t, \rho^{\mathbb V}(\bar t),S^{\mathbb V}(\bar t)) \Big],
\end{align*}
where $(\rho^{\mathbb V^{\epsilon}},S^{\mathbb V^{\epsilon}})$ solves the SWHS with the control $\mathbb V^{\epsilon}$. Then putting $\epsilon\to0$ finishes the proof of \eqref{ge-inequalityU}.  
\end{proof}

\subsection{Proof of Lemma \ref{coro2_local}}  
\begin{proof} We only  show the case that $\widehat U$ is a viscosity subsolution of \eqref{HJB_R} based on Definition \ref{def1} (i). The case of viscosity supersolution can be proved similarly and thus is omitted.

According to Definition \ref{def1} (i),  suppose that $\hat\varphi\in\mathcal C^{1,1,2}(\mathcal A_R)$ and $\widehat U-\hat\varphi$ has a maximum at $(t_0,\rho_0,x_0)\in\mathcal A_R$ satisfying $(\widehat U-\hat\varphi)(t_0,\rho_0,x_0)=0.$ We then aim to verify the viscosity subsolution inequality. Set $\varphi=\hat\varphi/\phi_{R,0}.$  Then $\varphi\in\mathcal C^{1,1,2}(\mathcal A_R)$ and  $$0=[(U-\varphi)\phi_{R,0}](t_0,\rho_0,x_0)\ge [(U-\varphi)\phi_{R,0}](t,\rho,x).$$ Due to $\phi_{R,0}>0$ we arrive at $0=(U-\varphi)(t_0,\rho_0,x_0)\ge (U-\varphi)(t,\rho,x),$ which means that $U-\varphi$ has a maximum at $(t_0,\rho_0,x_0)\in \mathcal A_R.$  Since $U$ is a viscosity subsolution of \eqref{HJB}, we have 
\begin{align*}
\partial_t\varphi(t_0,\rho_0,x_0)+H(t_0, \rho_0, x_0, \partial_{\rho} \varphi(t_0, \rho_0, x_0), D_x \varphi(t_0, \rho_0, x_0), D^2_x \varphi(t_0, \rho_0, x_0)) \ge 0.
\end{align*} 
Noticing $\hat\varphi=\varphi\phi_{R,0},$ and using the similar argument as that of \eqref{phi-multiply} and \eqref{HRU-def}, where $U$ is replaced by the test function $\varphi,$ give that 
\begin{align*}
\partial_t\hat\varphi(t_0, \rho_0, x_0)+H^R_{\varphi}(t_0, \rho_0, x_0, \partial_{\rho} \hat\varphi(t_0, \rho_0, x_0), D_x \hat\varphi(t_0, \rho_0, x_0), D^2_x \hat\varphi(t_0, \rho_0, x_0))\ge 0.
\end{align*} 
Since $\widehat U(t_0,\rho_0,x_0)=\hat\varphi(t_0,\rho_0,x_0),$ where $\widehat U=U\phi_{R,0}$ and $\hat \varphi=\varphi\phi_{R,0},$ we have $ U(t_0,\rho_0,x_0)=\varphi(t_0,\rho_0,x_0)$  and 
\begin{align*}
&\quad H^R_{\varphi}(t_0, \rho_0, x_0, \partial_{\rho} \hat \varphi(t_0, \rho_0, x_0), D_x \hat\varphi(t_0, \rho_0, x_0), D^2_x \hat\varphi(t_0, \rho_0, x_0))\\
&=H^R_U(t_0, \rho_0, x_0, \partial_{\rho} \hat\varphi(t_0, \rho_0, x_0), D_x \hat\varphi(t_0, \rho_0, x_0), D^2_x \hat\varphi(t_0, \rho_0, x_0)).
\end{align*}
Hence $\widehat U$ is a viscosity subsolution of \eqref{HJB_R} in $\mathcal A_R.$ 
\end{proof} 

\subsection{Some useful lemmas}
The following lemma shows that semiconvex and semiconcave functions are twice differentiable almost everywhere. 

\begin{lemma}
 \label{Alex}[Alexandrov's  Theorem \cite{Alex39}]
 Let $\mathbb Q\subset\mathbb R^d$ be a convex domain and $\psi:\mathbb Q\to\mathbb R$ be a semiconvex (or semiconcave) function, where $d$ is a positive integer. Then there exists a set  $\mathbb Q_1\subset \mathbb Q$ with zero Lebesgue measure such that at any $x\in \mathbb Q\backslash \mathbb Q_1$, $\psi$ is twice differentiable, i.e., there
are $p\in\mathbb R^d$ and a symmetric matrix $P$ such that 
\begin{align*}
\psi(x+y)=\psi(x)+p^{\top}y+\frac12y^{\top}Py+o(\|y\|^2),
\end{align*}
for all $\|y\|$ small enough. 
\end{lemma}

Next, we present  Jensen's lemma, 
which is fundamental in the viscosity solution theory, as it provides a tool to extract information about second-order derivatives for the functions at points of local maximum.

\begin{lemma}\label{Jensen}[Jensen's lemma \cite{Jensen88}] Let $\mathbb Q\subset \mathbb R^d$ be a convex domain, $\Psi:\mathbb Q\to\mathbb R$ be semiconvex, and  let  $z_0\in \mathbb Q^{\circ}$ be a local strict maximizer of $\Psi$, where $\mathbb Q^{\circ}$ denotes the interior of $\mathbb Q.$ Then  for
any small $r,\delta>0$, the set
\begin{align*}
\mathcal K:=\{z\in B_r(z_0):\exists \|p\|\leq \delta\;\text{such that } \Psi(\cdot)-p^{\top}(\cdot-z_0)\text{ attains a  local maximum at }z\}
\end{align*} 
has a positive Lebesgue measure.
\end{lemma}

\bibliographystyle{plain}
\bibliography{references.bib}

\end{document}